\documentclass[a4paper,12pt]{amsart}
\usepackage[foot]{amsaddr}

\usepackage{amsmath}
\usepackage{amsthm}
\usepackage{amssymb}
\usepackage{amsfonts}
\usepackage{enumerate}
\usepackage{mathrsfs}
\usepackage[shortlabels]{enumitem} % more enumerate/itemize options
\usepackage[dvipsnames]{xcolor}
\usepackage[normalem]{ulem}
\usepackage{marginnote}
\usepackage{esint}
\usepackage{hyperref}

\usepackage[numeric]{amsrefs}

\calclayout

\theoremstyle{plain}
\newtheorem{theorem}{Theorem}[section]
\newtheorem{thm}{Theorem}

\newtheorem{lemma}[theorem]{Lemma}
\newtheorem{corollary}[theorem]{Corollary}

\newtheorem*{claim*}{Claim}

\theoremstyle{definition}
\newtheorem*{definition*}{Definition}
\newtheorem{definition}[theorem]{Definition}
\newtheorem{example}[theorem]{Example}

\theoremstyle{remark}
\newtheorem{remark}[theorem]{Remark}

\numberwithin{equation}{section}

\def \one {\mathsf{1}}

\def \R {\mathbb{R}}
\def \A {\mathcal{A}}
\def \E {\mathbb{E}}
\def \H {\mathbb{H}}

\def \N {\mathbb{N}}
\def \P {\mathbb{P}}
\def \U {\mathcal{U}}

\def \F {\mathcal{F}}
\def \G {\mathcal{G}}
\def \Z {\mathbb{Z}}

\def \eps {\varepsilon}

\def \sbs {\subseteq}
\def \sps {\supseteq}

\def \bs {\boldsymbol}
\def \x {\boldsymbol{x}}
\def \y {\boldsymbol{y}}

\def \cembed {\overset{\mathrm{c}}{\hookrightarrow}}

\DeclareMathOperator{\conv}{conv}
\DeclareMathOperator{\cconv}{\overline{conv}}
\DeclareMathOperator{\dist}{dist}
\DeclareMathOperator{\diam}{diam}
\DeclareMathOperator{\Lip}{Lip}
\DeclareMathOperator{\LF}{LF}
\DeclareMathOperator{\Hyp}{Hyp}

\DeclareMathOperator{\supp}{supp}
\DeclareMathOperator*{\esssup}{ess\,sup}

\begin{document}
\title{Hyperbolic Metric Spaces and Stochastic Embeddings}

\author[Chris Gartland]{Chris Gartland}
\address{Department of Mathematics, University of California San Diego, 9500 Gilman Dr. La Jolla, CA 92093, USA.}

\thanks{The author was supported by the National Science Foundation under Grant Number DMS-2342644}
	


\begin{abstract}
Stochastic embeddings of finite metric spaces into graph-theoretic trees have proven to be a vital tool for constructing approximation algorithms in theoretical computer science. In the present work, we build out some of the basic theory of stochastic embeddings in the infinite setting with an aim towards applications to Lipschitz free space theory. We prove that proper metric spaces stochastically embedding into $\R$-trees have Lipschitz free spaces isomorphic to $L^1$-spaces. We then undergo a systematic study of stochastic embeddability of Gromov hyperbolic metric spaces into $\R$-trees by way of stochastic embeddability of their boundaries into ultrametric spaces. The following are obtained as our main results: (1) Every snowflake of a compact, finite Nagata-dimensional metric space stochastically embeds into an ultrametric space and has Lipschitz free space isomorphic to $\ell^1$. (2) The Lipschitz free space over hyperbolic $n$-space is isomorphic to the Lipschitz free space over Euclidean $n$-space. (3) Every infinite, finitely generated hyperbolic group stochastically embeds into an $\R$-tree, has Lipschitz free space isomorphic to $\ell^1$, and admits a proper, uniformly Lipschitz affine action on $\ell^1$.
\end{abstract}

\maketitle
\tableofcontents

\section{Introduction}
\label{s:intro}
Let $X$ be a pointed metric space: a metric space equipped with some fixed basepoint. The Banach space of Lipschitz functions $X \to \R$ vanishing at the basepoint has a canonical predual $\LF(X)$, called the \emph{Lipschitz free space}, or just \emph{free space}, of $X$. One of the main problems in the field is to determine when $\LF(X)$ isomorphically embeds into an $L^1$-space (a Banach space of the form $L^1(Y,\Sigma,\mu)$ for some measure space $(Y,\Sigma,\mu)$), or more strongly when $\LF(X)$ is isomorphic to an $L^1$-space (see \cite{Godard,FG}, \cite[$\S$8.6]{Weaver} for some examples and \cite[$\S$8]{Kalton}, \cite{KN,NaorSchechtman,DKO,BGS} for some nonexamples). The isometric version of this question has an elegant but rigid solution: $\LF(X)$ isometrically embeds into an $L^1$-space if and only if $X$ isometrically embeds into an $\R$-tree (\cite[Theorem~4.2]{Godard}, see also \cite{APP}). On the other hand, the isomorphic category permits much greater geometric flexibility; for \emph{any} infinite, compact, doubling\footnote{A metric space is \emph{doubling} if there exists $N\in\N$ such that for every point $x\in X$ and radius $r>0$, the ball $B_{2r}(x)$ can be covered by $N$ balls $B_r(x_1), B_r(x_2), \dots B_r(x_N)$ of half the radius.} space $(X,d)$ and exponent $\alpha \in (0,1)$, the snowflake space $(X,d^\alpha)$ has free space isomorphic to $\ell^1$ (see \cite[Theorem~8.49]{Weaver}).

The main goal of this article is to build out the theory of $L^1$-embeddings of free spaces, and to apply the theory to identify the isomorphism type of $\LF(X)$ for several choice metric spaces $X$: finite Nagata-dimensional snowflake spaces, hyperbolic $n$-space $\H^n$, and finitely generated Gromov hyperbolic groups. Towards this end, the principal role is played by \emph{stochastic biLipschitz embeddings}: random maps $\phi: X \to Y$ between pointed metric spaces $(X,d_X,p),(X,d_Y,q)$ for which there exists $L<\infty$ such that, for every $x,y \in X$, it holds that
\begin{itemize}
    \item $d_Y(\phi(x),\phi(y)) \geq L^{-1}d_X(x,y)$ almost surely,
    \item $\E[d_Y(\phi(x),\phi(y))] \leq Ld_X(x,y)$, and
    \item $\phi(p) = q$ almost surely.
\end{itemize}
See $\S$\ref{ss:randommaps} for the precise definition of a random map and for further discussion on stochastic embeddings. In the theoretical computer science literature, there is already an established finitary version of stochastic biLipschitz embeddings (first introduced in \cite{Bartal}, but not under this name), and this version is especially well-studied when the target spaces are tree metrics (\cite{Alon,FRT,ST}). In this finitary setting, a stochastic biLipschitz embedding into tree metrics is known, among other things, to induce an isomorphic embedding of the Lipschitz free space into an $L^1$-space with controlled distortion (\cite[Corollary~6.5]{BMSZ}). In this article, we extend the purview of stochastic embeddings to the infinite setting and obtain the analogous result that a stochastic biLipschitz embedding into an $\R$-tree induces an isomorphic embedding of the free space into an $L^1$-space. Subsequently, we work towards constructing such embeddings for certain Gromov hyperbolic spaces. In our proofs, a major role is played by Bonk-Schramm's hyperbolic fillings \cite{BS}, a construction which provides an inverse functor to the Gromov boundary functor. In addition to the intrinsic interest in understanding the free spaces of hyperbolic spaces, we were motivated by an application to the theory of group actions on Banach spaces (as described in \cite[Question~1]{CCD}), and we are able to conclude that every infinite, finitely generated hyperbolic group admits a proper, uniformly Lipschitz affine action on $\ell^1$.

We will now state precisely our main results and describe in more detail their proofs and the overall structure of the article. We also compare the present work to the related article \cite{KL} on stochastic biLipschitz embeddability of hyperbolic spaces into trees.

\subsection{Main Results, Outline of their Proofs, and Structure of Article}
\label{ss:results}
Following this introductory section, $\S$\ref{s:prelims} establishes the background on Lipschitz free spaces, metric spaces of finite Nagata dimension, Gromov hyperbolic spaces and their boundaries, and $\R$-trees and ultrametric spaces that are necessary for the remainder of the article. The results of $\S$\ref{s:prelims} are generally known and not particularly new. However, there are a few important facts -- Theorem~\ref{thm:freespacecomplementedinL1}, Corollary~\ref{cor:Nagatauniformbnddextension}, and Lemma~\ref{lem:Hyp(U)T} -- that do not appear in the literature to the level of specificity we require, and so we include their full proofs here.

Next, we obtain in $\S$\ref{s:stochastic} our base result, Theorem~\ref{thm:A} (via Theorem~\ref{thm:stochasticembedtreeLFisoL1}), that stochastic biLipschitz embeddings into $\R$-trees induce isomorphic embeddings into $L^1$-spaces. Under the additional hypothesis of properness\footnote{A metric space is \emph{proper} if all its closed and bounded subsets are compact.}, we can use existing results in Lipschitz free space theory described in $\S$\ref{ss:LF} to get a surjective isomorphism to an $L^1$-space. Recall that a metric space is \emph{purely 1-unrectifiable} if it admits no biLipschitz embedding from a positive measure subset of $\R$.

\begin{thm}[Theorem~\ref{thm:stochasticembedtreeLFisoL1}] \label{thm:A}
Let $X$ be an infinite, proper, pointed metric space. If $X$ stochastic biLipschitzly embeds into a pointed $\R$-tree, then $\LF(X) \approx L^1$ or $\LF(X) \approx \ell^1$, with the latter isomorphism holding if and only if the completion of $X$ is purely 1-unrectifiable.
\end{thm}

\noindent Here and in what follows, we use the notation ``$V\approx W$" to assert the existence of an isomorphism between two Banach spaces $V,W$.

The subsequent sections $\S$\ref{s:stochasticums}-\ref{s:stochasticfillings} are aimed at producing stochastic biLipschitz embeddings of certain classes of Gromov hyperbolic spaces into $\R$-trees. Our method of proof proceeds by analyzing the corresponding problem on the Gromov boundary. Since ultrametric spaces are the boundaries of $\R$-trees, we are led to seek conditions on which the boundary of a hyperbolic metric space -- which can be any bounded, complete metric space -- will stochastic biLipschitzly embed into an ultrametric space. This is the content of $\S$\ref{s:stochasticums}. It turns out that spaces of finite Nagata dimension provide a natural setting for this problem, and the main result of $\S$\ref{s:stochasticums} (Theorem~\ref{thm:Nagatasnowflakestochastic}) concerns these spaces. When we combine Theorem~\ref{thm:Nagatasnowflakestochastic} with Theorem~\ref{thm:A}, we obtain the next main theorem of the article.

\begin{thm} \label{thm:B}
Let $(X,d)$ be an infinite, compact, finite Nagata-dimensional metric space. Then for every $\alpha \in (0,1)$, the metric space $(X,d^\alpha)$ admits a stochastic biLipschitz embedding (for every choice of basepoint) into a pointed ultrametric space, and $\LF(X,d^\alpha) \approx \ell^1$.
\end{thm}

\begin{proof}
Let $\alpha \in (0,1)$. Equip $(X,d^\alpha)$ with any basepoint. Then by Theorem~\ref{thm:Nagatasnowflakestochastic}, there is a stochastic biLipschitz embedding of $(X,d^\alpha)$ into a pointed ultrametric space. Since ultrametric spaces isometrically embed into $\R$-trees, there is a
stochastic biLipschitz embedding of $(X,d^\alpha)$ into a pointed $\R$-tree. Then since $(X,d^\alpha)$ has purely 1-unrectifiable completion\footnote{That $(X,d^\alpha)$ has purely 1-unrectifiable completion for $\alpha<1$ is well-known and easy to prove using, for example, \cite[Proposition~2.3]{TW} and the Lebesgue density theorem.}, we have by Theorem~\ref{thm:A} that $\LF(X,d^\alpha) \approx \ell^1$. 
\end{proof}

It holds that $\LF(X,d^{\alpha}) \approx \ell^1$ whenever $(X,d)$ is an infinite, compact, doubling metric space and $\alpha \in (0,1)$ (see \cite[Theorem~8.49]{Weaver}). However, for distortion functions $\omega: [0,\infty) \to [0,\infty)$ (see $\S$\ref{ss:pconcave} for the definition) satisfying $\lim_{t\to 0}\frac{t^\alpha}{\omega(t)} = 0$ for all $\alpha \in (0,1)$, for example when $\omega(t) = \log(\frac{1}{t})^{-1}$ for $t$ sufficiently small, the space $(X,\omega \circ d)$ generally fails to be doubling. As far as we are aware, at the time of this writing, the validity of statement $\LF(X,\omega \circ d) \approx \ell^1$ for such $\omega$ remained an open question (see \cite[page 294]{Weaver}). But, since post-composition with a distortion function $\omega$ does not change the Nagata dimension (see Lemma~\ref{lem:Nagatadistortion}), we are able to invoke Theorem~\ref{thm:B} and prove in Theorem~\ref{thm:pconcave} that whenever $(X,d)$ is an infinite, compact, finite Nagata-dimensional metric space and $\omega$ is a distortion function with $\omega(t) = \log(\frac{1}{t})^{-1}$ for $t$ sufficiently small (see Example~\ref{ex:omega_p}), the isomorphism
\begin{equation} \label{eq:distortedLF}
    \LF(X,\omega \circ d) \approx \ell^1
\end{equation}
holds.

In $\S$\ref{s:stochasticfillings}, we return to the problem of constructing stochastic biLipschitz embeddings of hyperbolic spaces $X$ into $\R$-trees. Our objective is to understand specifically when this problem can be answered by stochastically embedding $\partial X$, the boundary of $X$, into an ultrametric space. This type of problem falls squarely within the realm of \emph{hyperbolic fillings}, which are constructions of hyperbolic metric spaces with prescribed boundaries. The hyperbolic filling technique we find most fitting for our purposes is the one of Bonk-Schramm \cite{BS}, as described in $\S$\ref{ss:hyperbolic}. With this apparatus in mind, we are able to identify ``log-stochastic biLipschitz embeddings" as the correct notion to study for boundaries of hyperbolic spaces. A \emph{log-stochastic biLipschitz embedding} between pointed metric spaces is defined in the same way as a stochastic biLipschitz embedding, but with the weaker condition
\begin{equation*}
    \E[\log(d_Y(\phi(x),\phi(y)))] \leq \log(Ld_X(x,y))
\end{equation*}
replacing $\E[d_Y(\phi(x),\phi(y))] \leq Ld_X(x,y)$ (this is indeed weaker by Jensen's inequality). See $\S$\ref{s:stochasticfillings} for a more precise discussion of log-stochastic biLipschitz embeddings. The fundamental theorem on these embeddings is Theorem~\ref{thm:C} (via Theorem~\ref{thm:Hypstochasticums}). Recall that a metric space $(X,d)$ is \emph{uniformly discrete} if there exists $\theta > 0$ such that $d(x,y) \geq \theta$ for all $x\neq y \in X$, and is \emph{locally finite} if all its bounded sets are finite. See $\S$\ref{ss:maps} for the definition of rough biLipschitz embedding and $\S$\ref{ss:hyperbolic} for the definition of visual metric spaces.

\begin{thm}[Theorem~\ref{thm:Hypstochasticums}] \label{thm:C}
Let $A$ be an infinite, uniformly discrete, locally finite, pointed metric space, and let $X$ be a visual, hyperbolic, pointed metric space. If $A$ rough biLipschitzly embeds into $X$ and $\partial X$ log-stochastic biLipschitzly embeds (for every choice of basepoint) into a bounded, pointed ultrametric space, then $A$ stochastic biLipschitzly embeds into a pointed $\R$-tree and $\LF(A) \approx \ell^1$.
\end{thm}

As an application of Theorem~\ref{thm:C}, we show through Corollary~\ref{cor:LF(Hn)} in $\S$\ref{s:LF(Hn)} that the free space over hyperbolic $n$-space $\H^n$ is isomorphic to the free space over Euclidean $n$-space $\R^n$. See $\S$\ref{s:LF(Hn)} for background on the space $\H^n$.

\begin{thm}[Corollary~\ref{cor:LF(Hn)}] \label{thm:D}
$\LF(\H^n) \approx \LF(\R^n)$.
\end{thm}

\noindent The proof of Theorem~\ref{thm:D} can be summarized as follows: we establish in Lemma~\ref{lem:Nagatal1sum} a general local-to-global criterion for finite Nagata-dimensional spaces $X$, which states that $\LF(X)$ is isomorphic to a complemented subspace of the $\ell^1$-sum $\left(\oplus^1_{a\in A}\LF(B_{r}(a))\right) \oplus \LF(A)$, where $0<r<\infty$ is some fixed radius and $A \sbs X$ is a uniformly discrete subset. We can then apply this lemma to $X = \H^n$ and use the fact that any infinite, uniformly discrete $A \sbs \H^n$ satisfies the hypotheses of Theorem~\ref{thm:C} to obtain that $\LF(\H^n)$ is isomorphic to a complemented subspace of $\left(\oplus^1_{a\in A}\LF(B_{r}(a))\right) \oplus \ell^1$. We can then use existing free space theory and the fact that $B_{r}(a)$ is uniformly (in $a \in A$) biLipschitz equivalent to a ball in $\R^n$ to obtain $\LF(\H^n) \approx \LF(\R^n)$.

In the final section, $\S$\ref{s:actions}, we apply Theorem~\ref{thm:C} and \cite[Embedding Theorem~1.1]{BS} and prove in Lemma~\ref{lem:fghyperbolic} that every infinite, finitely generated hyperbolic group $\Gamma$ stochastic biLipschitzly embeds into a pointed $\R$-tree and thus satisfies
\begin{equation} \label{eq:fghyperbolic}
    \LF(\Gamma) \approx \ell^1.
\end{equation}
This answers \cite[Question~1]{CCD}. In particular, we get Theorem~\ref{thm:E} (via Theorem~\ref{thm:l1action}) as an immediate consequence. The unfamiliar reader may consult \cite{Nowak} for the definition of proper, uniformly Lipschitz affine actions.

\begin{thm}[Theorem~\ref{thm:l1action}] \label{thm:E}
For every infinite hyperbolic group $\Gamma$ and finite generating set $S \sbs \Gamma$, there exists an isometric affine action $\alpha$ of $\Gamma$ on a Banach space $V$ isomorphic to $\ell^1$ such that the orbit map $\gamma \mapsto \alpha(\gamma)(0)$ is an isometric embedding $\Gamma \hookrightarrow V$ (with respect to the $S$-word metric). In particular, every hyperbolic group admits a proper, uniformly Lipschitz affine action on $\ell^1$.
\end{thm}

\noindent Theorem~\ref{thm:E} fits into a line of research inspired by a conjecture of Shalom positing that every hyperbolic group admits a proper, uniformly Lipschitz affine action on $\ell^2$ (\cite[Open~Problem~14]{Ober}, \cite[Conjecture~35]{Nowak}). We are very grateful to Michal Doucha for bringing this problem to our attention and explaining the role of Lipschitz free space theory (see the paragraph following \cite[Question~1]{CCD}). We briefly mention in $\S$\ref{s:actions} other recent work concerning this problem.

\subsection{Relationship to Existing Work}
In \cite[Theorem~5.4]{KL}, Krauthgamer and Lee proved that for finite subsets $X$ of certain visual hyperbolic metric spaces satisfying a type of local doubling condition (which holds for $\H^n$), $X$ admits a stochastic biLipschitz embedding into a ``tree of neighborhoods" of $X$. A tree of neighborhoods is constructed by gluing together balls $B_\tau(x_i) \sbs X$ in a tree-like fashion, which amounts to the same process as described in \cite[Lemma~3.12]{Weaver}. By \cite[Lemma~3.12]{Weaver}, the free space of this tree of neighborhoods is isomorphic to the (finite) $\ell^1$-sum $\oplus^1_i \LF(B_\tau(x_i))$. We believe that this argument could be substituted for $\S$\ref{s:stochasticums}-\ref{s:stochasticfillings} in the proofs of Theorem~\ref{thm:D} and \eqref{eq:fghyperbolic}, allowing us to (potentially) obtain those results more quickly. However, we opt for the current approach for at least two reasons. First, like the present article, the proof of \cite[Theorem~5.4]{KL} uses boundary methods in an essential way, but the boundary itself does not explicitly appear in the statement of that theorem. In the present article, it is our desire to flesh out a more complete picture of the precise relationship between $\partial X$ and the stochastic biLipschitz embeddability of $X$ into trees. We believe that this is accomplished (at least partially) in $\S$\ref{s:stochasticfillings} with Theorem~\ref{thm:Hypstochasticums}, and as a consequence this theorem provides a blackbox that could find future use in proving stochastic biLipschitz embeddability of a hyperbolic metric space into a tree. Second, on the way to proving Theorem~\ref{thm:D} and \eqref{eq:fghyperbolic}, we are able to establish \eqref{eq:distortedLF} as a byproduct. As previously discussed, this result answers the open problem described in \cite[page 294]{Weaver}.

\section{Preliminaries}
\label{s:prelims}
\subsection{Metric Spaces and Maps between them}
\label{ss:maps}
Whenever $(X,d)$ is a metric space and $A,B \sbs X$ are nonempty subsets, we write $\diam(A) := \sup_{a,b \in A} d(a,b) \in [0,\infty]$ and $\dist(A,B) := \inf_{a\in A, b \in B}d(a,b) \in [0,\infty)$. If $B = \{b\}$ is a singleton, we condense notation and write $\dist(A,b)$ for $\dist(A,\{b\})$. If $x \in X$  and $r \geq 0$, we write $B_r(x)$ for the closed ball $\{y \in X: d(x,y) \leq r\}$.

Let $(X,d_X),(Y,d_Y)$ be metric spaces and $f: X \to Y$ a map. Then $f$ is \emph{$L$-Lipschitz} (for some $L<\infty$) if for all $x,y\in X$, it holds that $d_Y(f(x),f(y)) \leq Ld_X(x,y)$. The minimal such $L$ is called the \emph{Lipschitz constant} of $f$ and is denoted by $\Lip(f)$. The map $f$ is a \emph{biLipschitz embedding of distortion $D$} if $f$ is injective and there exists a \emph{scaling factor} $s\in (0,\infty)$ such that $f$ is $sD$-Lipschitz and $f^{-1}$ is $s^{-1}$-Lipschitz. If $f$ is also surjective, then it is a \emph{biLipschitz equivalence}. BiLipschitz embeddings and equivalences of distortion 1 are \emph{isometric embeddings} and \emph{isometries}, respectively. The map $f$ is a \emph{rough biLipschitz embedding} (or \emph{quasi-isometric embedding}) if there are $K,L<\infty$ such that $L^{-1}d_X(x,y) - K \leq d_Y(f(x),f(y)) \leq Ld_X(x,y) + K$ for all $x,y \in X$. If, in addition, the image of $f$ is \emph{coarsely dense}, meaning there exists $C<\infty$ such that $\dist(f(X),y) < C$ for every $y \in Y$, then $f$ is a \emph{rough biLipschitz equivalence} (or \emph{quasi-isometry}). If there exist $\alpha \in (0,\infty)$ and $L<\infty$ such that $L^{-1}d_X(x,y) \leq d_Y(f(x),f(y))^\alpha \leq Ld_X(x,y)$ for all $x,y \in X$, then $f$ is a \emph{snowflake embedding}. Surjective snowflake embeddings are called \emph{snowflake equivalences}. Note that snowflake embeddings (and thus biLipschitz embeddings) are always measurable\footnote{Throughout this article, every topological space is considered as a measurable space equipped with its Borel $\sigma$-algebra, and when we refer to a topological-space-valued map as being measurable, we mean with respect to the Borel $\sigma$-algebra.} since they are continuous, but rough biLipschitz embeddings need not be measurable (indeed, any map between two bounded metric spaces is a rough biLipschitz equivalence).

\subsection{Banach and Bochner-Lebesgue Spaces}
Let $T: V \to W$ be a linear map between Banach spaces\footnote{All Banach spaces in this article are over the real numbers $\R$, and all linear maps are $\R$-linear.} and $C<\infty$. The map $T$ is \emph{$C$-bounded} if it is $C$-Lipschitz, $T$ is a \emph{$C$-isomorphic embedding} if it is biLipschitz with distortion $C$, and $T$ is a \emph{$C$-isomorphism} if it is a surjective $C$-isomorphic embedding. We will use the notation ``$V \approx W$" to indicate that there exists an isomorphism between $V$ and $W$. Linear maps that are 1-isomorphic embeddings or 1-isomorphisms are \emph{linear isometric embeddings} and \emph{linear isometries}, respectively. A subspace $U \sbs W$ is \emph{$C$-complemented} if there exists a $C$-bounded linear map $R: W \to U$ such that $R(u) = u$ for every $u \in U$. It holds that $V$ is $C_1$-isomorphic to a $C_1$-complemented subspace of $W$ $\iff$ there exist $C_2$-bounded linear maps $T: V \to W$ and $R: W \to V$ such that $R \circ T = id_V$ (we say that $R$ \emph{retracts} $T$) $\iff$ there exists a Banach space $U$ such that $U \oplus V$ is $C_3$-isomorphic to $W$, where in each equivalence, $C_{i+1}$ depends only on $C_{i}$ and vice versa. We will use the notation ``$V\cembed W$" to indicate that $V$ is isomorphic to a complemented subspace of $W$. Here and in the sequel, whenever $\{(V_\lambda,\|\cdot\|_\lambda)\}_{\lambda\in\Lambda}$ is an indexed family of Banach spaces, we denote the $\ell^1$-sum of the family by $\oplus^1_{\lambda\in\Lambda} V_\lambda$, or just $\oplus_{\lambda\in\Lambda} V_\lambda$ without the ``1" in the superscript. That is, $\oplus^1_{\lambda\in\Lambda} V_\alpha$ consists of all $\Lambda$-indexed generalized sequences $(\boldsymbol{v}_\lambda)_{\lambda\in\Lambda}$, $\boldsymbol{v}_\lambda \in V_\lambda$, with finite norm $\|(\boldsymbol{v}_\lambda)_{\lambda\in\Lambda}\| := \sum_{\lambda\in\Lambda} \|\boldsymbol{v}_\lambda\|_\lambda$. Whenever each of two Banach spaces $V,W$ is $C$-isomorphic to a $C$-complemented subspace of the other, and one of the two, say $V$, is $C$-isomorphic to its countable $\ell^1$-sum $\oplus^1_{n\in\N} V$, then the \emph{Pe\l czy\'{n}ski decomposition method} yields a $C'$-isomorphism $V \approx W$, where $C'$ depends only on $C$ \cite[Proposition~F.9]{BL}.

As previously stated, a Banach space is an \emph{$L^1$-space} if it equals the Lebesgue space $L^1(\Omega,\F,\mu)$ for some measure space $(\Omega,\F,\mu)$. When $\Omega$ equals $[0,1]$ equipped with the Borel $\sigma$-algebra and Lebesgue measure, we abbreviate $L^1(\Omega,\F,\mu)$ by $L^1$. When $\Omega$ equals $\N$ equipped with the power set $\sigma$-algebra and counting measure, we abbreviate $L^1(\Omega,\F,\mu)$ by $\ell^1$. It holds that any separable subspace of an $L^1$-space is contained in a separable $L^1$-subspace (see, for example, \cite[Fact~1.20]{Ostrovskii}), and any separable, infinite dimensional $L^1$-space is isomorphic to $L^1$ or to $\ell^1$ \cite[page~15]{JL}.

Let $(\Omega,\F,\mu)$ be a measure space, $(V,\|\cdot\|)$ a Banach space, and $g: \Omega \to V$ a function. We say that $g$ is \emph{Bochner measurable} if $g$ is the $\mu$-almost everywhere pointwise limit of a sequence of simple functions, where a function is \emph{simple} if it belongs to the linear span of the indicator functions $\{\one_Fv\}_{F\in\F,v\in V}$. The function $g$ is \emph{weakly measurable} if $\lambda \circ g: \Omega \to \R$ is measurable for every $\lambda \in V^*$. The function $g$ is \emph{essentially-separably-valued} if there exist a separable subspace $S \sbs V$ and a $\mu$-null set $N$ such that $g(\omega) \in S$ for every $\omega\in\Omega\setminus N$. Pettis' measurability theorem \cite[Theorem~1.2]{DU} characterizes Bochner measurable functions as exactly those that are weakly measurable and essentially-separably-valued. Since Bochner measurability immediately implies measurability and measurability immediately implies weak measurability, it holds that a function is Bochner measurable if and only if it is measurable and essentially-separably-valued. This latter characterization is the one we will most frequently use throughout the article. The set of Bochner measurable functions $g$ with finite norm $\|g\|_{L^1(\Omega;V)} := \int_\Omega \|g\| d\mu < \infty$ forms a Banach space (the \emph{Bochner-Lebesgue space}), denoted by $L^1(\Omega,\F,\mu;V)$, or just $L^1(\Omega;V)$ if the $\sigma$-algebra and measure are understood. Also, if $V=\R$, we compress notation again and just write $L^1(\Omega,\F,\mu)$ or $L^1(\Omega)$. Simple functions are dense in $L^1(\Omega;V)$, and thus the integral $\int_\Omega h d\mu$, which is defined in the obvious way for simple functions $h$, admits a unique continuous extension to all $h \in L^1(\Omega;V)$. Furthermore, if $A$ is a dense subset of $V$, then the set of simple functions $\Omega \to A$ taking values in $A$ is dense in $L^1(\Omega;V)$. Using this fact, it is easy to prove that, whenever $(\Omega',\F',\mu')$ is a second measure space, $L^1(\Omega,\F,\mu;L^1(\Omega',\F',\mu'))$ is canonically linearly isometric to $L^1(\Omega \times \Omega',\F\otimes\F',\mu\otimes\mu')$. This implies that if $V$ is an $L^1$-space, then $L^1(\Omega;V)$ is another $L^1$-space. If $\mu$ is a probability measure, we will often use probabilistic notation and denote $\int_\Omega h d\mu$ by $\E_{\omega\in\Omega}[h_\omega]$, $\E_\omega[h_\omega]$, or just $\E[h]$. We refer the reader to \cite{DU} for further background on Bochner-Lebesgue spaces.

\subsection{Lipschitz Free Spaces}
\label{ss:LF}
A \emph{pointed set} is a set $X$ equipped with some distinguished \emph{basepoint} $x_0 \in X$. A \emph{pointed metric space} is a metric space $(X,d)$ such that $X$ is a pointed set. We adopt the following important convention throughout the article:
\begin{center}
\emph{Every subset of a vector space containing 0 will automatically be considered as a pointed metric space with basepoint 0.}
\end{center}
A map between pointed sets is \emph{basepoint-preserving} if it maps the basepoint of the domain to the basepoint of the codomain. Given a pointed metric space $X$, we denote by $\Lip_0(X)$ the vector space of basepoint-preserving Lipschitz functions $f:X\to\R$. The vector space $\Lip_0(X)$ becomes a Banach space when equipped with the norm $\|f\|_{\Lip_0(X)} := \Lip(f)$. There is a canonical isometric embedding $\delta: X \hookrightarrow \Lip_0(X)^*$ given by the pointwise evaluation map: $\delta(x) := \delta_x$ for $x \in X$, where $\delta_x(f) := f(x)$ for $f \in \Lip_0(X)$. The closed linear span of $\delta(X)$ in $\Lip_0(X)^*$ is called the \emph{Lipschitz free space}, or just \emph{free space}, of $X$, which we denote by $\LF(X)$, or by $\LF(X,d)$ when we want to emphasize the metric. The isometry type of $\LF(X)$ is independent of the chosen basepoint, and thus we omit it from the notation. The dual space of $\LF(X)$ is $\Lip_0(X)$, and the induced weak$^*$ topology on bounded subsets of $\Lip_0(X)$ is the topology of pointwise convergence. By the Krein-Smulian theorem, a bounded linear map $T: \Lip_0(X) \to V^*$ is the adjoint of a linear map $T_*: V \to \LF(X)$ if and only if $T(f_\alpha)$ weak$^*$ converges to $T(f)$ whenever $f_\alpha$ is a bounded net in $\Lip_0(X)$ converging pointwise to $f$. In this case, $T_*$ is unique and called the \emph{preadjoint} of $T$. Whenever $f: X \to E$ is a basepoint-preserving Lipschitz map into a Banach space, there exists a unique bounded linear map $T_f: \LF(X) \to E$ satisfying $T_f(\delta_x) = f(x)$ for every $x \in X$, and the operator norm of $T_f$ is $\Lip(f)$. We call $T_f$ the \emph{induced linear map} of $f$. If $Y \sbs X$ contains the basepoint, the formal identity map $\LF(Y) \to \LF(X)$ is a linear isometric embedding, and it is surjective if and only if $Y$ is dense in $X$. This follows from the \emph{(Whitney-)McShane extension theorem}, which states that every Lipschitz map $Y \to \R$ extends to a Lipschitz map $X \to \R$ with the same Lipschitz constant. Other standard facts we will use are that $\LF(X)$ is separable if and only if $X$ is separable and that biLipschitz embeddings (resp. equivalences) $X \to Y$ of distortion $C<\infty$ induce $C$-isomorphic embeddings (resp. equivalences) $\LF(X) \to \LF(Y)$. See \cite[Chapter~3]{Weaver} for a standard reference on Lipshitz free spaces, and note that these spaces are also referred to as \emph{Arens-Eells spaces} in that text and elsewhere.

The class of Lipschitz free spaces often exhibits qualitative behavior similar to that of the class of $L^1$-spaces. For example, the following dichotomy in $L^1$-spaces holds: either $L^1(\Omega)$ contains a complemented subspace isomorphic to $L^1$, or it has the Radon-Nikodym property (see \cite[Chapter~III, $\S$1, Definition~3]{DU} for the definition), with the latter property holding if and only if $\Omega$ is purely atomic. The same dichotomy holds in Lipschitz free spaces, with the property of pure 1-unrectifiability taking the place of pure atomicity.

\begin{lemma}[almost Theorem~C, \cite{AGPP}] \label{lem:freespacedichotomy}
Let $X$ be a metric space. Then the following dichotomy holds: either $\LF(X)$ contains a complemented subspace isomorphic to $L^1$, or $\LF(X)$ has the Radon-Nikodym property, with the latter property holding if and only if the completion of $X$ is purely 1-unrectifiable.
\end{lemma}

\begin{proof}
Assume that $\LF(X)$ fails to have the Radon-Nikodym property. Then by \cite[Theorem~C]{AGPP}, $\overline{X}$ is not purely 1-unrectifiable, where $\overline{X}$ is the completion of $X$. Hence, there is a positive measure subset $A \sbs \R$ and a biLipschitz embedding $f: A \to \overline{X}$. Since Nagata dimension (see $\S$\ref{ss:Nagata} for background) is nonincreasing under biLipschitz embeddings and the Nagata dimension of $\R$ is 1, the Nagata dimension of $f(A)$ is at most 1. Then by \cite[Lemma~3.2]{FG}, $\LF(f(A))$ is complemented in $\LF(\overline{X}) = \LF(X)$. Since $f$ is biLipschitz, $\LF(f(A)) \approx \LF(A)$, and by \cite[Corollary~3.4]{Godard}, $\LF(A) \approx L^1$.
\end{proof}

With the previous dichotomy lemma in hand, we can obtain a positive solution to the \emph{Complemented Subspace of $L^1$ Problem} -- which asserts that separable, complemented subspaces of $L^1$-spaces are isomorphic to $L^1$-spaces (see \cite[page~129]{AO}) -- when restricted to the class of Lipschitz free spaces.

\begin{theorem} \label{thm:freespacecomplementedinL1}
Let $X$ be an infinite separable metric space. If $\LF(X)$ is isomorphic to a complemented subspace of an $L^1$-space, then $\LF(X) \approx L^1$ or $\LF(X) \approx \ell^1$, with the latter isomorphism holding if and only if the completion of $X$ is purely 1-unrectifiable.
\end{theorem}

\begin{proof}
Assume that $\LF(X)$ is isomorphic to a complemented subspace of an $L^1$-space. Since $X$ is separable, we may assume that this $L^1$-space is $L^1$. By Lemma~\ref{lem:freespacedichotomy}, we only need to consider the cases where $\LF(X)$ contains a complemented subspace isomorphic to $L^1$ or $\LF(X)$ has the Radon-Nikodym property, with the latter case holding if and only if the completion of $X$ is purely 1-unrectifiable. In the first case, it follows from the Pe\l czy\'{n}ski decomposition method that $\LF(X) \approx L^1$. In the second case, it follows from the Lewis-Stegall theorem \cite[Theorem~IV.5.3]{DU} that $\LF(X) \approx \ell^1$.
\end{proof}

\subsection{Nagata Dimension and Lipschitz Extensions}
\label{ss:Nagata}
Let $X$ be a metric space. A \emph{Nagata cover of dimension $n\in\N$, constant $\gamma<\infty$, and scale $s>0$} is a cover $\mathcal{C}$ of $X$ such that $\diam(C) \leq \gamma s$ for every $C \in \mathcal{C}$ and, for every $A \subset X$ with $\diam(A) \leq s$, it holds that $|\{C \in \mathcal{C}: C \cap A \neq \emptyset\}| \leq n+1$ (where $|S|$ denotes the cardinality of a set $S$). It is not hard to see that such a cover exists if and only if there exists a cover $B_1,B_2,\dots B_{n+1}$ of $X$ such that each $B_i$ is the union of a family of sets that is \emph{$\gamma s$-bounded and $s$-separated}, meaning $B_i = \bigcup_{\lambda\in\Lambda_i} B_i^\lambda$, where $\diam(B_i^{\lambda_1}) \leq \gamma s$ and $\dist(B_i^{\lambda_1},B_i^{\lambda_2}) > s$ for all $\lambda_1 \neq \lambda_2 \in \Lambda_i$. In this case, the sets $\{B_i^\lambda\}_{i=1,\lambda\in\Lambda_i}^{n+1}$ form a Nagata cover. A metric space $X$ has \emph{(Assouad-)Nagata dimension $n \in \N$ with constant $\gamma < \infty$} if, for every $s>0$, there exists a Nagata cover of dimension $n$, constant $\gamma$, and scale $s$. If such an $n$ and $\gamma$ exist, we say that $X$ has \emph{finite Nagata dimension}. We also call $\gamma$ the \emph{$n$-Nagata-dimensional constant of $X$}.

The article of Lang and Schlichenmeier \cite{LS} (and references therein) contains some foundational results on Nagata dimension. For example, Nagata dimension is nonincreasing under snowflake embeddings (\cite[Lemma~2.1]{LS}), the Nagata dimension of an open subset of $\R^n$ is $n$ (\cite[page~3626]{LS}), every doubling space has finite Nagata dimension (\cite[Lemma~2.3]{LS}), ultrametric spaces have Nagata dimension 0, and nontrivial $\R$-trees have Nagata dimension 1 (see $\S$\ref{ss:umsRtrees} for these last two statements). The most important result of \cite{LS} for us concerns extensions of Lipschitz functions. Let us say that a metric space $Z$ is \emph{Lipschitz $\infty$-connected} with constant $L<\infty$ if for every Banach space $V$, every Lipschitz map $f: S_V \to Z$ extends to an $L\Lip(f)$-Lipschitz map $B_V \to Z$ (where, as usual, $S_V$ denotes the unit sphere and $B_V$ the unit ball of $V$). Theorem~1.6 of \cite{LS} implies that every Lipschitz map $f: M \to Z$ from a closed subset $M \sbs Y$ of a metric space of finite Nagata dimension into a Lipschitz $\infty$-connected metric space $Z$ extends to a Lipschitz map $\overline{f}: Y \to Z$, where $\Lip(\overline{f})$ depends only on the relevant data of $Y,Z,f$.

\begin{lemma}[Theorem~1.6, \cite{LS}] \label{lem:LS}
Let $Y$ be a metric space of Nagata dimension $n\in\N$ with constant $\gamma<\infty$, $M \sbs Y$ a closed subset, and $Z$ a Lipschitz $\infty$-connected metric space with constant $L<\infty$. Then for every Lipschitz map $f: M \to Z$, there exists an $L'\Lip(f)$-Lipschitz extension $\overline{f}: Y \to Z$, where $L'$ depends only on $n,\gamma,L$.
\end{lemma}

\noindent We will use Lemma~\ref{lem:LS} in the proof of Lemma~\ref{lem:Nagataconvexprojection} to build uniform-bounded Lipschitz extension operators on metrics spaces of finite Nagata dimension, which will be defined shortly.

Let $K<\infty$, $(Y,d)$ a metric space, and $M \sbs Y$ a closed subset. We call a map $y \mapsto \mu_y: Y \to \LF(M)$ a \emph{convex $K$-random projection onto $M$} if, for every $y,z \in Y$ and $m \in M$,
\begin{itemize}
    \item $\mu_y \in \cconv(\delta(M))$,
    \item $\mu_m = \delta_m$, and
    \item $\|\mu_y-\mu_z\|_{\LF} \leq Kd(y,z)$,
\end{itemize}
where $\cconv(S)$ denotes the closed convex hull of a subset $S$ of a Banach space. Convex $K$-random projections were defined in \cite{AP} under the name \emph{$K$-strong random projections} to distinguish them from the larger class of \emph{$K$-random projections} which include all maps satisfying the second and third items but not necessarily the first. We use the term ``convex" in place of ``strong" as the former is a bit more descriptive.

Let $K,L<\infty$, $X$ a pointed metric space, and $M \sbs X$ a closed subset containing the basepoint. We call $E: \Lip_0(M) \to \Lip_0(X)$ a \emph{$K$-Lipschitz-bounded linear extension operator} if it is a $K$-bounded linear map and $E(f)(m) = f(m)$ for every $f \in \Lip_0(M)$ and $m \in M$. If additionally, $\|E(f)\|_\infty \leq L \|f\|_\infty$ for every $f \in \Lip_0(M)$ (where, as usual, $\|f\|_\infty := \sup_{m\in M}|f(m)|$), then $E$ is \emph{$L$-uniform-bounded}. We say that $X$ has the \emph{$K$-Lipschitz-bounded linear extension property} if for every closed subset $M \sbs X$ containing the basepoint, there exists a $K$-Lipschitz-bounded linear extension operator $E: \Lip_0(M) \to \Lip_0(X)$. If $E$ can additionally be chosen to be $L$-uniform-bounded, then $X$ has the \emph{$K$-Lipschitz-bounded, $L$-uniform-bounded linear extension property}.

\begin{lemma}
Let $(Y,d)$ be a pointed metric space, $M \sbs Y$ a closed subset containing the basepoint, and $K<\infty$. If $Y$ admits a convex $K$-random projection onto $M$, then there exists a $K$-Lipschitz-bounded, 1-uniform-bounded linear extension operator $\Lip_0(M) \to \Lip_0(Y)$.
\end{lemma}

\begin{proof}
Assume that there exists a convex $K$-random projection $y \mapsto \mu_y:$ $Y \to \cconv(\delta(M))$ $\sbs \LF(M)$. Define $E: \Lip_0(M) \to \Lip_0(Y)$ by $E(f)(y) := f(\mu_y)$. Then for all $y,z \in Y$ and $m \in M$, $|E(f)(y)-E(f)(z)| \leq \Lip(f)\|\mu_y-\mu_z\|_{\LF} \leq K\Lip(f)d(y,z)$ and $E(f)(m) = f(\delta_m) = f(m)$. Therefore, $E$ is a $K$-Lipschitz-bounded linear extension operator. Next, we verify 1-uniform-boundedness.

Let $f\in \Lip_0(M)$ and $y \in Y$. It is immediate that $|f(\nu)| \leq \|f\|_{\infty}$ for any $\nu \in \conv(\delta(M)) \sbs \LF(M)$, and by Lipschitz continuity of $f$, the inequality remains valid for $\nu \in \cconv(\delta(M))$. Therefore, since $\mu_y \in \cconv(\delta(M))$, we have that $|E(f)(y)| = |f(\mu_y)| \leq \|f\|_\infty$. Since $y \in Y$ was arbitrary, we get $\|E(f)\|_\infty \leq \|f\|_\infty$.
\end{proof}

The next lemma we present can, in many cases, be deduced by combining established results (specifically, one could apply \cite[Lemma~5.1]{NS}, \cite[Theorem~4.1]{LN}, \cite[Lemma~3.8]{LN}, and \cite[Theorem~2.10]{AP}). However, we found there to be less technical overhead in citing Lemma~\ref{lem:LS} and then filling in the remaining details ourselves.

\begin{lemma} \label{lem:Nagataconvexprojection}
Let $Y$ be a metric space of Nagata dimension $n\in\N$ with constant $\gamma<\infty$. Then there exists $K<\infty$, depending only on $n,\gamma$, such that $Y$ admits a convex $K$-random projection onto every closed subset.
\end{lemma}

\begin{proof}
Let $M \sbs Y$ be closed. By definition of convex $K$-random projections, we need to find a $K$-Lipschitz map $Y \to \cconv(\delta(M))$ extending the isometric embedding $\delta: M \to \LF(M)$, where $K$ depends only on $n,\gamma$. By Lemma~\ref{lem:LS}, it is enough to verify that every closed convex subset of a Banach space is Lipschitz $\infty$-connected with universal constant (we will get a constant of 4).

Let $V$ be a Banach space and $C$ a closed convex subset of some Banach space. Let $f: S_V \to C$ be a Lipschitz map. By composing with translations, it suffices to assume that $0 \in f(S_V)$. We will define the Lipschitz extension $\overline{f}: B_V \to C$ on $B_V \setminus \{0\}$, and then by Lipschitz continuity and completeness of $C$, it will extend to a map on all of $B_V$ with the same Lipschitz constant.

There is a homeomorphism from $B_V \setminus \{0\}$ to $(0,1] \times S_V$ given by $v \mapsto (\|v\|, \frac{v}{\|v\|})$ with inverse $(t,\theta) \mapsto t\theta$. By the reverse triangle inequality, $\|s\theta - t\eta\| \geq |t-s|$, and then applying the reverse triangle inequality again, one can get $\|s\theta - t\eta\| \geq \max\{s,t\}\|\theta-\eta\| - |t-s|$. Averaging these yields
\begin{equation} \label{eq:polarcoords1}
    \|s\theta - t\eta\| \geq \max\{|t-s|,\tfrac{1}{2}\max\{s,t\}\|\theta-\eta\|\}.
\end{equation}
We define $\overline{f}$ on $t\theta$ by $\overline{f}(t\theta) := tf(\theta)$. This is well-defined because $C$ is convex, $t \in [0,1]$, and $0 \in f(S_V) \sbs C$, and it obviously extends $f$. It remains to bound the Lipschitz constant of $\overline{f}$. Let $t\theta,s\eta \in B_V\setminus\{0\}$. Then by the triangle inequality, and the fact that $\|f\|_\infty \leq 2\Lip(f)$ (since $0 \in f(S_V)$ and $\diam(S_V) = 2$), we get
\begin{equation} \label{eq:polarcoords2}
    \|\overline{f}(s\theta)-\overline{f}(t\eta)\| = \|sf(\theta)-tf(\eta)\| \leq (2|t-s| + \min\{s,t\}\|\theta-\eta\|)\Lip(f)
\end{equation}
Combining \eqref{eq:polarcoords1} and \eqref{eq:polarcoords2} yields $\Lip(\overline{f}) \leq 4\Lip(f)$.
\end{proof}

Combining these two lemmas, we obtain:

\begin{corollary} \label{cor:Nagatauniformbnddextension}
Let $Y$ be a pointed metric space of finite Nagata dimension $n\in\N$ with constant $\gamma<\infty$. Then there exists $K<\infty$, depending only on $n,\gamma$, such that $Y$ has the $K$-Lipschitz bounded, 1-uniform bounded linear extension property.
\end{corollary}

\subsection{Hyperbolic Metric Spaces}
\label{ss:hyperbolic}
We mainly follow the presentation of hyperbolic metric spaces and their boundaries as given in Bonk-Schramm \cite{BS}. For more comprehensive accounts of this material, see \cite{Buyalo} or \cite[Chapter~III.H]{BH}.

Let $(X,d)$ be a metric space and $x,y,p \in X$. The \emph{Gromov product} of $x,y$ with respect to $p$ is defined by
\begin{equation*} \label{eq:Gromovproductdef}
    (x|y)_p := \frac{1}{2}(d(x,p)+d(p,y)-d(x,y)).
\end{equation*}
Let $\delta \in [0,\infty)$. Then $X$ is said to be \emph{$\delta$-hyperbolic} if for all $x,y,z,p \in X$,
\begin{equation*}
    (x|z)_p \geq \min\{(x|y)_p,(y|z)_p\} - \delta.
\end{equation*}
The space $X$ is \emph{Gromov hyperbolic}, or just \emph{hyperbolic}, if it is $\delta$-hyperbolic for some $\delta \in [0,\infty)$.

Suppose $(X,d)$ is a hyperbolic space. A sequence $\{x_i\}_{i=1}^\infty \sbs X$ is said to \emph{converge at $\infty$} if for some (equivalently, for all) $p \in X$, $\lim_{i,j\to\infty} (x_i|x_j)_p = \infty$. Two sequences $\{x_i\}_{i=1}^\infty, \{y_i\}_{i=1}^\infty \sbs X$ converging at $\infty$ are \emph{equivalent} if for some (equivalently, for all) $p \in X$, $\lim_{i\to\infty} (x_i|y_i)_p = \infty$.

This is an equivalence relation on the set of sequences converging at $\infty$ by hyperbolicity, and the set of equivalences classes is called the \emph{Gromov boundary}, or just \emph{boundary}, of $X$, denoted by $\partial X$. The \emph{Gromov product} of $\xi,\upsilon \in \partial X$ with respect to $p \in X$ is defined by
\begin{equation*}
    (\xi|\upsilon)_p := \sup \{\liminf_{i\to\infty} (x_i|y_i)_p: \{x_i\}_{i=1}^\infty \in \xi,\{y_i\}_{i=1}^\infty \in \upsilon\}.
\end{equation*}
A metric $d$ on $\partial X$ is said to be \emph{visual} if there exist $p \in X$, $\eps > 0$, and $C< \infty$ such that $C^{-1}e^{-\eps(\xi|\upsilon)_p} \leq d(\xi,\upsilon) \leq Ce^{-\eps(\xi|\upsilon)_p}$ (in Bonk-Schramm, the set of visual metrics is called the \emph{canonical $B$-structure} on $\partial X$ \cite[page~282]{BS}). Visual metrics on $\partial X$ always exist and are unique up to snowflake equivalence, and it is clear from the definition that any metric snowflake equivalent to a visual metric is also visual. Thus, we always equip $\partial X$ with some visual metric, and we may unambiguously assert that $\partial X$ satisfies or fails a metric property $P$ as long as $P$ is invariant under snowflake equivalences (for example, having Nagata dimension $n\in\N$). The boundary of a hyperbolic space is always bounded and complete. Under the additional hypothesis of \emph{visibility}, which will be described in the next paragraph, the boundary of a hyperbolic space completely determines the hyperbolic space up to rough biLipschitz equivalence. This theory was worked out in full by Bonk and Schramm \cite{BS}.

Let $(X,d)$ be a metric space, $K<\infty$, and $I = [0,b]$ or $I = [0,\infty)$. The image of $I$ under a map $f: I \to X$ satisfying $|s-t| - K \leq d(f(s),f(t)) \leq |s-t| + K$ is called a \emph{$K$-rough geodesic} if $I = [0,b]$ and a \emph{$K$-rough geodesic ray} if $I = [0,\infty)$. If $I = [0,b]$, the points $f(0),f(b)$ are called the \emph{endpoints} of the rough geodesic, and if $I = [0,\infty)$, the point $f(0)$ is called the \emph{origin} of the rough geodesic ray. A metric space is \emph{$K$-roughly geodesic} if every pair of points are the endpoints of some $K$-rough geodesic. A \emph{geodesic} is a 0-rough geodesic, and similar definitions hold for geodesic rays and geodesic metric spaces. A metric space is \emph{visual} if there exists (equivalently, for all) $p \in X$ and there exists $K' <\infty$ such that $X$ is the union of all $K'$-rough geodesic rays originating from $p$.

Let $(Z,d_Z)$ be a bounded metric space. The \emph{(Bonk-Schramm) hyperbolic filling} of $Z$ is the set
\begin{equation} \label{eq:Hypdef}
    \Hyp(Z) := Z \times (0,\diam(Z)]
\end{equation}
(note that $\Hyp(Z)$ is called Con$(Z)$ in \cite{BS}) equipped with the metric
\begin{equation} \label{eq:rhodef}
    \rho_Z((z,h),(z',h')) := 2\log\left(\dfrac{d_Z(z,z')+h \vee h'}{\sqrt{hh'}}\right),
\end{equation}
where here and in the sequel, $\log$ denotes the natural logarithm and $h \vee h'$ denotes $\max\{h,h'\}$. Note that $\rho_Z$ induces the product topology on $Z \times (0,\diam(Z)]$. By \cite[Theorem~7.2]{BS}, $\Hyp(Z)$ is always a visual hyperbolic metric space. If $f: Z \to Y$ is any map to another bounded metric space with $\diam(Z) \leq \diam(Y)$, then there is an induced map $\Hyp(f): \Hyp(Z) \to \Hyp(Y)$ defined by
\begin{equation*}
    \Hyp(f)(z,h) := (f(z),h).
\end{equation*}
The salient feature of hyperbolic fillings is that, in the visual case, they invert Gromov boundaries up to rough biLipschitz equivalence.

\begin{lemma}[Theorem~8.2, \cite{BS}] \label{lem:BS}
If $X$ is a visual hyperbolic metric space, then $X$ and $\Hyp(\partial X)$ are roughly biLipschitz equivalent.
\end{lemma}

\noindent Lemma~\ref{lem:BS} plays a crucial role in the proof of Theorem~\ref{thm:Hypstochasticums}.

\subsection{Ultrametric Spaces and $\R$-Trees} \label{ss:umsRtrees}
A subset of a topological space is an \emph{arc} if it is homeomorphic to $[0,1]$. The points in the image of $\{0,1\}$ under any such homeomorphism are called the \emph{endpoints} of the arc. A metric space $X$ is an \emph{$\R$-tree} if any two points $x,y \in X$ are the endpoints of a unique arc, denoted $[x,y]$, and this arc is a geodesic. It holds that a metric space is an $\R$-tree if and only if it is geodesic and 0-hyperbolic (This follows, for example, from \cite[Theorem~3.40]{Evans}), and every 0-hyperbolic metric space isometrically embeds into an $\R$-tree (see \cite[Theorem~3.38]{Evans}). For future reference, we note that $(X,d)$ is $0$-hyperbolic if and only if it satisfies the 4-point condition
\begin{equation} \label{eq:0hypdef}
    d(w,x) + d(y,z) \leq \max\{d(w,y)+d(x,z),d(w,z)+d(x,y)\}
\end{equation}
for all $w,x,y,z \in X$ (see \cite[page~410]{BH} or \cite[Lemma~3.12]{Evans}). The \emph{convex hull} of a subset $S \sbs T$ of an $\R$-tree is the set $\cup_{s,t \in S} [s,t]$. The convex hull of $S$ is itself an $\R$-tree, and it is separable whenever $S$ is separable. Consequently, any separable 0-hyperbolic metric space isometrically embeds into a separable $\R$-tree. The Nagata dimension of any nontrivial $\R$-tree (one that contains more than one point) is 1 with constant $\gamma \leq 5$ (indeed, by the proofs of \cite[Lemma~3.1, Theorem~3.2]{LS}, $\gamma \leq 4\gamma_{\R}+1$, where $\gamma_\R$ 1-Nagata-dimensional constant of $\R$, which is easily seen to be 1). The Lipschitz free space of any $\R$-tree is linearly isometric to an $L^1$-space \cite[Corollary~3.3]{Godard}.

A metric $d$ on a set $U$ is called an \emph{ultrametric} if the inequality
\begin{equation*}
    d(u,w) \leq \max\{d(u,v),d(v,w)\}
\end{equation*}
holds for every $u,v,w \in U$. In this case, the pair $(U,d)$ (or just $U$ if $d$ is understood) is called an \emph{ultrametric space}. A metric space is biLipschitz equivalent to an ultrametric space if and only if it is snowflake equivalent to an ultrametric space if and only if it has Nagata dimension 0 (see, for example, \cite[Theorem~3.3]{Ndim0ums}). Every ultrametric space is 0-hyperbolic and thus isometrically embeds into an $\R$-tree. The Lipschitz free space of any infinite, separable ultrametric space is isomorphic to $\ell^1$ \cite[Theorem~2]{CD}.

It is well-known that the Gromov boundary of an $\R$-tree is (biLipschitz equivalent to) an ultrametric space (see, for example, \cite[page 320]{umsfillings2}). In the proof of Theorem~\ref{thm:C} (via Theorem~\ref{thm:Hypstochasticums}), we will need a converse statement that the hyperbolic filling of a bounded, complete ultrametric space is roughly biLipschitz equivalent to an $\R$-tree. This fact is also well-known (see, for example, \cite[Proposition~3.2]{umsfillings1}), but we will require the rough biLipschitz map from the Bonk-Schramm filling $\Hyp(U)$ to the $\R$-tree to be measurable and to map separable sets to separable sets. This is proved in the next lemma. The $\R$-tree is constructed using a common hyperbolic filling technique similar to \cite[Chapter~6]{Buyalo}.

\begin{lemma} \label{lem:Hyp(U)T}
Let $U$ be a bounded ultrametric space. Then there exist an $\R$-tree $T$ and a rough biLipschitz embedding $\Hyp(U) \to T$ that is measurable, maps separable subsets to separable subsets, and has coarsely dense image (and therefore $\Hyp(U)$ is roughly biLipschitz equivalent to $T$).
\end{lemma}

\noindent For this proof, we need to recall some definitions. Let $\theta>0$. A subset $A$ of a metric space $(X,d)$ is \emph{$\theta$-separated} if $d(a,b) > \theta$ for all $a \neq b \in A$. A \emph{maximal $\theta$-separated subset} of a metric space is one that is $\theta$-separated and not properly contained in any other $\theta$-separated subset. Equivalently, a $\theta$-separated subset $A \sbs X$ is maximal if $\{B_\theta(a)\}_{a\in A}$ is a cover of $X$. By Zorn's lemma, any point in a metric space is contained in some maximal $\theta$-separated subset. Note that a subset of a metric space is uniformly discrete if and only if it is $\theta$-separated for some $\theta>0$.

\begin{proof}[Proof of Lemma~\ref{lem:Hyp(U)T}]
First, we observe that once a rough biLipschitz embedding $f: \Hyp(U) \to T'$ into an $\R$-tree $(T',d)$ has been produced, the image will automatically be coarsely dense in its convex hull $T$ (which is another $\R$-tree) due to the fact that $\Hyp(U)$ is coarsely path-connected. More precisely, by \cite[Theorem~7.2]{BS}, $\Hyp(U)$ is roughly geodesic, and thus for each $x,y \in \Hyp(U)$, we can find a sequence of points $\{x_i\}_{i=0}^m \sbs \Hyp(U)$ such that $x_0 = x$, $x_m = y$, and for each $1 \leq i \leq m$, $\rho_U(x_{i-1},x_i) \leq K'$ for some $K' < \infty$ (independent of $x,y$). Then since $f$ is roughly Lipschitz, $d(f(x_{i-1}),f(x_i)) \leq K$ for some $K<\infty$ (independent of $x,y$). Hence, the $K$-neighborhood of $f(\Hyp(U))$ contains $\cup_{i=1}^m [f(x_{i-1}),f(x_i)]$. This letter set is a topological curve containing $\{f(x),f(y)\}$, and hence by uniqueness of arcs in $T'$, it contains $[f(x),f(y)]$. Since $x,y \in \Hyp(U)$ were arbitrary, this shows that $f(\Hyp(U))$ is $K$-coarsely dense in its convex hull. It remains to produce a rough biLipschitz embedding of $\Hyp(U)$ into an $\R$-tree. 

Let $d_U$ denote the metric on $U$ and $k_0 := \min \{k\in\Z: \diam(U) \leq e^k\}$. For this proof, it will be more convenient to not work with $\Hyp(U)$ exactly as it is defined in \eqref{eq:Hypdef}, but instead to replace the underlying set $U \times (0,\diam(U)]$ with the set $U \times (0,e^{k_0}]$. Obviously, the inclusion $U \times (0,\diam(U)] \sbs U \times (0,e^{k_0}]$ is an isometric embedding when each set is equipped with the metric $\rho_U$ defined in \eqref{eq:rhodef}.
% and it can be easily checked that $U \times (0,\diam(U)]$ is $\log(2)$-coarsely dense in $U \times (0,2^{k_0}]$.
Therefore, to prove the lemma, we can take $\Hyp(U)$ to be the set $U \times (0,e^{k_0}]$ equipped with the metric $\rho_U$.

For each $k \leq k_0$, choose a maximal $e^k$-separated subset $N_k$ of $U$. Note that $N_k$ is countable whenever $U$ is separable. Since $\diam(U) \leq e^{k_0}$, $N_{k_0}$ is a singleton. By maximality of $N_k$, for any $u \in U$, there exists $w \in N_k$ such that $d_U(u,w) \leq e^k$. Additionally, for any $u_1,u_2 \in U$ and $w_1, w_2 \in N_k$ with $d_U(u_i,w_i) \leq e^k$, the ultrametric inequality implies
\begin{align*}
    d_U(w_1,w_2) \leq \max\{d_U(w_1,u_1),d_U(u_1,u_2),d_U(u_2,w_2)\} \leq \max\{e^k,d_U(u_1,u_2)\}.
\end{align*}
Now, if $w_1 \neq w_2$, then this inequality and the $e^k$-separation of $N_k$ imply $\max\{e^k,d_U(u_1,u_2)\}$ $= d_U(u_1,u_2)$, and thus we get
\begin{align*}
    d_U(w_1,w_2) \leq d_U(u_1,u_2).
\end{align*}
Obviously, this inequality also holds if $w_1=w_2$. In particular, if we take $u_1=u=u_2$, this inequality shows that there is a unique $w \in N_k$ with $d_U(u,w) \leq e^k$. Putting this all together, we get that there exists a unique map $\pi_k: U \to N_k$ satisfying
\begin{equation} \label{eq:Hyp(U)embed1}
    d_U(u,\pi_k(u)) \leq e^k,
\end{equation}
and this map is 1-Lipschitz. One can quickly deduce from this that the tower property
\begin{equation} \label{eq:Hyp(U)embed2}
    \pi_{k} \circ \pi_{j} = \pi_k
\end{equation}
holds for all $j \leq k \leq k_0$, and additionally that
\begin{equation} \label{eq:pikretract}
    \pi_k(u) = u
\end{equation}
for all $u \in N_k$.

Set $N := \cup_{k \leq k_0} (N_k \times \{e^k\}) \sbs U \times (0,e^{k_0}] = \Hyp(U)$. Note that $N$ is countable whenever $U$ is separable. It can be quickly checked that for any two distinct elements $x=(u_x,e^{k_x}),y=(u_y,e^{k_y})$ of $N$, $\rho_U(x,y) \geq 1$ if $k_x \neq k_y$ and $\rho_U(x,y) \geq 2\log(2)$ if $k_x = k_y$, and hence in both cases we have
\begin{equation} \label{eq:Hyp(U)embed3}
    \rho_U(x,y) \geq 1.
\end{equation}
For each $x = (u_x,e^{k_x}), y = (u_y,e^{k_y}) \in N$, define $k(x,y) := \min\{k \geq k_x \vee k_y: \pi_k(u_x) = \pi_k(u_y)\}$ (note that this minimum is over a nonempty set since the range of $\pi_{k_0}$ is a singleton). For future reference, we record a couple of facts. The definition of $k(x,y)$, \eqref{eq:Hyp(U)embed1}, and the ultrametric inequality imply that
\begin{equation} \label{eq:Hyp(U)embed4}
   d_U(u_x,u_y) \leq e^{k(x,y)},
\end{equation}
and the definition of $k(x,y)$ and $e^{k(x,y)-1}$-separation of $N_{k(x,y)-1}$ imply that
\begin{equation*}
    k(x,y) = \max\{k_x,k_y\} \hspace{.25in} \text{or} \hspace{.25in} e^{k(x,y)-1} < d_U(u_x,u_y),
\end{equation*}
which in turn implies
\begin{equation} \label{eq:Hyp(U)embed5}
    \max\{e^{k_x},e^{k_y},d_U(u_x,u_y)\} > e^{k(x,y)-1}.
\end{equation}

Define the symmetric function $\rho_U^N$ on $N$ by
\begin{equation*}
    \rho_U^N(x,y) := 2k(x,y)-k_x-k_y.
\end{equation*}
One can use the definition of $k(x,y)$ and \eqref{eq:pikretract} to verify that $\rho_U^N(x,y) = 0$ if and only if $x=y$. Indeed, $x=y \implies \rho_U^N(x,y) = 0$ is clear. Now suppose that $\rho_U^N(x,y) = 0$. Then $k(x,y) = \frac{k_x+k_y}{2}$ by the definition of $\rho_U^N$, and $k(x,y) \geq k_x\vee k_y$ by the definition of $k(x,y)$. Together, these imply that $k(x,y) = k_x = k_y$. Thus, $u_x,u_y \in N_{k(x,y)}$, which by \eqref{eq:pikretract} implies that $\pi_{k(x,y)}(u_x) = u_x$ and $\pi_{k(x,y)}(u_y) = u_y$. Together with the definition of $k(x,y)$, this implies $u_x=u_y$. Hence, $x=y$. Since $\rho_U^N$ is integer-valued, we get as a consequence that
\begin{equation} \label{eq:Hyp(U)embed6}
    \rho_U^N(x,y) \geq 1
\end{equation}
for all $x \neq y \in N$.
The remainder of the proof will proceed as follows. We will prove that
\begin{itemize}
    \item[(i)] there is a rough biLipschitz embedding $(\Hyp(U),\rho_U) \to (N,\rho_U)$ that is measurable and maps separable subsets to separable subsets,
    \item[(ii)] the identity map $(N,\rho_U) \to (N,\rho_U^N)$ satisfies the two-sided estimate $\frac{1}{3}\rho_U^N(x,y) \leq \rho_U(x,y) \leq (2\log(2)+1)\rho_U^N(x,y)$ for all $x,y \in N$, and
    \item[(iii)] $(N,\rho_U^N)$ is a 0-hyperbolic metric space.
\end{itemize}
Since 0-hyperbolic spaces isometrically embed into $\R$-trees, the composition of the maps produced in these steps yields a rough biLipschitz embedding of $\Hyp(U)$ into an $\R$-tree that is measurable and maps separable subsets to separable subsets.

We begin with (i). The construction comes in two stages. We first define $U' := \cup_{k\leq k_0} (U \times \{e^k\}) \sbs \Hyp(U)$ and a map $\pi': \Hyp(U) \to U'$ defined by
\begin{equation*} \label{eq:Hyp(U)embed7}
    \pi'(u,h) := (u,e^{\lfloor\log(h)\rfloor}).
\end{equation*}
Obviously, this map is measurable, maps separable subsets to separable subsets, and satisfies
\begin{equation*}
    \rho_U(\pi'(u,h),(u,h)) \leq 1.
\end{equation*}
It is easy to check that this inequality implies that $\pi'$ is a rough biLipschitz embedding. It therefore remains to construct a rough biLipschitz embedding $\pi_N: U' \to N$ that is measurable and maps separable subsets to separable subsets. We define such a map by
\begin{equation*}
    \pi_N(u,e^k) := (\pi_k(u),e^k).
\end{equation*}
Since each $\pi_k$ is 1-Lipschitz, the map $\pi_N$ is 1-Lipschitz, and hence it is also measurable and maps separable subsets to separable subsets. Additionally,
\begin{align*}
    \rho_U(\pi_N(u,e^k),(u,e^k)) = 2\log(e^{-k}d_U(\pi_k(u),u)+1) \overset{\eqref{eq:Hyp(U)embed1}}{\leq} 2\log(2).
\end{align*}
As before, such an inequality implies that $\pi_N$ is a rough biLipschitz embedding.

We proceed to the proof of (ii). Let $x=(u_x,e^{k_x}),y=(u_y,e^{k_y})$ be distinct elements of $N$. Then we have the upper bound
\begin{align*}
    \rho_U(x,y) &= 2\log\left(\frac{d_U(u_x,u_y)+e^{k_x} \vee e^{k_y}}{\sqrt{e^{k_x}e^{k_y}}}\right) \\
    &\overset{\eqref{eq:Hyp(U)embed4}}{\leq}  2\log\left(\frac{e^{k(x,y)}+e^{k_x} \vee e^{k_y}}{\sqrt{e^{k_x}e^{k_y}}}\right) \\
    &\leq 2\log\left(\frac{e^{k(x,y)}+e^{k(x,y)}}{\sqrt{e^{k_x}e^{k_y}}}\right) \\
    &= 2\log(2)+\rho_U^N(x,y) \\
    &\overset{\eqref{eq:Hyp(U)embed6}}{\leq} (2\log(2)+1)\rho_U^N(x,y).
\end{align*}
For the lower bound, we need to split into two cases: (a) $0 < \rho_U^N(x,y) \leq 3$ and (b) $\rho_U^N(x,y) \geq 3$. In case (a), we have
\begin{align*}
    \rho_U(x,y) \overset{\eqref{eq:Hyp(U)embed3}}{\geq} 1 \geq \frac{1}{3}\rho_U^N(x,y),
\end{align*}
and in case (b) we have
\begin{align*}
    \rho_U(x,y) &= 2\log\left(\frac{d_U(u_x,u_y)+e^{k_x} \vee e^{k_y}}{\sqrt{e^{k_x}e^{k_y}}}\right) \\
    &\overset{\eqref{eq:Hyp(U)embed5}}{>} 2\log\left(\frac{e^{k(x,y)-1}}{\sqrt{e^{k_x}e^{k_y}}}\right) \\
    &= \rho_U^N(x,y)-2 \\
    &\geq \rho_U^N(x,y)-\frac{2}{3}\rho_U^N(x,y) \\
    &= \frac{1}{3}\rho_U^N(x,y).
\end{align*}
This finishes the proof of (ii).

We conclude by proving (iii). The first step towards this end is to prove that $k(\cdot,\cdot)$ satisfies the inequality
\begin{equation} \label{eq:Hyp(U)embed8}
    k(x,z) \leq \max\{k(x,y),k(y,z)\}
\end{equation}
for every $x,y,z \in N$. Let $x=(u_x,e^{k_x}),y=(u_y,e^{k_y}),z=(u_z,e^{k_z}) \in N$. Let $k_* := \max\{k(x,z),k(x,y),k(y,z)\}$. If $k_* \neq k(x,z)$, then \eqref{eq:Hyp(U)embed8} is proved. Thus, we may assume that $k_* = k(x,z)$. Suppose that \eqref{eq:Hyp(U)embed8} fails, and hence $k(x,y),k(y,z) \leq k_*-1$. This implies that $k_x,k_y,k_z \leq k_*-1$ and, by \eqref{eq:Hyp(U)embed2}, that $\pi_{k_*-1}(u_x) = \pi_{k_*-1}(u_y) = \pi_{k_*-1}(u_z)$, which in turn implies that $k(x,z) \leq k_*-1$. This contradicts $k_* = k(x,z)$. Hence, \eqref{eq:Hyp(U)embed8} must hold.

We are now in a position to verify \eqref{eq:0hypdef} for $\rho_U^N$. Note that \eqref{eq:0hypdef} already implies the triangle inequality, and thus the proof of (iii) is complete upon verification of \eqref{eq:0hypdef}. It is immediate from the definition of $\rho_U^N$ and some simple cancellations that \eqref{eq:0hypdef} is equivalent to
\begin{equation} \label{eq:Hyp(U)embed9}
    k(w,x)+k(y,z) \leq \max\{k(w,y)+k(x,z),k(w,z)+k(x,y)\}
\end{equation}
for all $w,x,y,z \in N$. Notice that \eqref{eq:Hyp(U)embed9} is invariant under each of the permutations $w \leftrightarrow x$ and $(w,x) \leftrightarrow (y,z)$. We will use these symmetries in the ensuing argument to reduce the number of cases needing verification.

Let $w,x,y,z \in N$. By the $(w,x) \leftrightarrow (y,z)$ symmetry, we may assume that
\begin{equation} \label{eq:Hyp(U)embed10}
    k(y,z) \leq k(w,x).
\end{equation}
Applying \eqref{eq:Hyp(U)embed8}, we have that $k(w,x) \leq \max\{k(w,y),k(x,y)\}$. Using this and the $w \leftrightarrow x$ symmetry (which preserves \eqref{eq:Hyp(U)embed10}), we may assume that
\begin{equation} \label{eq:Hyp(U)embed11}
    k(w,x) \leq k(w,y).
\end{equation}
Applying \eqref{eq:Hyp(U)embed8} again, we have
\begin{align}
\label{eq:Hyp(U)embed12}    k(w,x) &\leq k(x,z) \hspace{.25in} \text{or} \\
\label{eq:Hyp(U)embed13}    k(w,x) &\leq k(w,z).
\end{align}
In the case where \eqref{eq:Hyp(U)embed12} holds, we have
\begin{align*}
    k(w,x)+k(y,z) \overset{\eqref{eq:Hyp(U)embed10}}{\leq} 2k(w,x) \overset{\eqref{eq:Hyp(U)embed11},\eqref{eq:Hyp(U)embed12}}{\leq} k(w,y) + k(x,z),
\end{align*}
which verifies \eqref{eq:Hyp(U)embed9}. We may now assume that \eqref{eq:Hyp(U)embed13} holds. Applying \eqref{eq:Hyp(U)embed8} yet again, we have that
\begin{align}
\label{eq:Hyp(U)embed14}    k(y,z) &\leq k(x,y) \hspace{.25in} \text{or} \\
\label{eq:Hyp(U)embed15}    k(y,z) &\leq k(x,z).
\end{align}
In the case where \eqref{eq:Hyp(U)embed14} holds, we have
\begin{align*}
    k(w,x)+k(y,z) \overset{\eqref{eq:Hyp(U)embed13},\eqref{eq:Hyp(U)embed14}}{\leq} k(w,z) + k(x,y),
\end{align*}
and in the case where \eqref{eq:Hyp(U)embed15} holds, we have
\begin{align*}
    k(w,x)+k(y,z) \overset{\eqref{eq:Hyp(U)embed11},\eqref{eq:Hyp(U)embed15}}{\leq} k(w,y) + k(x,z).
\end{align*}
In both cases, \eqref{eq:Hyp(U)embed9} is proved.
\end{proof}

\section{Stochastic Embeddings and Lipschitz Free Spaces}
\label{s:stochastic}
We begin this section by precisely defining random maps and various types of stochastic embeddings. Then in the proceeding subsection, we introduce a space of random Lipschitz functions on a pointed metric space $X$ that naturally maps to $L^1(\Omega;\LF(X))^*$. This space of random Lipschitz functions will be needed to prove Theorem~\ref{thm:stochasticembedLFembed}, which is the main contributory theorem towards Theorem~\ref{thm:stochasticembedtreeLFisoL1} (which proves Theorem~\ref{thm:A}). These theorems are proved in the last subsection.

\subsection{Random Maps and Stochastic Embeddings}
\label{ss:randommaps}
A \emph{random map} from a pointed set $(X,x_0)$ to a pointed topological space $(Y,y_0)$ is a $Y^X$-valued random variable that is \emph{pointwise measurable}: a measurable function from some probability space $(\Omega,\F,\P)$ to the topological space $Y^X$ (equipped with the product topology). Sometimes we will denote a random map by $\{\phi_\omega: X \to Y\}_{\omega\in\Omega}$, and other times we will suppress the underlying probabilistic data and just write $\phi: X \to Y$. There are two additional properties we require as part of the definition of random maps $\phi$. The first is that $\phi$ is \emph{pointwise essentially-separably-valued}: for every $x \in X$, there exists a separable subset $S \sbs Y$ and a $\P$-null net $N$ such that $\phi_\omega(x) \in S$ for every $\omega\in\Omega\setminus N$. This crucial technical requirement ensures Bochner measurability of associated Banach-space-valued functions defined in the next paragraph, and it is also needed for other basic constructions in this section. Of course, this condition is automatic if $Y$ is itself separable, which will often be the case. The second is that $\phi$ is \emph{almost surely basepoint-preserving}, meaning that for $\P$-almost every $\omega\in\Omega$, it holds that $\phi_\omega(x_0) = y_0$. This condition will be necessary to define an induced bounded linear map $T_\phi$ on $\LF(X)$ when $X$ and $Y$ are pointed metric spaces. Whenever $W,X$ are pointed sets, $Y,Z$ are pointed topological spaces, $f_1: W \to X$ is a basepoint-preserving map, $f_2: Y \to Z$ is a basepoint-preserving continuous map, and $\{\phi_\omega: X \to Y\}_{\omega\in\Omega}$ is a random map, it holds that $\{f_2 \circ \phi_\omega \circ f_1: W \to Z\}_{\omega\in\Omega}$ is a random map.

Let $\{\phi_\omega: X \to Y\}_{\omega\in\Omega}$ be a random map between pointed metric spaces $(X,d_X)$, $(Y,d_Y)$ and $D< \infty$. The random map is \emph{$D$-Lipschitz in expectation} if for every $x,y\in X$, the inequality $\E_\omega[d_Y(\phi_\omega(x),\phi_\omega(y))] \leq D d_X(x,y)$ holds. In this case, there is an associated basepoint-preserving $D$-Lipschitz map $T_\phi: X \to L^1(\Omega;\LF(Y))$ defined by $T_\phi(x) := \{\delta_{\phi_\omega(x)}\}_{\omega\in\Omega}$. The pointwise measurability and essential-separable-valuedness of $\phi$, together with Pettis' measurability theorem \cite[Theorem~1.2]{DU}, ensure that $T_\phi(x)$ is Bochner measurable. We also write $T_\phi$ to denote the induced $D$-bounded linear map $\LF(X) \to L^1(\Omega;\LF(Y))$. The random map is a \emph{stochastic biLipschitz embedding of distortion $D$ and scaling factor $s\in (0,\infty)$} if $\phi$ is $sD$-Lipschitz in expectation and, for all $x,y \in X$, it holds that $d_Y(\phi_\omega(x),\phi_\omega(y)) \geq sd_X(x,y)$ for $\P$-almost every $\omega\in\Omega$. We refer to this latter property as \emph{almost-sure noncontractivity}. We will see in Theorem~\ref{thm:stochasticembedLFembed} that this property implies that the induced map $T_\phi$ is a $D$-isomorphic embedding. If there are $D,D',K<\infty$ such that, for every $x,y \in X$, we have $\E_\omega[d_Y(\phi_\omega(x),\phi_\omega(y))] \leq D d_X(x,y) + K$ and $d_Y(\phi_\omega(x),\phi_\omega(y)) \geq (D')^{-1}d_X(x,y)-K$ for $\P$-a.e. $\omega\in\Omega$, then $\phi$ is a \emph{stochastic rough biLipschitz embedding}. If we can choose $D=D'=1$, then $\phi$ is a \emph{stochastic rough isometric embedding}.

Whenever $X$ is countable and $\{\phi_\omega: X \to Y\}_{\omega\in\Omega}$ is a stochastic rough biLipschitz embedding, then by passing to a $\P$-full-measure subset, we may assume that there exists a separable subset $S \sbs Y$ such that, for \emph{every} $\omega\in\Omega$, it holds that $\phi_\omega(X) \sbs S$, $\phi_\omega$ is basepoint-preserving, and $d_Y(\phi_\omega(x),\phi_\omega(y)) \geq (D')^{-1}d_X(x,y) - K$ for \emph{every} $x,y \in X$. We will use this simple reduction in the proof of Lemma~\ref{lem:roughtobiLip}.

Whenever $W,X,Y,Z$ are pointed metric spaces, $f_1: W \to X$ is a basepoint-preserving biLipschitz embedding, $f_2: Y \to Z$ is a basepoint-preserving biLipschitz embedding, and $\{\phi_\omega: X \to Y\}_{\omega\in\Omega}$ is a stochastic biLipschitz embedding, it holds that $\{f_2 \circ \phi_\omega \circ f_1: W \to Z\}_{\omega\in\Omega}$ is a stochastic biLipschitz embedding. On the other hand, the situation is more delicate for rough biLipschitz embeddings. Suppose instead that $f_1,f_2$ are rough biLipschitz embeddings and that $\{\phi_\omega\}_{\omega\in\Omega}$ is a stochastic rough biLipschitz embedding. We still have that $\{\phi_\omega \circ f_1\}_{\omega\in\Omega}$ is a stochastic rough biLipschitz embedding. However, $\{f_2 \circ \phi_\omega\}_{\omega\in\Omega}$ need not be a stochastic rough biLipschitz embedding. This is due to the fact that rough biLipschitz embeddings are generally not measurable, nor do they map separable subsets to separable subsets, so $\{f_2 \circ \phi_\omega\}_{\omega\in\Omega}$ can fail to be a random map. However, these are the only obstructions; if $f_2$ is measurable and maps separable subsets to separable subsets, then $\{f_2 \circ \phi_\omega\}_{\omega\in\Omega}$ is indeed a random map and thus a stochastic rough biLipschitz embedding. We will be using these facts throughout the next few sections.

\subsection{Spaces of Random Lipschitz Functions}
Let $X$ be a pointed metric space and $(\Omega,\F,\P)$ a probability space. We write $L^\infty_{pwm}(\Omega,\F,\P;\Lip_0(X))$ to denote the vector space of random maps $\{f_\omega: X \to \R\}_{\omega\in\Omega}$ having finite seminorm
\begin{equation*}
    \|f\|_{L^\infty_{pwm}} := \sup_{x\neq y\in X} \esssup_{\omega\in\Omega} \dfrac{|f_\omega(y)-f_\omega(x)|}{d_X(x,y)},
\end{equation*}
where the essential supremum is with respect to the $\P$-null sets. The subscript $pwm$ stands for ``pointwise measurable". Notice that the seminorm is also given by
\begin{equation} \label{eq:pwmnormdef} 
    \|f\|_{L^\infty_{pwm}} = \sup_{I\in\mathcal{P}(X)_{\leq\aleph_0}}\esssup_{\omega\in\Omega} \Lip(f_\omega\big|_{I}),
\end{equation}
where the supremum is over all countable subsets $I\sbs X$. We will often abbreviate notation and simply write $L^\infty_{pwm}(\Omega;\Lip_0(X))$. We could create an honest normed vector space by identifying two functions whose difference has norm 0, but this offers no advantage as we will work with $L^\infty_{pwm}(\Omega,\F,\P;\Lip_0(X))$ only through its image in $L^1(\Omega;\LF(X))^*$, which we explain beginning in the next paragraph.

There is a natural linear isometric map (not claimed to be surjective) from the seminormed space $L^\infty_{pwm}(\Omega;\Lip_0(X))$ into the normed space $L^1(\Omega;\LF(X))^*$ given by the pairing
\begin{equation*}
    f(\mu) := \E_\omega[f_\omega(\mu_\omega)].
\end{equation*}
The definition of this expectation requires further explanation. Let $f = \{f_\omega\}_{\omega\in\Omega} \in L^\infty_{pwm}(\Omega;\Lip_0(X))$ with $\|f\|_{L^\infty_{pwm}} \leq 1$ and $\mu = \{\mu_\omega\}_{\omega\in\Omega} \in L^1(\Omega;\LF(X))$. First of all, since $\{\mu_\omega\}_{\omega\in\Omega}$ is Bochner-measurable, there exists a countable subset $I\sbs X$ (which we may assume contains the basepoint) such that $\P_\omega(\mu_\omega \in \LF(I)) = 1$. By \eqref{eq:pwmnormdef}, $f_\omega\big|_I$ is 1-Lipschitz for $\P$-a.e. $\omega\in\Omega$. Therefore, for $\P$-a.e. $\omega$, the pairing $f_\omega(\mu_\omega) = f_\omega\big|_I(\mu_\omega)$ is the usual duality pairing $\Lip_0(I) \to \LF(I)^*$, and we have that $|f_\omega(\mu_\omega)| \leq \|\mu_\omega\|_{\LF}$. Then since the map $\omega \mapsto f_\omega(\mu_\omega)$ is measurable (it is the composition of the measurable map $\omega \mapsto (f_\omega,\mu_\omega)$ with the jointly continuous map $(g,\nu) \mapsto g(\nu)$, where the intermediate space is $B_{\Lip_0(X)} \times \LF(X)$ equipped with the weak$^*$ $\times$ norm topology), the expectation $\E_\omega[f_\omega(\mu_\omega)]$ is a well-defined real number with $|\E_\omega[f_\omega(\mu_\omega)]| \leq \E_\omega[\|\mu_\omega\|_{\LF}] = \|\mu\|_{L^1(\LF)}$. This inequality demonstrates contractivity of the map.

To see that the map is isometric, let $f \in L^\infty_{pwm}(\Omega;\Lip_0(X))$ with $\|f\|_{L^\infty_{pwm}} > 1$. Then we can find $x\neq y \in X$ and $A \in \F$ with $\P(A) > 0$ such that $\dfrac{|f_\omega(y)-f_{\omega}(x)|}{d_X(x,y)} > 1$ for all $\omega\in A$. Hence we can find $B \in \F$ with $B \sbs A$, $\P(B) > 0$, and either $f_\omega(y)-f_\omega(x) > d_X(x,y)$ for all $\omega\in B$ or $f_\omega(y)-f_\omega(x) < -d_X(x,y)$ for all $\omega\in B$. Then $\mu := \dfrac{\delta_y-\delta_x}{d_X(x,y)}\dfrac{\one_{B}}{\P(B)}$ is a norm-1 element of $L^1(\Omega;\LF(X))$ with $|f(\mu)| > 1$. Throughout this section, we will always consider any function $\{f_\omega\}_{\omega\in\Omega} \in L^\infty_{pwm}(\Omega;\Lip_0(X))$ as an element of $L^1(\Omega;\LF(X))^*$ under this pairing. There are natural questions that arise concerning the closedness or density of the image of $L^\infty_{pwm}(\Omega;\Lip_0(X))$ in $L^1(\Omega;\LF(X))^*$ with respect to different topologies, but we do not pursue them here as they are unnecessary for the focus of the article.

\subsection{Stochastic Embeddings into $\R$-Trees and Isomorphisms to $L^1$-Spaces}
In this subsection we work towards our base result that stochastic biLipschitz embeddability of a pointed metric space into an $\R$-tree implies that its free space is isomorphic to an $L^1$-space (Theorem~\ref{thm:stochasticembedtreeLFisoL1}). Following that theorem, there is one final lemma of the section (Lemma~\ref{lem:roughtobiLip}) that will be used to prove Theorem~\ref{thm:C} (via Theorem~\ref{thm:Hypstochasticums}). This lemma asserts that a stochastic rough biLipschitz embedding into an $\R$-tree can be upgraded to a stochastic biLipschitz embedding into an $\R$-tree when the domain is countable and uniformly discrete. We begin with a general result, and then successively specialize to cases where the target space has additional structure.

\begin{theorem} \label{thm:stochasticembedLFembed}
Let $X,Y$ be pointed metric spaces and $\{\phi_\omega: X \to Y\}_{\omega\in\Omega}$ a stochastic biLipschitz embedding of distortion $D<\infty$. Let $K,L,C<\infty$. Then the following statements are true.
\begin{enumerate}
    \item[(1)] The induced linear map $T_\phi: \LF(X) \to L^1(\Omega;\LF(Y))$ is a $D$-isomorphic embedding.
    \item[(2)] If $Y$ has the $K$-Lipschitz-bounded, $L$-uniform-bounded linear extension property and $X$ is compact, then there is a $K$-bounded linear map $R: L^1(\Omega;\LF(Y)) \to \LF(X)$ retracting $T_\phi$.
    \item[(3)] If $Y$ has the $K$-Lipschitz-bounded linear extension property and $\LF(X)$ is $C$-complemented in $\LF(X)^{**}$, then there is a $CK$-bounded linear map \\ $R: L^1(\Omega;\LF(Y)) \to \LF(X)$ retracting $T_\phi$.
\end{enumerate}
\end{theorem}

\begin{proof}
Write $d_X,x_0$ for the metric and basepoint on $X$, and $d_Y,y_0$ for the metric and basepoint on $Y$. By rescaling\footnote{Throughout the article, by \emph{rescaling} a metric $d$ on a set $X$, we mean equipping $X$ with the new metric $td$ for some $t\in(0,\infty)$.} $d_Y$, and observing that Lipschitz free spaces over rescaled metrics are linearly isometric, we may assume that the scaling factor $s$ in the definition of stochastic biLipschitz embedding equals 1. Let $(\Omega,\F,\P)$ denote the probability space underlying the stochastic biLipschitz embedding $\{\phi_\omega\}_{\omega\in\Omega}$. The three statements of the theorem will be proved simultaneously in the following way: we will construct a set-theoretic map $E_\infty: \Lip_0(X) \to L^1(\Omega;\LF(Y))^*$ (thought of as a random extension operator) splitting $T_\phi^*$ (i.e., $T_\phi^* \circ E_\infty = id_{\Lip_0(X)}$). In the general case with no additional assumptions on $Y$ or $X$, the map $E_\infty$ will not be linear but will satisfy $\|E_\infty(f)\|_{L^1(\LF)^*} \leq \Lip(f)$. In the proof, we will refer to this as ``case (1)", and we refer to inequalities of the form $\|E(h)\| \leq L\|h\|$ by saying that ``$E$ is $L$-bounded", even if $E$ is only a set-theoretic map between normed spaces and not linear. Statement (1) of the theorem then follows. In the case where $Y$ is assumed to have the $K$-Lipschitz-bounded, $L$-uniform-bounded linear extension property and $X$ is compact, the map $E_\infty$ will be linear, $K$-bounded, and weak$^*$-weak$^*$ continuous. We will refer to this as ``case (2)" in the proof. In this case, the preadjoint $(E_\infty)_*$ serves as the retract $R$. In the case where $Y$ is assumed to have the $K$-Lipschitz-bounded linear extension property, the map $E_\infty$ will be linear and $K$-bounded. We will refer to this as ``case (3)" in the proof. In this case, we then use the assumption that $\LF(X)$ is $C$-complemented in $\LF(X)^{**}$ to construct the retract $R$. We now proceed with the construction of the random extension operator $E_\infty: \Lip_0(X) \to L^1(\Omega;\LF(Y))^*$.

Let $\mathcal{P}_{<\infty}(X)$ denote the directed set of finite subsets of $X$ containing the basepoint, ordered by inclusion. Let $\mathscr{A}$ denote the directed set of $\sigma$-subalgebras of $\F$, ordered by inclusion, consisting of those $\A$ generated by some countable partition of $\Omega$. We consider $\mathcal{P}_{<\infty}(X) \times \mathscr{A}$ as a directed set equipped with its product preorder. Let $(F,\A) \in \mathcal{P}_{<\infty}(X) \times \mathscr{A}$, and suppose that $\A$ is generated by the countable partition $\{A_i\}_{i=1}^\infty$ of $\Omega$ (so $\{A_i\}_{i=1}^\infty$ is the set of atoms of the purely atomic $\sigma$-algebra $\A$). For each atom $A_i \in \A$ with $\P(A_i) > 0$, choose $\omega_{F,i} \in A_i$ such that $\phi_{\omega_{F,i}}(x_0) = y_0$ and
\begin{equation} \label{eq:noncontractiveonF}
    d_Y(\phi_{\omega_{F,i}}(x),\phi_{\omega_{F,i}}(y)) \geq d_X(x,y)
\end{equation}
for every $x,y \in F$. We can make such a choice by the almost-sure basepoint preservation and almost-sure noncontractivity of $\phi$ (recalling that $s=1$), and the fact that $F$ is finite (note that if $F$ was uncountable, we may not be able to choose any $\omega_{F,i} \in A_i$ with this property). Let $E_{F,i}: \Lip_0(\phi_{\omega_{F,i}}(F)) \to \Lip_0(Y)$ be an extension operator (noting that each subset $\phi_{\omega_{F,i}}(F)$ contains the basepoint $y_0$). In case (1) with no assumption on $Y$, $E_{F,i}$ is constructed using the McShane extension theorem, and it is not linear but is 1-bounded. In cases (2)-(3) where $Y$ is assumed to have the $K$-Lipschitz-bounded linear extension property, $E_{F,i}$ is chosen to be $K$-bounded and linear. In case (2), we also choose $E_{F,i}$ to be $L$-uniform-bounded.

Define a map $E_{(F,\A)}: \Lip_0(X) \to L_{pwm}^\infty(\Omega,\A;\Lip_0(Y))$ by
\begin{equation*}
    E_{(F,\A)}(f)_\omega := \sum_{i=1}^\infty E_{F,i}(f \circ \phi_{\omega_{F,i}}^{-1})\one_{A_i}(\omega),
\end{equation*}
where we consider $\phi_{\omega_{F,i}}^{-1}$ as a map from $\phi_{\omega_{F,i}}(F)$ to $X$, which is 1-Lipschitz by \eqref{eq:noncontractiveonF}. Then $E_{(F,\A)}$ is 1-bounded in case (1) and $K$-bounded and linear in cases (2)-(3). After composing with the natural pairing $L_{pwm}^\infty(\Omega,\A;\Lip_0(Y)) \to L^1(\Omega,\F;\LF(Y))^*$, we may consider $E_{(F,\A)}$ as having values in the dual space $L^1(\Omega,\F;\LF(Y))^*$, and define the map $E_\infty: \Lip_0(X) \to L^1(\Omega,\F;\LF(Y))^*$ by
\begin{equation*}
    E_{\infty}(f) := \U\text{-}w^*\lim_{(F,\A)\to\infty} E_{(F,\A)}(f),
\end{equation*}
where $\U$ is any fixed eventuality ultrafilter\footnote{The \emph{eventuality filter} on a directed set $A$ is the collection of all subsets $B \sbs A$ such that $B \sps A_{\geq\beta}$ for some $\beta \in A$, where $A_{\geq\beta} := \{\alpha \in A: \alpha \geq \beta\}$. An \emph{eventuality ultrafilter} on $A$ is an ultrafilter on $A$ containing the eventuality filter.} on the directed set $\mathcal{P}_{<\infty}(X)\times\mathscr{A}$, and the limit is the $\U$-ultralimit with respect to the (compact) weak$^*$ topology on the ball of radius $K\Lip(f)$ in $L^1(\Omega,\F;\LF(Y))^*$. In case (1), $E_{\infty}$ is 1-bounded and (generally) nonlinear, and in cases (2)-(3), $E_{\infty}$ is linear and $K$-bounded.

Our first order of business is to prove that $E_\infty$ splits $T_\phi^*$. Fix $f \in \Lip_0(X)$ and $x \in X$. Let $\eps>0$ be arbitrary. Since $\phi$ is pointwise essentially-separably-valued, there exists a separable subset $S \sbs Y$ such that $\phi_\omega(x) \in S$ for $\P$-a.e. $\omega\in\Omega$. Hence, we may find a countable, Borel measurable partition $\{S_i\}_{i=1}^\infty$ of $S$ such that $\diam(S_i)<\frac{\eps}{K\Lip(f)}$ for every $i$. Then since the map $\omega \to \phi_\omega(x)$ is measurable and takes values in $S$ almost surely, the preimage $\{G_i\}_{i=1}^\infty$ of $\{S_i\}_{i=1}^\infty$ under this map is a countable, $\F$-measurable partition of $\Omega \setminus N$ for some $\P$-null set $N\in\F$. Let $\G$ be the $\sigma$-subalgebra of $\F$ generated by $\{G_i,N\}_{i=1}^\infty$, so that $\G \in \mathscr{A}$. Then, for any $\A \in \mathscr{A}$ with $\A \sps \G$, for any atom $A_i \in \A$ with $\P(A_i)>0$, and for any $\omega,\omega' \in A_i$, it holds that
\begin{equation*}
    d_Y(\phi_\omega(x),\phi_{\omega'}(x)) < \frac{\eps}{K\Lip(f)}.
\end{equation*}
Then we get the estimate
\begin{align*}
    |T^*_\phi(E_\infty(f))&(x) - f(x)| \\
    &= |E_\infty(f)(T_\phi(x)) - f(x)| \\
    &= \left|\left[\U\text{-}w^*\lim_{(F,\A)\to\infty} E_{(F,\A)}(f)\right](T_\phi(x)) - f(x)\right| \\
    &= \U\text{-}\lim_{(F,\A)\to\infty} |E_{(F,\A)}(f)(T_\phi(x)) - f(x)| \\
\end{align*}
\begin{align*}
    &= \U\text{-}\lim_{(F,\A)\to\infty} \left|\sum_{i=1}^\infty \int_{A_i} E_{F,i}(f \circ \phi_{\omega_{F,i}}^{-1})(\phi_\omega(x)) \mathrm{d}\P(\omega) - f(x)\right| \\
    &\leq \sup_{F\ni x, \A\sps\G} \sum_{i=1}^\infty \int_{A_i} |E_{F,i}(f \circ \phi_{\omega_{F,i}}^{-1})(\phi_\omega(x)) - E_{F,i}(f \circ \phi_{\omega_{F,i}}^{-1})(\phi_{\omega_{F,i}}(x))| \mathrm{d}\P(\omega) \\
    &\hspace{.3in} + \sum_{i=1}^\infty \int_{A_i} | E_{F,i}(f \circ \phi_{\omega_{F,i}}^{-1})(\phi_{\omega_{F,i}}(x)) - f(x)| \mathrm{d}\P(\omega) \\
    &\leq \sup_{F\ni x,\A\sps\G} \sum_{i=1}^\infty \int_{A_i} K\Lip(f)\frac{\eps}{K\Lip(f)} \mathrm{d}\P(\omega) + \sum_{i=1}^\infty \int_{A_i} | (f \circ \phi_{\omega_{F,i}}^{-1})(\phi_{\omega_{F,i}}(x)) - f(x)| \mathrm{d}\P(\omega) \\
    &= \eps.
\end{align*}
Since $\eps>0$ was arbitrary, this proves $T^*_\phi(E_\infty(f))(x) = f(x)$. Since $f \in \Lip_0(X)$ and $x \in X$ were arbitrary, this proves $T_\phi^* \circ E_\infty = id_{\Lip_0(X)}$. The proof of statement (1) is now complete, and from here on, we work in cases (2)-(3) where $E_\infty$ is linear and $K$-bounded.

We focus on case (2) first, using the $L$-uniform-boundedness of $E_{F,i}$ and compactness of $X$ to prove that $E_\infty$ is weak$^*$-weak$^*$ continuous. Let $f_{\alpha} \in B_{\Lip_0(X)}$ be a net converging pointwise to $f \in B_{\Lip_0(X)}$. Then since $X$ is compact, $f_\alpha \overset{\alpha\to\infty}{\to} f$ uniformly. We need to prove that, for every $\mu = \{\mu_\omega\}_{\omega\in\Omega} \in L^1(\Omega;\LF(Y))$, $E_\infty(f_\alpha)(\mu) \overset{\alpha\to\infty}{\to} E_\infty(f)(\mu)$. By replacing $f_\alpha$ with $f_\alpha-f$, we may assume that $f=0$. Since the linear functionals $E_\infty(f_\alpha)$ are uniformly (in $\alpha$) bounded, it suffices to prove the desired convergence for all $\mu$ in a chosen subset of $L^1(\Omega;\LF(Y))$ whose linear span is dense. Towards this end, we consider $\mu$ of the form $\delta_y\one_B$, where $y \in Y$ and $B\in\F$. Let $\eps>0$ be arbitrary. Since $f_\alpha \to 0$ uniformly, we may choose $\alpha_0$ large enough so that for all $\alpha \geq \alpha_0$, we have $\|f_\alpha\|_\infty \leq \frac{\eps}{L}$. Then, for all $\alpha \geq \alpha_0$, we have
\begin{align*}
    \left|E_\infty(f_\alpha)(\delta_y\one_B)\right| &= \left|\left[\U\text{-}w^*\lim_{(F,\A)\to\infty} E_{(F,\A)}(f_\alpha)\right](\delta_y\one_B)\right| \\
    &= \U\text{-}\lim_{(F,\A)\to\infty} |E_{(F,\A)}(f_\alpha)(\delta_y\one_B)| \\
    &\leq \sup_{(F,\A)} |E_{(F,\A)}(f_\alpha)(\delta_y\one_B)| \\
    &= \sup_{(F,\A)} \left|\sum_{i=1}^\infty E_{F,i}(f_\alpha \circ \phi_{\omega_{F,i}}^{-1})(y) \P(A_i\cap B)\right| \\
    &\leq \sup_{(F,\A)} \sum_{i=1}^\infty \|E_{F,i}(f_\alpha \circ \phi_{\omega_{F,i}}^{-1})\|_\infty \P(A_i\cap B) \\ 
    &\leq \sup_{(F,\A)} \sum_{i=1}^\infty L\|f_\alpha\|_\infty \P(A_i\cap B) \\
    &\leq \eps\P(B) \leq \eps.
\end{align*}
Since $\eps > 0$ was arbitrary, this shows $E_\infty(f_\alpha)(\delta_y\one_B) \overset{\alpha\to\infty}{\to} 0$, completing the proof of weak$^*$-weak$^*$ continuity. Then the preadjoint $(E_\infty)_*$ is $K$-bounded and satisfies
\begin{equation*}
    ((E_\infty)_* \circ T_\phi)^* = T_\phi^* \circ E_\infty = id_{\Lip_0(X)} = id_{\LF(X)}^*,
\end{equation*}
which implies $(E_\infty)_* \circ T_\phi = id_{\LF(X)}$. This completes the proof of statement (2).

We now find ourselves in case (3). In this final case, the lack of uniform-boundedness of $E_{F,i}$ prevents us from being able to prove weak$^*$-weak$^*$ continuity of $E_\infty$, and hence we cannot use a preadjoint $(E_\infty)_*$ to retract $T_\phi$. Naturally, the assumption that $\LF(X)$ is complemented in $\LF(X)^{**}$ provides an alternate route. Define the $CK$-bounded linear map $R: L^1(\Omega,\F;\LF(Y)) \to \LF(X)$ by $R := P \circ E_\infty^* \circ J_Y$, where $P: \LF(X)^{**} \to \LF(X)$ is a $C$-bounded linear map retracting the canonical embedding $J_X: \LF(X) \to \LF(X)^{**} = \Lip_0(X)^*$, and $J_Y: L^1(\Omega,\F;\LF(Y)) \to L^1(\Omega,\F;\LF(Y))^{**}$ is the canonical embedding. It is then immediate to verify that $R$ retracts $T_\phi$. To see this, let us introduce the notation $J_{X^{**}}: \LF(X)^{**} \to \LF(X)^{****}$ and $J_{Y^{**}}: L^1(\Omega,\F;\LF(Y))^{**} \to L^1(\Omega,\F;\LF(Y))^{****}$ for the canonical embeddings, and, letting $J_{U}: U \to U^{**}$ denote the canonical embedding for a general Banach space $U$, recall that for any Banach spaces $V,W$ and bounded linear map $S: V \to W$, we have $J_{V}^{**} \circ J_V = J_{V^{**}} \circ J_V$ and $S^{**} \circ J_V = J_W \circ S$. Using these facts together with the definition of $R$ and the identity $T_\phi^* \circ E_\infty = id_{\LF(X)^*}$ established above, we get
\begin{align*}
    (R \circ T_\phi)^{**} \circ J_X &= P^{**} \circ E_\infty^{***} \circ J_Y^{**} \circ T_\phi^{**} \circ J_X \\
    &= P^{**} \circ E_\infty^{***} \circ J_Y^{**} \circ J_Y \circ T_\phi \\
    &= P^{**} \circ E_\infty^{***} \circ J_{Y^{**}} \circ J_Y \circ T_\phi \\
    &= P^{**} \circ E_\infty^{***} \circ J_{Y^{**}} \circ T_\phi^{**} \circ J_X \\
    &= J_X \circ P \circ E_\infty^{*} \circ T_\phi^{**} \circ J_X \\
    &= J_X \circ P \circ J_X \\
    &= J_X,
\end{align*}
which implies $R \circ T_\phi = id_{\LF(X)}$.
\end{proof}

We now specialize to the case where the target space has finite Nagata dimension and apply Corollary~\ref{cor:Nagatauniformbnddextension}.

\begin{corollary} \label{cor:stochasticembedLFcompembed}
Let $X$ be a pointed metric space stochastic biLipschitzly embedding with distortion $D<\infty$ into a pointed metric space $Y$ of finite Nagata dimension (say, dimension $n\in\N$ with constant $\gamma<\infty$). If $X$ is compact or if $\LF(X)$ is $C$-complemented in $\LF(X)^{**}$, then $\LF(X)$ is $D$-isomorphic to a $C'$-complemented subspace of $L^1(\Omega;\LF(Y))$, where $C'$ depends only on $D,C,n,\gamma$, and $\Omega$ is the probability space underlying the stochastic biLipschitz embedding.
\end{corollary}

\begin{proof}
Assume that $X$ is compact or that $\LF(X)$ is $C$-complemented in $\LF(X)^{**}$ for some $C<\infty$. By Corollary~\ref{cor:Nagatauniformbnddextension}, $Y$ has the $K$-Lipschitz bounded, 1-uniform bounded linear extension property, where $K$ depends only on $n,\gamma$. Then the conclusion follows from Theorem~\ref{thm:stochasticembedLFembed}.
\end{proof}

We now reach our base result, which essentially follows from Corollary~\ref{cor:stochasticembedLFcompembed}, the fact that $\R$-trees have finite Nagata dimension, and a widely used decomposition lemma of Kalton \cite[Lemma~4.2]{Kalton}. This lemma allows one, in many cases, to restrict to bounded subsets of a metric space to prove isomorphic properties of its free space.

\begin{theorem} \label{thm:stochasticembedtreeLFisoL1}
Let $X$ be an infinite, separable, pointed metric space and $D<\infty$. If $X$ stochastic biLipschitzly embeds with distortion $D$ into a pointed $\R$-tree, and if $X$ is proper or $\LF(X)$ is complemented in $\LF(X)^{**}$, then $\LF(X) \approx L^1$ or $\LF(X) \approx \ell^1$, with the latter isomorphism holding if and only if the completion of $X$ is purely 1-unrectifiable.
\end{theorem}

\begin{proof}
Assume that $X$ stochastic biLipschitzly embeds into a pointed $\R$-tree with distortion $D<\infty$ and that $X$ is proper or that $\LF(X)$ is complemented in $\LF(X)^{**}$. We will prove that $\LF(X)$ is isomorphic to a complemented subspace of an $L^1$-space, and then the conclusion follows from Theorem~\ref{thm:freespacecomplementedinL1}.

First we work in the case where $\LF(X)$ is complemented in $\LF(X)^{**}$. Since $\R$-trees have Nagata dimension 1 with constant $\gamma \leq 5$ and have free spaces linearly isometric to $L^1$-spaces (see $\S$\ref{ss:umsRtrees} for each of these facts), Corollary~\ref{cor:stochasticembedLFcompembed} implies that $\LF(X)$ is isomorphic to a complemented subspace of an $L^1$-space.

Now we work in the other case where $X$ is proper. By \cite[Lemma~4.2]{Kalton}, there is a sequence $\{B_n\}_{n\in\Z}$ of closed, bounded subsets of $X$, containing the basepoint, such that $\LF(X) \cembed \oplus^1_{n\in\Z} \LF(B_n)$, and thus it suffices to prove that there exists $C'<\infty$ such that $\LF(B)$ is $C'$-isomorphic to a $C'$-complemented subspace of an $L^1$-space for every closed, bounded $B \sbs X$ containing the basepoint. Let $B \sbs X$ be closed, bounded, and basepoint-containing. Then $B$ is compact and stochastic biLipschitzly embeds into a pointed $\R$-tree with distortion $D$. Then we use  Corollary~\ref{cor:stochasticembedLFcompembed} again to obtain that $\LF(B)$ is $D$-isomorphic to a $C'$-complemented subspace of an $L^1$-space, where $C'$ depends only on $D$.
\end{proof}

\begin{remark}
It is plausible that the conclusion of Theorem~\ref{thm:stochasticembedtreeLFisoL1} could hold without assuming that $X$ is proper or that $\LF(X)$ is complemented in $\LF(X)^{**}$. We expect that a proof of such a general result would need to engage more directly with the geometry of $\R$-trees and would be decidedly more technical than the proof of Theorem~\ref{thm:stochasticembedtreeLFisoL1} we provided.
\end{remark}

The final result of this section states that if $X$ is countable and uniformly discrete, then a stochastic rough biLipschitz embedding of $X$ into an $\R$-tree can be upgraded to a stochastic biLipschitz embedding. This lemma will be used in the proof of Theorem~\ref{thm:Hypstochasticums}.

\begin{lemma} \label{lem:roughtobiLip}
Let $(X,d_X,p)$ be a countable, uniformly discrete, pointed metric space. If $X$ stochastic rough biLipschitzly embeds into a pointed $\R$-tree, then $X$ stochastic biLipschitzly embeds into a pointed $\R$-tree. Moreover, the distortion of the stochastic biLipschitz embedding depends only on the implicit constants of the stochastic rough biLipschitz embedding and the separation constant of $X$.
\end{lemma}

\begin{proof}
Let $\{\phi_\omega: X \to Y\}_{\omega\in\Omega}$ be a random map into a separable subset $Y$ of a pointed $\R$-tree $(T,d_T,q)$ and $D,D',K<\infty$ such that, for every $\omega\in\Omega$ and $x,y \in X$, $d_T(\phi_\omega(x),\phi_\omega(y))$ $\geq (D')^{-1}d_X(x,y) - K$, $\phi_\omega(p) = q$, and $\E_\omega[d_T(\phi_\omega(x),\phi_\omega(y))] \leq Dd_X(x,y) + K$ (we can ask for these things to happen for all $\omega\in\Omega$ due to countability of $X$). By replacing $d_T$ with $(D')^{-1}d_T$, and noting that rescalings of $\R$-trees are still $\R$-trees, we may assume that $D'=1$. Furthermore, by replacing $T$ with the convex hull of $Y$, we may assume that $T$ is separable.

Let $N_X \sbs X$ be a maximal $2K$-separated subset of $X$ containing $p$ and $N_T$ a maximal $\frac{K}{6}$-separated subset of $T$ containing $q$. Then $N_X$ and $N_T$ are countable since $X,T$ are separable. Equip each of these sets with a well-order coming from any bijection with $\N$, and then define projections $\pi_X: X \to N_X$ and $\pi_T: T \to N_T$ by $\pi_X(x) = \min\{n\in N_X: d_X(n,x) \leq 2\dist(N_X,x)\}$ and $\pi_T(t) = \min\{n\in N_T: d_T(n,t) \leq 2\dist(N_T,t)\}$. By maximality of $N_X$, it holds that
\begin{equation} \label{eq:roughtobiLip1}
    d_X(\pi_X(x),x) \leq 4K
\end{equation}
for every $x \in X$, and hence
\begin{align}
\label{eq:roughtobiLip2}    d_X(x,y) \leq 8K && \text{ if } \pi_X(x) = \pi_X(y), \\
\label{eq:roughtobiLip3}    d_X(x,y) \leq 8K + d_X(\pi_X(x),\pi_X(y)) &\leq 5d_X(\pi_X(x),\pi_X(y)) & \text{ if } \pi_X(x) \neq \pi_X(y).
\end{align}

By the countability of $N_T$, it is easily seen that $\pi_T$ is measurable. The maximality of $N_T$ implies that $d_T(\pi_T(t),t) \leq \frac{K}{3}$ for every $t \in T$. Using this, together with $2K$-separation of $N_X$, we have for any $n \neq m \in N_X$ and $\omega\in\Omega$,
\begin{align}
\nonumber   d_T(\pi_T(\phi_\omega(m)),\pi_T(\phi_\omega(n))) &\geq d_T(\phi_\omega(m),\phi_\omega(n)) - \frac{2K}{3} \\
\nonumber    &\geq d_X(m,n) - K - \frac{2K}{3} \\
\label{eq:roughtobiLip4}    &\geq \frac{1}{6}d_X(m,n).
\end{align}
Additionally,
\begin{align}
\nonumber    \E_\omega[d_T(\pi_T(\phi_\omega(m)),\pi_T(\phi_\omega(n)))] &\leq \E_\omega[d_T(\phi_\omega(m),\phi_\omega(n))] + \frac{2K}{3} \\
\nonumber    &\leq Dd_X(m,n) + K + \frac{2K}{3} \\
\label{eq:roughtobiLip5}     &\leq \left(D+\frac{5}{6}\right)d_X(m,n).
\end{align}
This shows that $\{(\pi_T \circ \phi_\omega)\big|_{N_X}: N_X \to N_T\}_{\omega\in\Omega}$ is a stochastic biLipschitz embedding. We will next use the uniform discreteness assumption to extend this to all of $X$ by precomposing with $\pi_X$. However, we need to enlarge our target tree. There are many ways to do this -- we prioritize simplicity of the exposition over optimization of the distortion constants involved. Towards this end, we employ a star-type construction.

Define the metric space $(Z,d_Z)$ by
\begin{align*}
    Z &:= X \times N_T, \\
    d_Z((x,m),(y,n)) & := \begin{cases}
        0 & (x,m)=(y,n) \\
        1+d_T(m,n) & (x,m) \neq (y,n)
    \end{cases}.
\end{align*}
One can verify that $Z$ is countable, uniformly discrete, and 0-hyperbolic since $X$ is countable and $N_T$ is countable and 0-hyperbolic. Thus, $Z$ is isometric to a countable, uniformly discrete subset of an $\R$-tree (see $\S$\ref{ss:umsRtrees}). Equip $Z$ with basepoint $(p,q)$. We will show that the random map $\{\psi_\omega: X \to Z\}_{\omega\in\Omega}$ given by
\begin{equation*}
    \psi_\omega(x) := (x,(\pi_T \circ \phi_\omega \circ \pi_X)(x))
\end{equation*}
is a stochastic biLipschitz embedding. First note that $\psi$ is indeed pointwise measurable and essentially-separably-valued since $\pi_T$ is measurable and $Z$ is countable. It is also basepoint-preserving since $p\in N_X$ and $q\in N_T$. Now let $\theta > 0$ such that $d_X(x,y) > \theta$ for every $x\neq y \in X$. Let $x\neq y \in X$. We treat two cases: (i) $\pi_X(x) = \pi_X(y)$ and (ii) $\pi_X(x) \neq \pi_X(y)$. Assume case (i) holds. Then for every $\omega$, we have
\begin{align*}
    d_{Z}(\psi_\omega(x),\psi_\omega(y)) = 1 \overset{\eqref{eq:roughtobiLip2}}{\geq} \frac{1}{8K}d_X(x,y).
\end{align*}
Likewise,
\begin{align*}
    \E_\omega[d_{Z}(\psi_\omega(x),\psi_\omega(y))] = 1 \leq \theta^{-1}d_X(x,y).
\end{align*}
Now assume that we are in case (ii). Then by \eqref{eq:roughtobiLip3}, $d_X(x,y) \leq 5d_X(\pi_X(x),\pi_X(y))$. Thus, for every $\omega\in\Omega$ we have
\begin{align*}
    d_{Z}(\psi_\omega(x),\psi_\omega(y)) &= 1+d_{T}(\pi_T(\phi_\omega(\pi_X(x))),\pi_T(\phi_\omega(\pi_X(y)))) \\
    &\overset{\eqref{eq:roughtobiLip4}}{\geq} 1 + \frac{1}{6}d_X(\pi_X(x),\pi_X(y)) \\
    &\geq 1 + \frac{1}{30}d_X(x,y) \\
    &> \frac{1}{30}d_X(x,y).
\end{align*}
Likewise,
\begin{align*}
    \E_\omega[d_{Z}(\psi_\omega(x),\psi_\omega(y))] &= 1 + \E_\omega[d_{T}(\pi_T(\phi_\omega(\pi_X(x))),\pi_T(\phi_\omega(\pi_X(y))))] \\
    &\overset{\eqref{eq:roughtobiLip5}}{\leq} 1 + \left(D+\frac{5}{6}\right)d_X(\pi_X(x),\pi_X(y)) \\
    &\overset{\eqref{eq:roughtobiLip1}}{\leq} 1 + \left(D+\frac{5}{6}\right)(d_X(x,y) + 8K) \\
    &\leq \left(\theta^{-1}+\left(D+\frac{5}{6}\right)(1+8K\theta^{-1})\right)d_X(x,y).
\end{align*}
\end{proof}

\section{Stochastic Embeddings into Ultrametric Spaces}
\label{s:stochasticums}
The main result of this section is Theorem~\ref{thm:Nagatasnowflakestochastic} (which proves Theorem~\ref{thm:B}), stating the every nontrivial snowflake of a separable, bounded, finite Nagata-dimensional, pointed metric space $X$ admits a stochastic biLipschitz embedding into a pointed ultrametric space. We begin the section by establishing the core geometric tool underlying Theorem~\ref{thm:Nagatasnowflakestochastic}, which is Theorem~\ref{thm:[0,1]snowflakestochastic}. This theorem gives the same conclusion as Theorem~\ref{thm:Nagatasnowflakestochastic}, but for $X = [0,1]$. The construction we use in its proof is based on the proof of \cite[Theorem~8.43]{Weaver}. Indeed, we directly cite two of the intermediate lemmas from that text (\cite[Lemmas~8.39,~8.40]{Weaver}). In subsections $\S$\ref{ss:l0+small}-\ref{ss:threshold}, we bootstrap our way up from $[0,1]$ to general separable, bounded, finite Nagata-dimensional spaces in Theorem~\ref{thm:Nagatasnowflakestochastic}. Finally, in $\S$\ref{ss:pconcave}, we apply Theorem~\ref{thm:Nagatasnowflakestochastic} to distortions of finite Nagata-dimensional spaces by $p$-concave functions, and obtain in Theorem~\ref{thm:pconcave} that the free spaces over such distortions (under some additional assumptions) are isomorphic to $\ell^1$.

Before getting to the proof of Theorem~\ref{thm:[0,1]snowflakestochastic}, we establish in Lemma~\ref{lem:stochasticcompact} a compactness property concerning stochastic biLipschitz embeddings.

\begin{lemma} \label{lem:stochasticcompact}
Let $(X,d_X,p)$ be a pointed metric space, $(K,d_K,q)$ a pointed compact metric space, $D<\infty$, and $s \in (0,\infty)$. If there exists a dense basepoint-containing subset $A \sbs X$ such that every finite basepoint-containing subset of $A$ stochastic biLipschitzly embeds into $K$ with distortion $D$ and scaling factor $s$, then $X$ stochastic biLipschitzly embeds into $K$ with distortion $D$ and scaling factor $s$.
\end{lemma}

\begin{remark} \label{rmk:stochastic}
To prove Lemma~\ref{lem:stochasticcompact}, it is helpful to reorient our view of stochastic biLipschitz embeddings. If $X$ is a set, $Y$ a separable metric space, and $\P$ a probability measure on $Y^X$, then we get a random map $\{\phi_\omega: X \to Y\}_{\omega\in\Omega}$, where the underlying probability space $\Omega$ is $Y^X$ equipped with the Borel $\sigma$-algebra and probability measure $\P$, and $\phi_\omega$ equals $\omega$. Conversely, suppose $\{\phi_\omega: X \to Y\}_{\omega\in\Omega}$ is a random map into a separable space with underlying probability measure $\P$. Then we can push $\P$ forward to a probability measure $\P_\phi$ on $Y^X$ defined by $\P_\phi(E) := \P_\omega(\phi_\omega \in E)$. Therefore, up to equality in distribution, every random map $X \to Y$ can be taken to have underlying measurable space $Y^X$ (equipped with Borel $\sigma$-algebra), and what determines the salient properties of the random map (e.g., if it is a stochastic biLipschitz embedding) is the probability measure on $Y^X$.
\end{remark}

\begin{proof}[Proof of Lemma~\ref{lem:stochasticcompact}]
Assume that there is some dense subset $p\in A \sbs X$ such that every finite subset of $A$, containing $p$, almost-sure basepoint-preservingly stochastic biLipschitzly embeds into $K$ with distortion $D$ and scaling factor $s$. For the following three sentences, fix a finite subset $p\in F \sbs A$. Let $\pi_F: X \to F$ be any nearest neighbor projection map; $\pi_F$ is any map $X \to F$ satisfying $d(x,\pi_F(x)) = \dist(F,x)$. Let $\{\phi^F_\omega: F \to K\}_{\omega\in\Omega}$ be a stochastic biLipschitz embedding of distortion $D$ with underlying probability space $(\Omega_F,\F,\P^F)$, and set $\psi_\omega^F := \phi_\omega^F \circ \pi_F: X \to K$. Let $\widetilde{\P}_F := \P^F_{\psi_F}$ denote the pushforward probability measures on $K^X$ as described in Remark~\ref{rmk:stochastic}. Then by the Riesz representation theorem and weak$^*$-compactness, we can find a probability measure $\widetilde{\P}_\infty$ on $K^X$ that is a weak$^*$-subnet limit of the net of measures $\widetilde{\P}_F$. Let us check that this probability measure yields (in the sense of Remark~\ref{rmk:stochastic}) an almost-sure basepoint-preserving stochastic biLipschitz embedding with distortion $D$ and scaling factor $s$.

First we verify almost-sure basepoint preservation. For every $F \sbs A$ finite with $p \in F$, we have
\begin{align*}
    \int_{K^X} d_K(\phi(p),q)\mathrm{d}\widetilde{\P}_F(\phi) = \int_{\Omega_F}d_K(\psi_\omega^F(p),q)\mathrm{d}\P^F(\omega) = 0.
\end{align*}
Since $F \ni p$ was arbitrary, this implies
\begin{align*}
    \int_{K^X} d_K(\phi(p),q)\mathrm{d}\widetilde{\P}_\infty(\phi) = 0.
\end{align*}
Hence, for $\widetilde{\P}_\infty$-a.e. $\phi \in K^X$, $\phi(p) = q$.

Next we show almost-sure noncontractivity. Fix $x,y \in X$. Let $\Phi: K^X \to [0,\infty)$ be a nonnegative continuous function. Let $\eps > 0$, and choose $a,b \in A$ such that $d_X(x,a),d_X(y,b) < \frac{\eps}{2s\|\Phi\|_\infty}$, which exist since $A \sbs X$ is dense. For every $F \sbs A$ finite with $\{a,b\} \sbs F$, we have
\begin{align*}
    \int_{K^X} (d_K(\phi(x),\phi(y))-sd_X(x,y))&\Phi(\phi)\mathrm{d}\widetilde{\P}_F(\phi) \\
    &= \int_{\Omega_F} (d_K(\psi_\omega^F(x),\psi_\omega^F(y))-sd_X(x,y))\Phi(\psi_\omega^F)\mathrm{d}\P^F(\omega) \\
    &\geq \int_{\Omega_F} (sd_X(\pi_F(x),\pi_F(y))-sd_X(x,y))\Phi(\psi_\omega^F)\mathrm{d}\P^F(\omega) \\
    &\geq \int_{\Omega_F} -\frac{s\eps}{s\|\Phi\|_\infty}\Phi(\psi_\omega^F)\mathrm{d}\P^F(\omega) \geq -\eps.
\end{align*}
Since $F \sps \{a,b\}$ was arbitrary, this implies
\begin{equation*}
    \int_{K^X} (d_K(\phi(x),\phi(y))-sd_X(x,y))\Phi(\phi)\mathrm{d}\widetilde{\P}_\infty(\phi) \geq -\eps.
\end{equation*}
Since $\eps>0$ was arbitrary, this implies
\begin{equation*}
    \int_{K^X} (d_K(\phi(x),\phi(y))-sd_X(x,y))\Phi(\phi)\mathrm{d}\widetilde{\P}_\infty(\phi) \geq 0.
\end{equation*}
Since $\Phi: K^X \to [0,\infty)$ was arbitrary, this implies
\begin{equation*}
    d_K(\phi(x),\phi(y)) \geq sd_X(x,y)
\end{equation*}
for $\widetilde{\P}_\infty$-a.e. $\phi \in K^X$.
This proves almost-sure noncontractivity.

For the Lipschitz in expectation bound, we again let $\eps > 0$ and choose $a,b \in A$ such that $d_X(x,a),d_X(y,b) < \frac{\eps}{2sD}$. We have, for every $F \sbs A$ finite with $\{a,b\} \sbs F$,
\begin{align*}
    \int_{K^X} d_K(\phi(x),\phi(y))\mathrm{d}\widetilde{\P}_F(\phi) &= \int_{\Omega_F} d_K(\phi_\omega^F(\pi_F(x)),\phi_\omega^F(\pi_F(y)))\mathrm{d}\P^F(\omega) \\
    &\leq sDd_X(\pi_F(x),\pi_F(y)) \leq \eps + sDd_X(x,y).
\end{align*}
Since $F \sps \{a,b\}$ was arbitrary, this implies
\begin{equation*}
    \int_{K^X} d_K(\phi(x),\phi(y))\mathrm{d}\widetilde{\P}_\infty(\phi) \leq \eps + sDd_X(x,y).
\end{equation*}
Since $\eps>0$ was arbitrary, this implies
\begin{equation*}
    \int_{K^X} d_K(\phi(x),\phi(y))\mathrm{d}\widetilde{\P}_\infty(\phi) \leq sDd_X(x,y).
\end{equation*}
\end{proof}

The following theorem, stating that snowflakes of the unit interval $[0,1]$ stochastic biLipschitzly embed into ultrametric spaces, is the core geometric tool powering several of the main results of the article (Theorems~\ref{thm:B},~\ref{thm:D},~\ref{thm:E}, via Theorem~\ref{thm:Nagatasnowflakestochastic}). It is essential here that the snowflake exponent is nontrivial (not equal to 1), see Remark~\ref{rmk:[0,1]nembedums}.

\begin{theorem} \label{thm:[0,1]snowflakestochastic}
For every $\alpha \in (0,1)$, the pointed metric space $([0,1],|\cdot|^\alpha)$ stochastic biLipschitzly embeds into a compact, pointed ultrametric space with distortion $D \leq \frac{2^{5\alpha+1}}{(2^\alpha-1)(2-2^\alpha)}$.
\end{theorem}

\begin{proof}
Let $\alpha \in (0,1)$. For $n \in \N$, set $D_n := \{j2^{-n}: j \in \Z \cap [0,2^n]\}$ to be the level-$n$ dyadics. For $k \in \N \cup \{\infty\}$, let $\Gamma_k$ denote the set of sequences $(\gamma_i)_{i=0}^k \in \Pi_{i=0}^{k} D_i$ with $|\gamma_{i}-\gamma_{i-1}| \leq 2^{-i}$ for $i \geq 1$. We canonically extend any finite sequence $\gamma \in \Gamma_k$ to an infinite (eventually constant) sequence $\overline{\gamma} \in \Gamma_\infty$ defined by $\overline{\gamma}_i := \gamma_i$ if $i \leq k$ and $\overline{\gamma}_i := \gamma_k$ if $i \geq k$, and we equip $\Gamma_k$ and $\Gamma_\infty$ with the constant 0 sequences as their basepoints. In this way, we get canonical basepoint-preserving injections $E_k: \Gamma_k \to \Gamma_\infty$, with the convention that $E_\infty: \Gamma_\infty \to \Gamma_\infty$ is the identity map. Equip each $\Gamma_k$ with a metric $d$ defined by $d(\gamma,\eta) := 2^{-i_{\gamma\eta}\alpha}$, where $i_{\gamma\eta} := \min \{i \in \Z \cap [0,k]: \gamma_i \neq \eta_i\}$ (with the convention $\min \emptyset = \infty$). Note that the canonical injections are isometric: for every $k \in \N \cup \{\infty\}$ and $\gamma,\eta \in \Gamma_k$, $d(E_k(\gamma),E_k(\eta)) = d(\gamma,\eta)$. It is standard to verify that $\Gamma_\infty$ equipped with this metric is a compact ultrametric space, and we omit that verification.

In the ensuing proof, we will produce a stochastic biLipschitz embedding of $(D_k,|\cdot|^\alpha)$ into $\Gamma_k$ for each $k \in \N$ with distortion $D \leq \frac{2^{5\alpha+1}}{(2^\alpha-1)(2-2^\alpha)}$ and scaling factor $s=4^{-\alpha}$. Since $\Gamma_k$ maps into $\Gamma_\infty$ isometrically, we get stochastic biLipschitz embeddings of each $(D_k,|\cdot|^\alpha)$ into $\Gamma_\infty$ with distortion $D \leq \frac{2^{5\alpha+1}}{(2^\alpha-1)(2-2^\alpha)}$. Then by the compactness property\footnote{Alternatively, we could observe that the stochastic biLipschitz embeddings $D_k \to \Gamma_\infty$ we produce satisfy a compatibility condition and therefore induce a stochastic biLipschitz embedding $[0,1] \to \Gamma_\infty$, but the compactness argument we present in Lemma~\ref{lem:stochasticcompact} is basic and likely to find further use outside of Theorem~\ref{thm:[0,1]snowflakestochastic}.} of Lemma~\ref{lem:stochasticcompact}, we get a stochastic biLipschitz embedding of $([0,1],|\cdot|^\alpha)$ into $\Gamma_\infty$ with the same distortion $D \leq \frac{2^{5\alpha+1}}{(2^\alpha-1)(2-2^\alpha)}$.

Before constructing the stochastic biLipschitz embeddings, we first note that the condition $|\gamma_{i}-\gamma_{i-1}| \leq 2^{-i}$ implies that we may sum a geometric series to obtain for all $k \geq 0$ (and $i_{\gamma\eta} \geq 1$) the estimates $|\gamma_k-\gamma_{i_{\gamma\eta}-1}|, |\eta_k-\eta_{i_{\gamma\eta}-1}| \leq 2^{1-i_{\gamma\eta}}$, and hence 
\begin{equation} \label{eq:endpointestimate}
    |\gamma_k-\eta_k| \leq 2^{2-i_{\gamma\eta}}
\end{equation}
since $\gamma_{i_{\gamma\eta}-1} = \eta_{i_{\gamma\eta}-1}$ (and note that this inequality also holds trivially if $i_{\gamma\eta} = 0$). 

For $k \geq 0$, consider $\{-1,1\}^k$ as a probability space equipped with the uniform probability measure. We define a sequence of random maps $\{\phi_\eps: D_k \to \Gamma_k\}_{\eps\in\{-1,1\}^k}$ recursively (interpreting $\{-1,1\}^0$ as $\{\emptyset\}$) by
\begin{equation*}
    \phi_\emptyset := id_{\{0,1\}}: D_0 = \{0,1\} \to \{0,1\} = \Gamma_0
\end{equation*}
and, for $k \geq 0$, $\eps \in \{-1,1\}^k$, $\delta \in \{-1,1\}$, and $j2^{-k-1} \in D_{k+1}$,
$$\phi_{(\eps,\delta)}(j2^{-k-1}) :=
\begin{cases}
    (\phi_\eps(j2^{-k-1}),j2^{-k-1}) & j \text{ even} \\
    (\phi_\eps((j+\delta)2^{-k-1}),j2^{-k-1}) & j \text{ odd} \\
\end{cases}.$$

Observe that, for every $\eps \in \{-1,1\}^k$ and $x \in D_k$,
\begin{equation} \label{eq:phikendpoint}
    \phi_\eps(x)_k = x.
\end{equation}
This says that $\phi(x)$ is almost surely a path ending at $x$. Combining \eqref{eq:endpointestimate} and \eqref{eq:phikendpoint} yields, for every $\eps \in \{-1,1\}^k$ and $x,y \in D_k$,
\begin{equation*}
    d(\phi_\eps(x),\phi_\eps(y)) \geq 4^{-\alpha}|x-y|^{\alpha}.
\end{equation*}
Thus, $\phi: (D_k,|\cdot|^\alpha) \to \Gamma_k$ is almost surely noncontractive with scaling factor $s=4^{-\alpha}$.

Let $k \geq 0$. We next derive the recursive inequality
\begin{equation} \label{eq:Lkrecursion}
    L_{k+1} \leq \frac{1}{2} + 2^{\alpha-1}L_k
\end{equation}
for the numbers $L_0 := 1$ and
\begin{equation*}
    L_{k+1} := 2^{(k+1)\alpha}\sup_{x \in D_k \setminus \{1\}} \E_{\eps \in \{-1,1\}^{k+1}}[d(\phi_\eps(x),\phi_\eps(x+2^{-k-1}))].
\end{equation*}
As we will see, \eqref{eq:Lkrecursion} will directly lead to an upper bound for the Lipschitz-in-expectation constant of $\phi$.

Let $x \in D_{k} \setminus \{1\}$. Then we have
\begin{align*}
    \E_{\eps \in \{-1,1\}^{k+1}}&[d(\phi_\eps(x),\phi_\eps(x+2^{-k-1}))] \\
    =& \frac{1}{2}\E_{\eps \in \{-1,1\}^{k}}[d_{k+1}((\phi_\eps(x),x),(\phi_\eps(x),x+2^{-k-1}))] \\
    &+ \frac{1}{2}\E_{\eps \in \{-1,1\}^{k}}[d_{k+1}((\phi_\eps(x),x),(\phi_\eps(x+2^{-k}),x+2^{-k-1}))] \\
    =& \frac{1}{2}2^{(-k-1)\alpha} + \frac{1}{2}\E_{\eps \in \{-1,1\}^{k}}[d_{k}(\phi_\eps(x),\phi_\eps(x+2^{-k}))] \\
    \leq& \frac{1}{2}2^{(-k-1)\alpha} + \frac{1}{2}L_k2^{-k\alpha} \\
    =& \left(\frac{1}{2}+ 2^{\alpha-1}L_k\right)2^{(-k-1)\alpha}.
\end{align*}
Since $x \in D_{k} \setminus \{1\}$ was arbitrary, this proves \eqref{eq:Lkrecursion}. This recursion yields the explicit bound
\begin{equation} \label{eq:Lkbound}
    \sup_{k \geq 0} L_k \leq \frac{1}{2-2^\alpha}.
\end{equation}

Let $x,y \in D_k$. By \cite[Lemma~8.40]{Weaver}, there exists $\{x_i\}_{i=0}^n \sbs [0,1]$ such that
\begin{itemize}
    \item $x_0 = x$ and $x_n = y$,
    \item $\sum_{i=1}^n |x_i-x_{i-1}| \leq 4|y-x|$, and
    \item there exists a sequence $i \mapsto m_i: \{1,2, \dots n\} \to \{0,1, \dots k\}$ that is at most 2-to-1 (the preimage of any point has cardinality at most 2) such that $x_{i-1},x_i$ are consecutive points in $D_{m_i}$ ($\{x_{i-1},x_i\} \sbs D_{m_i}$ and $(x_{i-1},x_i) \cap D_{m_i} = (x_i,x_{i-1}) \cap D_{m_i} = \emptyset$).
\end{itemize}
By the third item, there is a subset $J \sbs \{0,1,\dots k\}$ (with $|J| = n$, although this is not important) such that for any function $h: [0,\infty] \to [0,\infty)$, the estimate
\begin{align*}
    \sum_{j\in J} h(2^{-j}) \leq \sum_{i=1}^n h(|x_i-x_{i-1}|) \leq 2\sum_{j\in J} h(2^{-j})
\end{align*}
holds.
Applying this estimate for $h(t) = t^\alpha$ and $h(t) = t$, applying \cite[Lemma~8.39]{Weaver}, and using the second item above, we get
\begin{align*}
    \sum_{i=1}^n |x_i-x_{i-1}|^\alpha &\leq 2\sum_{j\in J} 2^{-j\alpha} \\
    &\leq \frac{2^{\alpha+1}}{2^\alpha-1}\left(\sum_{j\in J} 2^{-j}\right)^\alpha \\
    &\leq \frac{2^{\alpha+1}}{2^\alpha-1}\left(\sum_{i=1}^n |x_i-x_{i-1}|\right)^\alpha \\
    &\leq \frac{2^{3\alpha+1}}{2^\alpha-1}|y-x|^\alpha.
\end{align*}
With this bound and \eqref{eq:Lkbound}, we have
\begin{align*}
    \E_{\eps \in \{-1,1\}^{k}}[d_{k}(\phi_\eps(x),\phi_\eps(y))] &\leq \sum_{i=1}^n \E_{\eps \in \{-1,1\}^{k}}[d_{k}(\phi_\eps(x_{i-1}),\phi_\eps(x_{i}))] \\
    &\overset{\eqref{eq:Lkbound}}{\leq} \sum_{i=1}^n \frac{1}{2-2^\alpha} |x_i-x_{i-1}|^\alpha \\
    &\leq \frac{2^{3\alpha+1}}{(2^\alpha-1)(2-2^\alpha)} |y-x|^\alpha.
\end{align*}
Thus, $\phi$ is $L$-Lipschitz in expectation with $L \leq \frac{2^{3\alpha+1}}{(2^\alpha-1)(2-2^\alpha)}$. Since the scaling factor is $s = 4^{-\alpha}$, we get that $\phi$ is a stochastic biLipschitz embedding of distortion $D$ with $D \leq \frac{2^{5\alpha+1}}{(2^\alpha-1)(2-2^\alpha)}$.
\end{proof}

\begin{remark} \label{rmk:unbndd}
We expect that the construction used in the proof of Theorem~\ref{thm:[0,1]snowflakestochastic} could be modified to produce a stochastic biLipschitz embedding of $(\R,|\cdot|^\alpha)$ into an (unbounded) ultrametric space. The ultrametric space would need to consist of bi-infinite paths $(\gamma_i)_{i=-\infty}^{\infty} \in \Pi_{i=-\infty}^{\infty} D_i$, where $D_i = 2^{-i}\Z$. We chose not to pursue this direction here since we need only the bounded case $([0,1],|\cdot|^\alpha)$ for applications to $\ell^1$-isomorphisms of Lipschitz free spaces.
\end{remark}

\subsection{$\ell^{0+}$-Small Subsets of $\ell^\infty$}
\label{ss:l0+small}
In this subsection, we expand the scope of Theorem~\ref{thm:[0,1]snowflakestochastic} from $[0,1]$ to all $\ell^{0+}$-small subsets of $\ell^\infty([0,1])$ (Theorem~\ref{thm:l0+smallsnowflakestochastic}). We begin by defining these terms.

Given a pointed metric space $(X,d,p)$ and indexing set $J$, we denote by \\ $(\ell^\infty(J;X,d,p),d_{J,\infty},p_{J,\infty})$ the pointed metric space with underlying set $\ell^\infty(J;X,d,p) := \{\x \in X^J: \sup_{j\in J}d(\x_j,p) < \infty\}$, metric $d_{J,\infty}(\x,\y) := \sup_{j\in J}d(\x_j,\y_j)$, and basepoint $(p_{J,\infty})_j := p$ for all $j \in J$. As before, we will often suppress notation and simply write $\ell^\infty(J;X,d,p)$ or $\ell^\infty(J;X)$. Additionally, when $J=\N$, we further suppress notation and write $\ell^\infty(X,d,p)$ or $\ell^\infty(X)$. Any bijection between indexing sets $J_1 \to J_2$ induces a basepoint-preserving isometry $\ell^\infty(J_1;X) \to \ell^\infty(J_2;X)$.

\begin{definition}
Let $(X,d)$ be a pointed metric space, $J$ an indexing set, and $p>0$. A subset $S \sbs \ell^\infty(J;X)$ is said to be \emph{$\ell^p$-small} (with respect to $d_{J,\infty}$) with constant $A_p<\infty$ if for every $\x,\y \in S$,
\begin{equation*}
    \left(\sum_{j\in J} d(\x_j,\y_j)^p\right)^{\frac{1}{p}} \leq A_pd_{J,\infty}(\x,\y).
\end{equation*}
We often refer to the least such $A_p$ as the \emph{$\ell^p$-smallness constant of $S$}. We say $S \sbs \ell^\infty(J)$ is \emph{$\ell^{0+}$-small} (with respect to $d_{J,\infty}$) if it is $\ell^p$-small for every $p>0$.
\end{definition}

\begin{remark}\label{rmk:l0+small}
Clearly, the notion of $\ell^p$-smallness is invariant with respect to reindexing; if $J_1$ and $J_2$ are in bijection, then a subset $S \sbs \ell^\infty(J_1;X)$ is $\ell^p$-small with constant $A_p$ if and only if its induced image in $\ell^\infty(J_2;X)$ is $\ell^p$-small with constant $A_p$. Of course, the same invariance therefore holds for $\ell^{0+}$-smallness.  
\end{remark}

Note that if $p < \infty$ and $X$ is separable, then any $\ell^p$-small subset of $\ell^\infty(X)$ is separable as well.

\begin{example}[Uniformly Finite Support Implies $\ell^{0+}$-Smallness] \label{ex:l0+small}
Suppose $(X,p)$ is a pointed metric space, $J$ an indexing set, and $N\in\N$. We denote by $\ell_{\supp\leq N}^\infty(J;X)$ the subset of $\ell^\infty(J;X)$ consisting of elements $\x$ for which $|\{j\in J: \x_j \neq p\}| \leq N$. It is immediate to verify that $\ell_{\supp\leq N}^\infty(J;X)$ is $\ell^q$-small for every $q>0$, with constant $A_q \leq (2N)^{1/q}$, and therefore is $\ell^{0+}$-small.
\end{example}

For any $\alpha \in (0,1)$, observe that $\ell^\infty(J;X,d,p)$ and $\ell^\infty(J;X,d^\alpha,p)$ are the same set, but equipped with the different metrics $d_{J,\infty}$ and $(d^\alpha)_{J,\infty}$, respectively, and observe the ``commutation" equality
\begin{equation} \label{eq:commutation}
    (d^\alpha)_{J,\infty} = (d_{J,\infty})^\alpha.
\end{equation}
In the next lemma, we show that the notion of $\ell^{0+}$-smallness does not depend on $\alpha$.

\begin{lemma} \label{lem:l0+smallsnowflake}
Let $(X,d,p)$ be a pointed metric space, $J$ an indexing set, and $\alpha \in (0,1]$\footnote{The statement of the lemma actually holds even for $\alpha > 1$. The only reason we restrict $\alpha$ to be at most 1 is because $d^\alpha$ is generally not a metric for $\alpha > 1$ (it could fail triangle inequality), and so we would be required to define all the relevant terms in the more general setting where the triangle inequality holds only up to a multiplicative constant.}. Let $S \sbs \ell^\infty(J;X,d,p)$ (which, as discussed above, is the same set as $\ell^\infty(J;X,d^\alpha,p)$). Then $S$ is $\ell^{0+}$-small with respect to $d_{J,\infty}$ if and only if it is $\ell^{0+}$-small with respect to $(d^\alpha)_{J,\infty}$, where the set of $\ell^q$-smallness constants of $S$ with respect to $(d^\alpha)_{J,\infty}$ depends only on $\alpha$ and the set of $\ell^q$-smallness constants of $S$ with respect to $d_{J,\infty}$ (and vice versa).
\end{lemma}

\begin{proof}
We only show the forward implication, as the reverse follows by the same argument. Assume that $S$ is $\ell^{0+}$-small with respect to $d_{J,\infty}$. Let $q>0$, and let $A_{\alpha q}<\infty$ be the $\ell^{\alpha q}$-smallness constant of $S$ with respect to $d_{J,\infty}$. We will show that $S$ is $\ell^{q}$-small with respect to $(d^\alpha)_{J,\infty}$ with constant $A_{\alpha q}^\alpha<\infty$. Let $\x,\y \in S$. Then we have
\begin{align*}
    \left(\sum_{j\in J} d^\alpha(\x_j,\y_j)^q\right)^{\frac{1}{q}} = \left(\left(\sum_{j\in J} d(\x_j,\y_j)^{\alpha q}\right)^{\frac{1}{\alpha q}}\right)^\alpha \leq \left(A_{\alpha q}d_{J,\infty}(\x,\y)\right)^\alpha = A_{\alpha q}^\alpha (d^\alpha)_{J,\infty}(\x,\y).
\end{align*}
\end{proof}

The following lemma states that stochastic biLipschitz embeddability of $X$ into ultrametric spaces passes to $\ell^1$-small subsets of $\ell^\infty(X)$.

\begin{lemma} \label{lem:l1smallstochastic}
Suppose $(X,d,p)$ is a pointed metric space that stochastic biLipschitzly embeds into a separable, pointed ultrametric space with distortion $D<\infty$. Then every $\ell^1$-small subset of $\ell^\infty(X)$ (containing the basepoint) with $\ell^1$-smallness constant $A_1<\infty$ admits a stochastic biLipschitz embedding into a separable, pointed ultrametric space with distortion $DA_1$.
\end{lemma}

\begin{proof}
Let $p_\infty \in S \sbs \ell^\infty(X)$ be $\ell^1$-small with constant $A_1<\infty$. Let $\{\phi_\omega: X \to U\}_{\omega\in\Omega}$ be a stochastic biLipschitz embedding into a separable, pointed ultrametric space $(U,\rho,q)$ with distortion $D<\infty$ and scaling factor $s\in(0,\infty)$. We form the pointed product space $(\ell^\infty(U),\rho_\infty,q_\infty)$. It is clear that the underlying metric space $(\ell^\infty(U),\rho_\infty)$ is again an ultrametric space (but not separable). Let $(\Omega,\F,\P)$ be the probability space underlying the stochastic biLipschitz embedding. We will define our stochastic biLipschitz embedding over the power probability space $(\Omega^\N,\otimes_\N \F, \otimes_\N \P)$. Define $\{\bs{\Phi}_{\bs{\omega}}: S \to \ell^\infty(U)\}_{\bs{\omega}\in\Omega^\N}$ by $\left(\bs{\Phi}_{\bs{\omega}}(\x)\right)_n := \phi_{\bs{\omega}_n}(\x_n)$. Let $\x,\y \in S$. We do not know yet whether $\bs{\Phi}_{\bs{\omega}}$ is a random map. However, we will begin with an estimate of the expected Lipschitz constant, which will also help us verify that $\bs{\Phi}_{\bs{\omega}}$ is indeed a random map.
\begin{align*}
    \E_{\bs{\omega}\in\Omega^\N}\left[\rho_\infty(\bs{\Phi}_{\bs{\omega}}(\x),\bs{\Phi}_{\bs{\omega}}(\y))\right] &\leq \E_{\bs{\omega}\in\Omega^\N}\left[\sum_{n\in\N}\rho(\phi_{\bs{\omega_n}}(\x_n),\phi_{\bs{\omega_n}}(\y_n))\right] \\
    &= \sum_{n\in\N} \E_{\bs{\omega_n}\in\Omega}\left[\rho(\phi_{\bs{\omega_n}}(\x_n),\phi_{\bs{\omega_n}}(\y_n))\right] \\
    &\leq \sum_{n\in\N} sDd(\x_n,\y_n) \\
    &\leq sDA_1d_\infty(\x,\y).
\end{align*}
This establishes the required Lipschitz-in-expectation upper bound (note that all the terms appearing inside the expectation above are indeed $\R$-valued measurable random variables, so their expectation is well-defined). Furthermore, one can quickly examine the string of inequalities above with $p_\infty$ in place of $\x$ and $q_\infty$ in place of $\bs{\Phi}_{\bs{\omega}}(\y)$ to get that $\bs{\Phi}$ is almost surely basepoint-preserving. We can use this to verify pointwise almost-sure measurability and separability. Examining the second and last terms in the string of estimates above applied to $\y = p_\infty$, we get that there exists a $\otimes_\N \P$-null set $N_{\bs{x}}$ (which depends on $\bs{x}$) such that, for every $\bs{\omega}\in\Omega^\N \setminus N_{\bs{x}}$,
\begin{equation*}
    \sum_{n\in\N}\rho(\phi_{\bs{\omega_n}}(\x_n),q) < \infty.
\end{equation*}
Thus, by redefining $\bs{\Phi}_{\bs{\omega}}(\bs{x})$ for $\bs{\omega} \in N_{\bs{x}}$ to equal $q_\infty$ identically, we may assume that $\bs{\Phi}(\bs{x})$ takes values in the subset
\begin{equation*}
    V := \left\{\bs{z} \in \ell^\infty(U): \sum_{n\in\N} \rho(\bs{z}_n,q) < \infty\right\},
\end{equation*}
which is obviously separable since $(U,\rho)$ is separable. Thus, $\{\bs{\Phi}_{\bs{\omega}}\}_{\bs{\omega}\in\Omega^\N}$ takes values in the separable ultrametric space $(V,\rho_\infty)$. Finally, we can also use this to verify pointwise measurability. Since $(V,\rho_\infty)$ is separable, its Borel $\sigma$-algebra is generated by the collection of balls $\{\bs{y}\in S: \sup_{n\in\N}\rho(\bs{z}_n,\y_n) \leq r\}$, where $\bs{z}$ ranges over all points in $V$ and $r$ over all radii in $(0,\infty)$. Hence, pointwise measurability follows from the observation that $\{\bs{\omega}\in\Omega^\N: \sup_{n\in\N}\rho(\bs{z}_n,\phi_{\bs{\omega}_n}(\bs{x}_n)) \leq r\}$ is $\otimes_\N \F$-measurable.

It remains to check almost-sure noncontractivity. For each fixed $n\in\N$, there is a $\P$-null set $N_n$ such that, for all $\omega\in\Omega\setminus N_n$,
\begin{equation*}
    \rho(\phi_\omega(\x_n),\phi_{\omega}(\y_n)) \geq sd(\x_n,\y_n).
\end{equation*}
Then $\Pi_{n\in\N} (\Omega\setminus N_n)$ is a $\otimes_\N \P$-full-measure set, and for every $\bs{\omega} \in \Pi_{n\in\N} (\Omega\setminus N_n)$ and every $n\in\N$, we have that
\begin{equation*}
    \rho(\phi_{\bs{\omega}_n}(\x_n),\phi_{\bs{\omega_n}}(\y_n)) \geq sd(\x_n,\y_n).
\end{equation*}
Since $n \in \N$ was arbitrary, this implies that
\begin{equation*}
    \rho_\infty(\bs{\Phi}_{\bs{\omega}}(\x),\bs{\Phi}_{\bs{\omega}}(\y)) \geq sd_\infty(\x,\y)
\end{equation*}
for every $\bs{\omega} \in \Pi_{n\in\N} (\Omega\setminus N_n)$.
\end{proof}

\begin{remark} \label{rmk:[0,1]nembedums}
Lemma~\ref{lem:l1smallstochastic} shows that the interval $[0,1]$ with the Euclidean metric cannot stochastic biLipschitzly embed into an ultrametric space. Indeed, suppose that it does. Then since $[0,1]^2$ is biLipschitzly contained inside the $\ell^1$-small subset $\ell_{\supp\leq 2}^\infty([0,1]) \sbs \ell^\infty([0,1])$, Lemma~\ref{lem:l1smallstochastic} implies that $[0,1]^2$ stochastic biLipschitzly embeds into an ultrametric space. Since ultrametric spaces isometrically embed into $\R$-trees, Theorem~\ref{thm:stochasticembedtreeLFisoL1} implies that $\LF([0,1]^2) \approx L^1$, contradicting a result of Naor-Schechtman \cite{NaorSchechtman}. Together with Theorem~\ref{thm:[0,1]snowflakestochastic}, this shows that stochastic biLipschitz embeddability into ultrametric spaces is not preserved under snowflake equivalence.
\end{remark}

Combining Lemma~\ref{lem:l1smallstochastic} with Theorem~\ref{thm:[0,1]snowflakestochastic} yields the next result.

\begin{theorem} \label{thm:l0+smallsnowflakestochastic}
For every $\alpha \in (0,1)$ and every $\ell^{0+}$-small subset of $S \sbs \ell^\infty([0,1])$ containing the basepoint, there exists a stochastic biLipschitz embedding of $(S,\|\cdot\|_\infty^\alpha)$ into a separable, pointed ultrametric space. Moreover, the distortion of the embedding depends only on $\alpha$ and the $\ell^p$-smallness constants of $S$.
\end{theorem}

\begin{proof}
Let $\alpha \in (0,1)$ and $0 \in S \sbs \ell^\infty([0,1])$ an $\ell^{0+}$-small subset, and for each $p\in(0,\infty)$, $A_p<\infty$ the $\ell^p$-smallness constant of $S$. The pointed metric space $(\ell^\infty([0,1]),(\|~\cdot~\|_\infty)^\alpha,0)$ is identical to the pointed product metric space $(\ell^\infty([0,1],|\cdot|^\alpha),(|\cdot|^\alpha)_\infty,0)$ (this is immediate -- see also \eqref{eq:commutation}), and by Lemma~\ref{lem:l0+smallsnowflake}, $S$ is still an $\ell^{0+}$-small subset of this product space, with $\ell^p$-smallness constants depending only on $\alpha$ and $A_p$. Therefore, $(S,\|\cdot\|_\infty^\alpha)$ is basepoint-preservingly isometric to an $\ell^{0+}$-small subset of $\ell^\infty([0,1],|\cdot|^\alpha)$, and then the desired conclusion follows from Theorem~\ref{thm:[0,1]snowflakestochastic} and Lemma~\ref{lem:l1smallstochastic}.
\end{proof}

\subsection{Threshold and Snowflake Embeddings}
\label{ss:threshold}
The goal of this subsection is to extend the conclusion of Theorem~\ref{thm:l0+smallsnowflakestochastic} from $\ell^{0+}$-small subsets of $\ell^\infty([0,1])$ to all bounded, separable metrics spaces of finite Nagata dimension. We will accomplish this by proving that $(X,d^\alpha)$ admits a biLipschitz embedding into an $\ell^{0+}$-small subset of $\ell^\infty([0,1])$ whenever $\alpha \in (0,1)$ and $(X,d)$ is bounded, separable, and finite Nagata-dimensional (Lemma~\ref{lem:Nagatasnowflakel0+small}). Towards this end, we begin by recalling the notion of threshold embeddings. These types of maps have frequently appeared in embedding theory and were formally named in \cite{DLP}.

A \emph{threshold embedding with distortion $C<\infty$ and scaling factor $s\in(0,\infty)$} is a collection of maps $\{\varphi_t: X \to Y\}_{t>0}$ between metric spaces $(X,d_X),(Y,d_Y)$ such that, for every $x,y \in X$ and $t>0$,
\begin{itemize}
    \item $d_X(x,y) \geq t \implies d_Y(\varphi_t(x),\varphi_t(y)) \geq st$ and
    \item $d_Y(\varphi_t(x),\varphi_t(y)) \leq sC\min\{d_X(x,y),t\}$.
\end{itemize}
If $X$ and $Y$ are also pointed, then $\{\varphi_t\}_{t>0}$ is \emph{basepoint-preserving} if $\varphi_t$ is basepoint-preserving for every $t>0$.

The proof of the next lemma is adapted from the proof of Assouad's embedding theorem (e.g., \cite[Proof~of~Theorem~3.15]{Heinonen}).

\begin{lemma} \label{lem:thresholdsnowflake}
Let $(X,d_X,q)$ be a pointed metric space. If $X$ basepoint-preservingly threshold embeds into an $\ell^{0+}$-small subset $S\sbs \ell^\infty([0,\infty))$ with distortion $C<\infty$, then for every $\alpha \in (0,1)$, the space $(X,d^\alpha,q)$ admits a basepoint-preserving biLipschitz embedding of distortion $L<\infty$ into an $\ell^{0+}$-small subset $S'\sbs \ell^\infty([0,\infty))$, where $L$ and the $\ell^p$-smallness constants of $S'$ depend only on $C,\alpha$, and the $\ell^p$-smallness constants of $S$.
\end{lemma}

\begin{proof}
Let $S$ be an $\ell^{0+}$-small subset of $\ell^\infty([0,\infty))=\ell^{\infty}(\N;[0,\infty))$ and $\{\bs{\varphi}_t: X \to S\}_{t>0}$ a basepoint-preserving threshold embedding with distortion $C<\infty$. By postcomposing with a dilation (which is basepoint-preserving and maps an $\ell^p$-small set to an $\ell^p$-small set with the same smallness constant), we may assume that the scaling factor of the threshold embedding is $s=1$. Let $\alpha \in (0,1)$. By Remark~\ref{rmk:l0+small}, it suffices to embed into an $\ell^{0^+}$-small subset of $\ell^\infty(\N \times \Z;[0,\infty))$. Define $\bs{\psi}: X \to \ell^\infty(\N \times \Z;[0,\infty))$ by $\bs{\psi}(x)_{m,n} := 2^{n(\alpha-1)}(\bs{\varphi}_{2^n}(x))_m$. Then to reach the desired conclusion, it suffices to prove that, for every $p>0$, there exists $C'<\infty$ depending only on $A_p,p,C,\alpha$ with
\begin{equation*}
    (C')^{-1} \left(\sum_{m\in\N,n\in\Z} |\bs{\psi}_{m,n}(x)-\bs{\psi}_{m,n}(y)|^p\right)^{\frac{1}{p}} \leq d_X^\alpha(x,y) \leq 2^\alpha \sup_{m\in\N,n\in\Z} |\bs{\psi}_{m,n}(x)-\bs{\psi}_{m,n}(y)|
\end{equation*}
for every $x,y \in X$. Let $p>0$ and $A_p<\infty$ the $\ell^p$-smallness constant of $S$. Let $x,y \in X$ and $j \in \Z$ such that $2^j \leq d_X(x,y) \leq 2^{j+1}$. We will prove the lower bound for $d_X^\alpha(x,y)$ first. We have
\begin{align*}
    \Biggl(\sum_{m\in\N,n\in\Z} & |\bs{\psi}_{m,n}(x)-\bs{\psi}_{m,n}(y)|^p\Biggr)^{\frac{1}{p}} = \left(\sum_{n\in\Z}2^{n(\alpha-1)p}\sum_{m\in\N} |(\bs{\varphi}_{2^n}(x))_m-(\bs{\varphi}_{2^n}(y))_m|^p\right)^{\frac{1}{p}} \\
    &\leq  \left(\sum_{n\in\Z}2^{n(\alpha-1)p} \left(A_p\|\bs{\varphi}_{2^n}(x)-\bs{\varphi}_{2^n}(y)\|_{\infty}\right)^p\right)^{\frac{1}{p}} \\
\end{align*}
\begin{align*}
    &\leq A_p\max\{2^{\frac{1}{p}-1},1\}\left(\sum_{n\leq j}2^{n(\alpha-1)p} \|\bs{\varphi}_{2^n}(x)-\bs{\varphi}_{2^n}(y)\|_{\infty}^p\right)^{\frac{1}{p}} \\
    &\hspace{.3in} + A_p\max\{2^{\frac{1}{p}-1},1\}\left(\sum_{n\geq j+1}2^{n(\alpha-1)p} \|\bs{\varphi}_{2^n}(x)-\bs{\varphi}_{2^n}(y)\|_{\infty}^p\right)^{\frac{1}{p}} \\
    &\leq A_p\max\{2^{\frac{1}{p}-1},1\} \left(\left(\sum_{n\leq j}2^{n(\alpha-1)p}C^p2^{np}\right)^{\frac{1}{p}} + \left(\sum_{n\geq j+1}2^{n(\alpha-1)p}C^p2^{(j+1)p}\right)^{\frac{1}{p}}\right) \\
    &= A_pC\max\{2^{\frac{1}{p}-1},1\} \left(\left(\sum_{n\leq j}2^{\alpha pn}\right)^{\frac{1}{p}} + 2^{j+1}\left(\sum_{n\geq j+1}2^{(\alpha-1)pn}\right)^{\frac{1}{p}}\right) \\
    &= A_pC\max\{2^{\frac{1}{p}-1},1\} \left(\left(\frac{2^{\alpha pj}}{1-2^{-\alpha p}}\right)^{\frac{1}{p}} + 2^{j+1}\left(\frac{2^{(\alpha-1)p(j+1)}}{1-2^{(\alpha-1)p}}\right)^{\frac{1}{p}}\right) \\
    &= C'2^{\alpha j} \leq C'd^\alpha_X(x,y),
\end{align*}
where $C'<\infty$ depends only on $A_p,p,C,\alpha$, and not on $x,y$. Note also that this bound and the fact that $\bs{\psi}$ is basepoint-preserving verifies that $\bs{\psi}$ indeed takes values in $\ell^\infty(\N\times\Z;[0,\infty))$. Establishing the upper bound for $d_X^\alpha(x,y)$ is simpler:
\begin{align*}
    \sup_{m\in\N,m\in\Z}|\bs{\psi}(x)_{m,n}-\bs{\psi}(y)_{m,n}| & \geq 2^{j(\alpha-1)}\|\bs{\varphi}_{2^j}(x)-\bs{\varphi}_{2^j}(y)\|_\infty \\
    &\geq 2^{j(\alpha-1)}2^j \geq 2^{-\alpha}d_X^\alpha(x,y).
\end{align*}
\end{proof}

The proof of the following lemma is based on the proof of \cite[Theorem~1.6]{LS}.

\begin{lemma} \label{lem:Nagatathreshold}
For every $n\in\N$, every separable pointed metric space $X$ of Nagata dimension $n\in\N$ with constant $\gamma<\infty$ basepoint-preservingly threshold embeds with distortion $C \leq 4\gamma$ into $\ell^\infty_{\supp\leq N}([0,\infty))$ for $N = 2(n+1)$ (see Example~\ref{ex:l0+small} for the definition of this set). Consequently, $X$ basepoint-preservingly threshold embeds with distortion $C \leq 4\gamma$ into an $\ell^{0+}$-small subset $S$ of $\ell^\infty([0,\infty))$, where the $\ell^p$-smallness constants of $S$ depends only on $n,\gamma$.
\end{lemma}

\begin{proof}
We begin by noting that the second sentence follows from the first and the discussion in Example~\ref{ex:l0+small}. We will now prove the first sentence.

Let $n\in\N$ and $(X,d,p)$ be a pointed separable metric space of Nagata dimension $n\in\N$ and constant $\gamma<\infty$. By replacing $\gamma$ with $\max\{\frac{1}{2},\gamma\}$, we may assume that $\gamma \geq \frac{1}{2}$. Let $t>0$. Let $\{B_i\}_{i\in\N}$ be a Nagata cover of dimension $n$, constant $\gamma$, and scale $s=(2\gamma)^{-1}t$. We can choose this cover to be countable since $X$ is separable. Define $\bs{\varphi}_t: X \to \ell^\infty([0,\infty))$ by $(\bs{\varphi}_t(x))_i = \max\{0,(4\gamma)^{-1}t-\dist(B_i,x)\}$. Since distance-to-subset function are 1-Lipschitz and the pointwise maximum of 1-Lipschitz functions is 1-Lipschitz, we obviously have $\|\bs{\varphi}_t(x)-\bs{\varphi}_t(y)\|_\infty \leq \max\{1,(4\gamma)^{-1}\}\min\{d(x,y),t\}$ for every $x,y \in X$.

Suppose $x,y \in X$ with $d(x,y) \geq t$. Choose $j\in\N$ such that $x \in B_j$. Then $\dist(B_j,y) \geq d(x,y) - \diam(B_j) \geq t/2$, which implies $(\bs{\varphi}_t(x))_j = (4\gamma)^{-1}t$ and $(\bs{\varphi}_t(y))_j = 0$ (since $\gamma \geq \frac{1}{2}$), and hence $\|\bs{\varphi}_t(x)-\bs{\varphi}_t(y)\|_\infty \geq (4\gamma)^{-1}t$. 

Finally, fix $x \in X$, and let $I_x = \{i\in\N: (\bs{\varphi}_t(x))_i \neq 0\}$. We will show that $|I_x| \leq n+1$. By definition of $\bs{\varphi}_t$ and $I_x$, we can find, for each $i \in I_x$, a point $y_i \in B_i$ such that $d(x,y_i) < (4\gamma)^{-1}t$. Then the set $\{y_i\}_{i\in I_x}$ has diameter at most $(2\gamma)^{-1}t$, and hence it can have nonempty intersection with $B_i$ for at most $n+1$ values of $i\in\N$. Since it has nonempty intersection with $B_i$ for every $i\in I_x$, we have $|I_x| \leq n+1$.

The collection of maps $\{\bs{\varphi}_t: X \to \ell^\infty([0,\infty))\}_{t>0}$ is thus a threshold embedding of distortion $\max\{1,4\gamma\}$ and scaling factor $s=(4\gamma)^{-1}$ into $\ell^\infty_{\supp\leq n+1}([0,\infty))$. However, it is not necessarily basepoint-preserving. To fix this, we simply instead consider the maps $\{\bs{\varphi}_t - \bs{\varphi}_t(p)\}_{t>0}$, which clearly form a threshold embedding with the same distortion and scaling factor into $\ell^\infty_{\supp\leq 2(n+1)}([0,\infty))$ and map $p$ to 0.
\end{proof}

We next combine Lemmas~\ref{lem:thresholdsnowflake}~and~\ref{lem:Nagatathreshold}. 

\begin{lemma} \label{lem:Nagatasnowflakel0+small}
For every separable pointed metric space $(X,d,p)$ with finite Nagata dimension (dimension $n\in\N$ with constant $\gamma<\infty$) and $\alpha \in (0,1)$, the pointed metric space $(X,d^{\alpha},p)$ basepoint-preservingly biLipschitzly embeds into an $\ell^{0+}$-small subset $S$ of $\ell^\infty([0,\infty))$ with distortion $L$, where $L$ and the $\ell^q$-smallness constants of $S$ depend only on $n,\gamma,\alpha$. If $(X,d)$ is also bounded, then $(X,d^{\alpha},p)$ basepoint-preservingly biLipschitzly embeds into an $\ell^{0+}$-small subset $S$ of $\ell^\infty([0,1])$ (with the same distortion and $\ell^q$-smallness constants).
\end{lemma}

\begin{proof}
Let $(X,d,p)$ be a pointed separable metric space with finite Nagata dimension and $\alpha \in (0,1)$. The first sentence follows immediately from Lemmas~\ref{lem:thresholdsnowflake}~and~\ref{lem:Nagatathreshold}.

For the second sentence, assume that $(X,d)$ is bounded. Then $(X,d^\alpha)$ is also bounded, and thus its image in $\ell^\infty([0,\infty))$ under the biLipschitz embedding must be contained in $\ell^\infty([0,B])$ for some $B< \infty$. By postcomposing with a rescaling (which is basepoint-preserving and does not affect the biLipschitz distortion), we can arrange the target space to be $\ell^\infty([0,1])$. Since the image of an $\ell^{0+}$-small set under a rescaling is also $\ell^{0+}$-small with the same $\ell^q$-smallness constants (this fact is immediate from the definition), this proves the second sentence.
\end{proof}

In the next lemma, we will see that the target ultrametric space of a stochastic biLipschitz embedding can always taken to be bounded when the domain is bounded.

\begin{lemma} \label{lem:unbnddtobndd}
Suppose $X$ is a pointed metric space that stochastic biLipschitzly embeds into a pointed ultrametric space $(U,d,q)$ with distortion $D<\infty$ and scaling factor $s\in(0,\infty)$. If $X$ is bounded, then $X$ stochastic biLipschitzly embeds into a bounded, pointed ultrametric space $U'$ with distortion $D$ and scaling factor $s$, and if $U$ is separable, then $U'$ can be taken separable as well.
\end{lemma}

\begin{proof}
Let $\{\phi_\omega: X \to (U,d,q)\}_{\omega\in\Omega}$ be a stochastic biLipschitz embedding of distortion $D<\infty$ and scaling factor $s\in(0,\infty)$. Assume that $X$ is bounded with $|X| \geq 2$ (the conclusion of the lemma is obvious if $|X| \leq 1$), and define a new function $\rho: U \times U \to [0,\infty)$ by $\rho(x,y) = \min\{\diam(X),d(x,y)\}$. It is not difficult to see that $\rho$ is an ultrametric on $U$ such that $(U,\rho)$ is bounded, $(U,\rho)$ is separable whenever $(U,d)$ is separable, and $\{\phi_\omega: X \to (U,\rho,q)\}_{\omega\in\Omega}$ is a stochastic biLipschitz embedding of distortion $D<\infty$ and scaling factor $s\in(0,\infty)$.  
\end{proof}

Finally, we arrive at our next main result of the article, which, together with Theorem~\ref{thm:A}, proves Theorem~\ref{thm:B}.

\begin{theorem} \label{thm:Nagatasnowflakestochastic}
For every bounded, separable, pointed metric space $(X,d,p)$ with finite Nagata dimension (dimension $n\in\N$ with constant $\gamma<\infty$) and every $\alpha \in (0,1)$, the pointed space $(X,d^\alpha,p)$ admits a stochastic biLipschitz embedding of distortion $D<\infty$ into a bounded, separable, pointed ultrametric space, where $D$ depends only on $n,\gamma,\alpha$.
\end{theorem}

\begin{proof}
Let $(X,d,p)$ be a bounded, separable, pointed metric space with Nagata dimension $n\in\N$ and constant $\gamma<\infty$, and let $\alpha \in (0,1)$. By Lemma~\ref{lem:Nagatasnowflakel0+small}, there is a basepoint-preserving biLipschitz embedding of distortion $L<\infty$ of $(X,d^{\sqrt{\alpha}},p)$ into an $\ell^{0+}$-small subset $S$ of $\ell^\infty([0,1])$, where $L$ and the set of $\ell^q$-smallness constants of $S$ depend only on $n,\gamma,\alpha$. Then by Theorem~\ref{thm:l0+smallsnowflakestochastic}, there is a stochastic biLipschitz embedding of distortion $D<\infty$ of $(X,(d^{\sqrt{\alpha}})^{\sqrt{\alpha}},p) = (X,d^\alpha,p)$ into a separable, pointed ultrametric space, where $D$ depends only on $n,\gamma,\alpha$. By Lemma~\ref{lem:unbnddtobndd}, we may assume that the ultrametric space is also bounded.
\end{proof}

\begin{remark}
If it could be established that $(\R,|\cdot|^\alpha)$ stochastic biLipschitzly embeds into a separable (unbounded) ultrametric space for every $\alpha < 1$ (see Remark~\ref{rmk:unbndd}), then the same argument used in the proof of Theorem~\ref{thm:Nagatasnowflakestochastic} would show that for every $\alpha<1$ and separable metric space $(X,d)$ with finite Nagata dimension, the snowflake space $(X,d^\alpha)$ stochastic biLipschitzly embeds into a separable (unbounded) ultrametric space.
\end{remark}

\subsection{Free Spaces over Distorted Finite Nagata-Dimensional Spaces}
\label{ss:pconcave}
We conclude this section by applying Theorem~\ref{thm:Nagatasnowflakestochastic} to study the Lipschitz free space of finite Nagata-dimensional metric spaces distorted by a concave function. Our methods allow for a solution to an open problem of Weaver concerning distorted spaces that fail to be doubling (see Theorem~\ref{thm:pconcave} and the preceding discussion). We begin by recalling concave distortion functions, following \cite[$\S$2.6]{Weaver}.

We call a function $\omega: [0,\infty) \to [0,\infty)$ a \emph{distortion function} if $\omega$ is continuous and concave\footnote{A function $\omega: [0,\infty) \to \R$ is \emph{concave} if $\omega(\lambda a + (1-\lambda)b) \geq \lambda\omega(a) + (1-\lambda)\omega(b)$ for all $\lambda\in [0,1]$ and $a,b \in [0,\infty)$.}, $\omega(0) = 0$, and $\omega(t) > 0$ for $t > 0$. Each distortion function $\omega$ is subadditive\footnote{A function $\omega: [0,\infty) \to \R$ is \emph{subadditive} if $\omega(s+t) \leq \omega(s) + \omega(t)$ for all $s,t \in [0,\infty)$.} and admits a unique extended positive number $a \in (0,\infty]$ such that $\omega$ is strictly increasing on $[0,a]$ (to be interpreted as $[0,\infty)$ if $a=\infty$) and then constant on $[a,\infty)$ (interpreting $[\infty,\infty)$ as the empty set). Because of this, we will unambiguously write $\omega^{-1}$ for the inverse of the homeomorphism $\omega\big|_{[0,a]}: [0,a] \to [0,\omega(a)]$. Whenever $\omega$ is a distortion function and $d$ is a metric on $X$, the distorted function $\omega \circ d$ is again a metric on $X$ inducing the same topology as that of $d$. In the next lemma, we record the effect on Nagata dimension of metric distortion (see also \cite[Lemma~2.1]{LS} for a general result along these same lines).

\begin{lemma} \label{lem:Nagatadistortion}
Suppose $(X,d)$ is a metric space of Nagata dimension $n$ with constant $\gamma$. Then for all distortion functions $\omega$, the metric space $(X,\omega\circ d)$ has Nagata dimension $n$ with constant at most $2\gamma$.
\end{lemma}

\begin{proof}
Let $\omega: [0,\infty) \to [0,\infty)$ be a distortion function. Let $s \in [0,\diam(X,\omega \circ d)]$ with $s<\infty$, and let $B_1,B_2,\dots B_{n+1}$ be a cover of $X$ such that each $B_i$ is the union of a family of sets that is $\gamma\omega^{-1}(s)$-bounded and $\omega^{-1}(s)$-separated with respect to $d$. 

Then with respect to $\omega \circ d$, these same families are $\omega(\gamma\omega^{-1}(s))$-bounded and $s$-separated. It remains to show that $\omega(\gamma\omega^{-1}(s)) \leq 2\gamma s$. By subadditivity and induction, it follows that $\omega(2^nt) \leq 2^n\omega(t)$ for every $n\in \Z$ and $t \in [0,\infty)$. Choose $n \in \Z$ such that $2^n \leq \gamma < 2^{n+1}$. Then by monotonicity and subadditivity, we get
\begin{align*}
    \omega(\gamma\omega^{-1}(s)) \leq \omega(2^{n+1}\omega^{-1}(s)) \leq 2^{n+1}s \leq 2\gamma s.
\end{align*}
\end{proof}

For $p \in [1,\infty)$, a distortion function $\omega$ is called \emph{$p$-concave} if $\omega^p$ is concave. In this case, $\omega^p$ is also a distortion function. In the following lemma, we provide a condition for a sufficiently smooth distortion function to be $p$-concave.

\begin{lemma} \label{lem:pconcave}
Suppose $p \in [1,\infty)$ and $\omega$ is a distortion function that is continuously differentiable on $(0,\infty)$. If $\omega''(t)$ exists and $\omega''(t) \leq (1-p)\dfrac{\omega'(t)^2}{\omega(t)}$ for all $t \in (0,\infty) \setminus D$, where $D$ is discrete, then $\omega$ is $p$-concave.
\end{lemma}

\begin{proof}
Assume that $\omega''(t)$ exists for all $t\in (0,\infty) \setminus D$, where $D$ is discrete. Since
\begin{equation*}
    \frac{d^2}{dt^2}(\omega(t)^p) = p\omega(t)^{p-2}(\omega(t)\omega''(t) + (p-1)\omega'(t)^2),
\end{equation*}
we see that $\omega^p$ is concave if and only if $\frac{d^2}{dt^2}(\omega(t)^p) \leq 0$ for all $t \in (0,\infty) \setminus D$ if and only if $\omega''(t) \leq (1-p)\dfrac{\omega'(t)^2}{\omega(t)}$ for all $t \in (0,\infty) \setminus D$.
\end{proof}

\begin{example} \label{ex:omega_p}
Fix $p \in [1,\infty)$. Define a distortion function $\omega_p: [0,\infty) \to [0,\infty)$ by
\begin{equation*}
    \omega_p(t) = \begin{cases}
        0 & t=0 \\
        \log(\frac{1}{t})^{-1} & 0 < t \leq e^{-p-1} \\
        \left(\frac{pe^{p+1}t+1}{(p+1)^{p+1}}\right)^{\frac{1}{p}} & e^{-p-1} \leq t
    \end{cases}.
\end{equation*}
It is a slightly tedious but routine verification that $\omega_p$ is a distortion function, continuously differentiable on $(0,\infty)$, which satisfies $\omega_p''(t) \leq (1-p)\dfrac{\omega_p'(t)^2}{\omega_p(t)}$ for all $t \in (0,\infty) \setminus \{e^{-p-1}\}$. Hence, $\omega_p$ is $p$-concave by Lemma~\ref{lem:pconcave}.
\end{example}

It holds that $\LF(X,d^{\alpha}) \approx \ell^1$ whenever $(X,d)$ is an infinite, compact, doubling metric space and $\alpha \in (0,1)$ (\cite[Theorem~8.49]{Weaver}). However, for distortion functions $\omega$ satisfying $\lim_{t\to 0}\frac{t^\alpha}{\omega(t)} = 0$ for all $\alpha \in (0,1)$ (such as $\omega = \omega_p$ from Example~\ref{ex:omega_p}), the space $(X,\omega \circ d)$ fails to be doubling, and the validity of statement $\LF(X,\omega \circ d) \approx \ell^1$ remained unknown (see \cite[page 294]{Weaver}). In the following theorem, we fill this gap for $p$-concave distortion functions, which includes Example~\ref{ex:omega_p} as a special case.

\begin{theorem} \label{thm:pconcave}
Let $p \in (1,\infty)$ and let $\omega$ be a $p$-concave distortion function. Let $(X,d,q)$ be a bounded, separable, finite Nagata-dimensional pointed metric space. Then there exists $D<\infty$ such that $(X,\omega \circ d,q)$ stochastic biLipschitzly embeds into a bounded, separable, pointed ultrametric space with distortion $D$. Furthermore, if $X$ is infinite and if $X$ is compact or $\LF(X,\omega \circ d)$ is complemented in $\LF(X,\omega \circ d)^{**}$, then $\LF(X,\omega \circ d) \approx \ell^1$.
\end{theorem}

\begin{proof}
Since $\omega^p$ is a distortion function and $(X,d)$ has finite Nagata dimension, the metric space $(X,\omega^p \circ d)$ has finite Nagata dimension by Lemma~\ref{lem:Nagatadistortion}. Therefore, by Theorem~\ref{thm:Nagatasnowflakestochastic}, the pointed snowflake space $(X,(\omega^p \circ d)^{\frac{1}{p}},q) = (X,\omega \circ d,q)$ stochastic biLipschitzly embeds into a bounded, separable, pointed ultrametric space with distortion. This proves the first conclusion.

For the second conclusion, assume that $X$ is infinite and that $X$ is compact or $\LF(X,\omega \circ d)$ is complemented in $\LF(X,\omega \circ d)^{**}$. Then since ultrametric spaces isometrically embed into $\R$-trees, the first conclusion of this theorem, Theorem~\ref{thm:stochasticembedtreeLFisoL1}, and the fact that snowflake spaces $(Y,d^\alpha)$ have purely 1-unrectifiable completion (whenever $(Y,d)$ is a metric space and $\alpha \in (0,1)$, see Proof of Theorem~\ref{thm:B}) imply that $\LF(X,\omega \circ d) \approx \ell^1$.
\end{proof}

\section{log-Stochastic Embeddings and Hyperbolic Fillings}
\label{s:stochasticfillings}
In this section, we show how Bonk-Schramm hyperbolic fillings (see Lemma~\ref{lem:BS}) can be used to induce stochastic embeddings between visual hyperbolic metric spaces from stochastic embeddings between their boundaries (Lemma~\ref{lem:Hypstochastic}). It turns out that for boundaries of hyperbolic spaces, the correct notion of stochastic embedding is a relaxation of stochastic biLipschitz embedding that we call ``log-stochastic" biLipschitz embeddings.

\begin{definition}
Let $(X,d_X),(Y,d_Y)$ be pointed metric spaces. A random map $\{\phi_\omega: X \to Y\}_{\omega\in\Omega}$ is a \emph{log-stochastic biLipschitz embedding of distortion $D<\infty$ and scaling factor $s\in(0,\infty)$} if, for all $x,y \in X$,
\begin{itemize}
    \item $d_Y(\phi_\omega(x),\phi_\omega(y)) \geq sd_X(x,y)$ for almost every $\omega\in\Omega$ and
    \item $\E_\omega\left[\log\left(d_Y(\phi_\omega(x),\phi_\omega(y))\right)\right] \leq \log(sDd_X(x,y))$.
\end{itemize}
\end{definition}

\begin{remark}
By Jensen's inequality, a stochastic biLipschitz embedding of distortion $D$ is a log-stochastic biLipschitz embedding of distortion $D$.
\end{remark}

One of the main reasons that log-stochastic biLipschitz embeddings are the correct notion to study for Gromov boundaries is that, for certain target spaces, they are inherited under snowflake embeddings. The next lemma reveals this to be the case when the target class is ultrametric spaces. This is opposite the situation for stochastic biLipschitz embeddability into ultrametric spaces, which is not inherited snowflake embeddings (see Remark~\ref{rmk:[0,1]nembedums}).

\begin{lemma}\label{lem:snowflakelogstochastic}
If $(X,d_X), (Y,d_Y)$ are pointed metric spaces such that $X$ basepoint-preservingly snowflake embeds into $Y$ with exponent $\alpha\in(0,\infty)$ and distortion $L<\infty$, and $Y$ log-stochastic biLipschitzly embeds into a pointed ultrametric space $U$ with distortion $D<\infty$, then $X$ log-stochastic biLipschitzly embeds into a pointed ultrametric space $U'$ with distortion $D^\alpha L$ and scaling factor $s=1$. Moreover, if $U$ is separable or bounded, then $U'$ is separable or bounded, respectively.
\end{lemma}

\begin{proof}
Let $\alpha \in (0,\infty)$, $L < \infty$, $s\in(0,\infty)$ and $f: X \to Y$ such that, for all $x,y \in X$, $sd_X(x,y) \leq d_Y(f(x),f(y))^\alpha \leq sLd_X(x,y)$. Let $\{\phi_\omega: Y \to U\}_{\omega\in\Omega}$ be a log-stochastic biLipschitz embedding of distortion $D<\infty$ and scaling factor $t\in(0,\infty)$ into a pointed ultrametric space $(U,d_U,q)$. Then the triple $(U',d_{U'},q)$ defined by $U' := U$ and $d_{U'} := s^{-1}(t^{-1}d_U)^{\alpha}$ is a pointed ultrametric space (which is separable or bounded if $U$ is), and we'll show that the composite $\{id \circ \phi_\omega \circ f: X \to U'\}_{\omega\in\Omega}$ is a log-stochastic biLipschitz embedding of distortion $D^\alpha 
L$ and scaling factor $1$, where $id: U \to U'$ is the identity map. First note that $\{id \circ \phi_\omega \circ f\}_{\omega\in\Omega}$ is clearly pointwise measurable and essentially-separably-valued since $id$ is a homeomorphism. It is clearly almost surely basepoint-preserving as well. Let $x,y \in X$. Then, for almost every $\omega$,
\begin{align*}
    d_{U'}(id(\phi_\omega(f(x))),id(\phi_\omega(f(y)))) &= s^{-1}(t^{-1}d_{U}(\phi_\omega(f(x)),\phi_\omega(f(y))))^{\alpha} \\
    &\geq s^{-1}d_{Y}(f(x),f(y))^{\alpha} \\
    &\geq d_X(x,y).
\end{align*}
Next,
\begin{align*}
    \E_\omega[\log(d_{U'}&(id(\phi_\omega(f(x))),id(\phi_\omega(f(y)))))] \\
    &= \E_\omega[\log(s^{-1}(t^{-1}d_{U}(\phi_\omega(f(x)),\phi_\omega(f(y))))^{\alpha})] \\
    &= -\log(s) -\alpha\log(t) + \alpha\E_\omega[\log(d_{U}(\phi_\omega(f(x)),\phi_\omega(f(y))))] \\
    &\leq -\log(s) -\alpha\log(t) + \alpha\log(tDd_{Y}(f(x),f(y))) \\
    &= \log(s^{-1}D^\alpha d_{Y}(f(x),f(y))^\alpha) \\
    &\leq \log(D^\alpha L d_{X}(x,y)).
\end{align*}
\end{proof}

Lemma~\ref{lem:snowflakelogstochastic} shows that the property of log-stochastic biLipschitz embeddability into (separable or bounded) ultrametric spaces is preserved under snowflake equivalences. Therefore, when $X$ is a hyperbolic metric space, we may unambiguously say that

\emph{``$\partial X$ log-stochastic biLipschitzly embeds into a bounded ultrametric space"} \\
without making reference to any particular visual metric on $\partial X$. We will make such a statement in Theorem~\ref{thm:Hypstochasticums} with this understanding in mind.

The following result is an immediate corollary of  Lemma~\ref{lem:snowflakelogstochastic} and Theorem~\ref{thm:Nagatasnowflakestochastic} and provides our main example of spaces that log-stochastic biLipschitzly embed into ultrametric spaces. We omit the proof.

\begin{corollary} \label{cor:Nagatalogstochastic}
Let $Z$ be a bounded, separable, pointed metric space of finite Nagata dimension. Then $Z$ log-stochastic biLipschitzly embeds into a separable, bounded, pointed ultrametric space.
\end{corollary}

The next result is the fundamental lemma of log-stochastic biLipschitz embeddings, and it should be seen as our motivation to define such a notion in the first place.

\begin{lemma} \label{lem:Hypstochastic}
If $\{\phi_\omega: X \to Y\}_{\omega\in\Omega}$ is a log-stochastic biLipschitz embedding of distortion $D<\infty$ and scaling factor $s=1$ of a pointed metric space $(X,d_X,p)$ into a bounded, pointed metric space $(Y,d_Y,q)$, then $\diam(X) \leq \diam(Y)$ and $\{\Hyp(\phi_\omega): \Hyp(X) \to \Hyp(Y)\}_{\omega\in\Omega}$ is a stochastic rough isometric embedding with roughness constant $K \leq 2\log(D)$. Here, $\Hyp(X)$ can be equipped with basepoint $(p,h)$ for any $h \in (0,\diam(X)]$, and $\Hyp(Y)$ is then equipped with basepoint $(q,h)$.
\end{lemma}

\noindent Before proving this result, we need a simple lemma that essentially amounts to a calculus exercise.

\begin{lemma} \label{lem:log-stochasticbound}
Let $Z$ be a $(0,\infty)$-valued random variable, $\alpha \in (0,\infty)$, and $D \in [1,\infty)$ such that $Z \geq \alpha$ almost surely and $\E[\log(Z)] \leq \log(D\alpha)$ (in particular, $\log(Z)$ is integrable). Then for all $\beta \in [0,\infty)$, $\E[\log(Z+\beta)] \leq \log(D(\alpha+\beta))$.
\end{lemma}

\begin{proof}
Let $(\Omega,\F,\P)$ denote the probability space underlying $Z$.
Define $f: [0,\infty) \to \R$ by
\begin{equation*}
    f(\beta) := \E[\log(Z+\beta)] - \log(D(\alpha+\beta)).
\end{equation*}
Then $f$ is continuous on $[0,\infty)$. For $\P$-almost every $\omega\in\Omega$, the function $\beta \mapsto \log(Z(\omega)+\beta)$ is differentiable on $(0,\infty)$ with derivative $\beta \mapsto (Z(\omega)+\beta)^{-1} \leq \alpha^{-1}$ bounded uniformly in $\omega$. Therefore, by the mean value theorem and dominated convergence theorem, we have that $f$ is differentiable on $(0,\infty)$ with
\begin{equation*}
    f'(\beta) = \E[(Z+\beta)^{-1}] - (\alpha+\beta)^{-1} = \E\left[\frac{\alpha-Z}{(Z+\beta)(\alpha+\beta)}\right] \leq 0.
\end{equation*}
Thus, $f$ is decreasing on $[0,\infty)$. Since $f(0) \leq 0$ by assumption, this implies $f(\beta) \leq 0$ for all $\beta \geq 0$, which proves the lemma.
\end{proof}

\begin{proof}[Proof of Lemma~\ref{lem:Hypstochastic}]
Let $\{\phi_\omega: (X,d_X,p) \to (Y,d_Y,q)\}_{\omega\in\Omega}$ be a log-stochastic biLipschitz embedding of distortion $D<\infty$ and scaling factor $s=1$ between bounded, pointed metric spaces. Choose any $h \in (0,\diam(X)]$, and equip $\Hyp(X)$ with basepoint $(p,h)$ and $\Hyp(Y)$ with basepoint $(q,h)$. First note that the almost-sure noncontractivity of $\{\phi_\omega\}_{\omega\in\Omega}$ ensures $\diam(X) \leq \diam(Y)$, and thus $\Hyp(\phi_\omega)$ is well-defined for each $\omega\in\Omega$. Next, note that $\{\Hyp(\phi_\omega)\}_{\omega\in\Omega}$ is clearly pointwise measurable and essentially-separably-valued since, for each $h \in (0,\diam(Y)]$, the map $y \mapsto (y,h)$ taking $Y$ into $\Hyp(Y)$ is a homeomorphic embedding. It is also clearly almost surely basepoint-preserving.

Let $(x,h) \neq (x',h') \in \Hyp(X)$. Then for almost every $\omega\in\Omega$,

\begin{align*}
    \rho_Y(\Hyp(\phi_\omega)(x,h),\Hyp(\phi_\omega)(x',h'))
    &= 2\log\left(\dfrac{d_Y(\phi_\omega(x),\phi_\omega(x')) + h \vee h'}{\sqrt{hh'}}\right) \\
    &\geq 2\log\left(\dfrac{d_X(x,x') + h \vee h'}{\sqrt{hh'}}\right) \\
    &= \rho_X((x,h),(x',h')).
\end{align*}

To prove that the stochastic roughness constant is $\leq 2\log(D)$, we need to treat two cases: (i) $x=x'$ and (ii) $x \neq x'$. In case (i), we have
\begin{align*}
    \E_\omega[\rho_Y(\Hyp(\phi_\omega)(x,h),\Hyp(\phi_\omega)(x',h'))] = \E_\omega\left[2\log\left(\frac{h \vee h'}{\sqrt{hh'}}\right)\right] = \rho_X((x,h),(x',h')).
\end{align*}
In case (ii), we apply Lemma~\ref{lem:log-stochasticbound} with $Z = \dfrac{d_Y(\phi_\omega(x),\phi_\omega(x'))}{\sqrt{hh'}}$, $\alpha = \dfrac{d_X(x,x')}{\sqrt{hh'}}$, $D=D$, and $\beta = \dfrac{h \vee h'}{\sqrt{hh'}}$, we get
\begin{align*}
    \E_\omega[\rho_Y(\Hyp(\phi_\omega)(x,h),\Hyp(\phi_\omega)(x',h'))] &= \E_\omega\left[2\log\left(\dfrac{d_Y(\phi_\omega(x),\phi_\omega(x')) + h \vee h'}{\sqrt{hh'}}\right)\right] \\
    &\leq 2\log\left(\dfrac{D(d_X(x,x') + h \vee h')}{\sqrt{hh'}}\right) \\
    &= \rho_X((x,h),(x',h')) + 2\log(D).
\end{align*}
\end{proof}

When Lemma~\ref{lem:Hypstochastic} holds with $Y$ an ultrametric space, we can apply Lemmas~\ref{lem:BS}~and~\ref{lem:Hyp(U)T} and get a general criterion for stochastic biLipschitz embeddability of uniformly discrete hyperbolic metric spaces into $\R$-trees, which proves Theorem~\ref{thm:C}. Observe that a uniformly discrete metric space is locally finite if and only if it is proper. 

\begin{theorem} \label{thm:Hypstochasticums}
Let $A$ be an infinite, uniformly discrete, locally finite, pointed metric space, and let $X$ be a visual, hyperbolic, pointed metric space. If $A$ rough biLipschitzly embeds into $X$ and $\partial X$ log-stochastic biLipschitzly embeds (for every choice of basepoint) into a bounded, pointed ultrametric space, then $A$ stochastic biLipschitzly embeds into a pointed $\R$-tree and $\LF(A) \approx \ell^1$.
\end{theorem}

\begin{proof}
Assume that there is a rough biLipschitz embedding $f_1: A \to X$ and, with respect to any basepoint in $\partial X$, a log-stochastic biLipschitz embedding from $\partial X$ into a bounded, pointed ultrametric space. Denote the basepoint in $A$ by $p$, and equip $X$ with basepoint $f_1(p)$.

By Lemma~\ref{lem:BS}, there is a rough biLipschitz embedding $f_2: X \to \Hyp(\partial X)$. Let $(p',h) := f_2(f_1(p))$. Equip $\Hyp(\partial X)$ with basepoint $(p',h)$ and $\partial X$ with basepoint $p'$. Fix a visual metric $d$ on $\partial X$, and let $\{\phi_\omega: (\partial X,d,p') \to (U,d_U,q')\}_{\omega\in\Omega}$ be a log-stochastic biLipschitz embedding into a bounded, pointed ultrametric space $U$. Since rescalings of bounded ultrametric spaces are bounded and ultrametric, we may assume that the scaling factor of the embedding is $s=1$. Thus, $\diam(\partial X) \leq \diam(U)$, and we equip $\Hyp(U)$ with basepoint $(q',h)$.

By Lemma~\ref{lem:Hypstochastic}, $\{\Hyp(\phi_\omega): \Hyp(\partial X) \to \Hyp(U)\}_{\omega\in\Omega}$ is a stochastic rough isometric embedding. By Lemma~\ref{lem:Hyp(U)T}, there is a measurable rough biLipschitz embedding $f_3: \Hyp(U) \to T$ into a pointed $\R$-tree $T$, which maps separable subsets to separable subsets. We equip $T$ with basepoint $f_3(q',h)$. Then $\{f_3 \circ \Hyp(\phi_\omega) \circ f_2 \circ f_1: A \to T\}_{\omega\in\Omega}$ is a stochastic rough biLipschitz embedding. By Lemma~\ref{lem:roughtobiLip}, there is a stochastic biLipschitz embedding from $A$ into a pointed $\R$-tree, and then by Theorem~\ref{thm:stochasticembedtreeLFisoL1}, $\LF(A) \approx \ell^1$.
\end{proof}

The last result of this section gives our main example of spaces stochastic biLipschitzly embedding into $\R$-trees. It follows immediately from Theorem~\ref{thm:Hypstochasticums} and Corollary~\ref{cor:Nagatalogstochastic}, and we omit the proof.

\begin{corollary} \label{cor:Nagatafilling}
Let $A$ be an infinite, uniformly discrete, locally finite, pointed metric space, and let $X$ be a visual, hyperbolic, pointed metric space. If $A$ rough biLipschitzly embeds into $X$ and $\partial X$ is separable and finite Nagata-dimensional, then $A$ stochastic biLipschitzly embeds into a pointed $\R$-tree and $\LF(A) \approx \ell^1$.
\end{corollary}

In the next two sections, we will apply Corollary~\ref{cor:Nagatafilling} to answer a couple of open questions.

\section{The Lipschitz Free Space over $\H^n$} \label{s:LF(Hn)}
Our first application of Corollary~\ref{cor:Nagatafilling} answers \cite[Question~7]{DK} concerning the isomorphism type of the free space over classic hyperbolic $n$-space $\H^n$ (Corollary~\ref{cor:LF(Hn)}).
We recall the definition and basic properties of this space here, following \cite{BH}.

Let $2 \leq n \in \N$. \emph{Hyperbolic $n$-space} $\H^n$ is the unique simply connected, $n$-dimensional, complete Riemannian manifold with constant sectional curvature $-1$ \cite[page~15]{BH}. By virtue of being a complete Riemannian manifold, it is proper and geodesic (\cite[Chapter~I.3: Proposition~3.7, Corollary~3.20]{BH}). The exponential map at every point of $\H^n$ is a diffeomorphism from $\R^n$ (\cite[page~94]{BH}), and thus, by homogeneity\footnote{A metric space $X$ is \emph{homogeneous} if for every $x,y \in X$, there exists an isometry $\Phi: X \to X$ with $\Phi(x) = y$.} of $\H^n$ (\cite[Chapter~I.2: Proposition~2.17]{BH}) and the fact that diffeomorphisms between Riemannian manifolds are biLipschitz on compact subsets, we have, for every $r<\infty$, that there exist $L<\infty$ and a ball $B \sbs \R^n$ (equipped with the Euclidean metric) such that for every $x \in \H^n$, the ball $B_r(x) \sbs \H^n$ (equipped with the hyperbolic metric) is biLipschitz equivalent to $B$ with distortion $L$. Hyperbolic $n$-space is Gromov hyperbolic (\cite[Chapter~III.H: Proposition~1.2]{BH}) and visual (this is most easily seen in the Poincar\'e ball model \cite[Proposition~6.8]{BH}). The Gromov boundary admits a visual metric such that $\partial \H^n = S^{n-1}$ isometrically, where the latter space is the $(n-1)$-sphere $\{x \in \R^n: \|x\|_2 =1\}$ equipped with (a multiple of) its usual geodesic metric (\cite[page~434]{BH}). In particular, $\partial\H^{n-1}$ is doubling, and thus separable and finite Nagata-dimensional. Because of this, we immediately deduce the following lemma from Corollary~\ref{cor:Nagatafilling}.

\begin{lemma} \label{lem:embedintoHn}
Let $A$ be an infinite, uniformly discrete, locally finite pointed metric space. If $A$ rough biLipschitzly embeds into $\H^n$ for some $2 \leq n \in \N$, then $A$ stochastic biLipschitzly embeds into a pointed $\R$-tree and $\LF(A) \approx \ell^1$.
\end{lemma}

\noindent Finally, $\H^n$ has finite Nagata dimension \cite[Theorem~3.6]{LS}. 

It was recently proved in \cite{Russo} (for $n=2,3,4$) that there is a uniformly discrete subset $A \sbs \H^n$ such that
\begin{equation} \label{eq:BLR}
    \LF(\H^n) \approx \LF(\R^n) \oplus \LF(A).
\end{equation}
In Corollary~\ref{cor:LF(Hn)}, we remove the factor of $\LF(A)$ and obtain
\begin{equation} \label{eq:LF(Hn)}
    \LF(\H^n) \approx \LF(\R^n).
\end{equation}
As a first step toward obtaining \eqref{eq:LF(Hn)} for every $n$, we establish isomorphisms of the type \eqref{eq:BLR} in the more general setting where $\H^n$ is replaced by any separable space of finite Nagata dimension and $\R^n$ is replaced by a Carnot group. This is the content of Theorem~\ref{thm:NagataCarnot} (see the paragraph preceding the theorem statement for a discussion on Carnot groups). We will require the following general lemma on finite Nagata-dimensional spaces in the proof of Theorem~\ref{thm:NagataCarnot}.

\begin{lemma} \label{lem:Nagatal1sum}
Let $(X,d)$ be a metric space of finite Nagata dimension. Then there exists $\lambda<\infty$ such that, for every scale $s\in (0,\infty)$ and every maximal $s$-separated subset $A \sbs X$, the free space $\LF(X)$ is isomorphic to a complemented subspace of the $\ell^1$-sum $\left(\oplus^1_{a\in A}\LF(B_{\lambda s}(a))\right) \oplus \LF(A)$.
\end{lemma}

\noindent Before proving the lemma, we need to recall the construction of quotient metrics and some results about their Lipschitz free spaces. Following \cite[$\S$1.4]{Weaver}, if $(X,d)$ is a metric space and $A \sbs X$ a closed subset, the quotient set $X/A$ is defined by $X/A := X/\sim$, where $x\sim y$ if and only if $x,y\in A$ or $x=y$. We denote by $\pi: X \to X/A$ the canonical projection. We equip $X/A$ with the \emph{quotient metric} defined by $\rho(\pi(x),\pi(y)):=\min\{d(x,y),\dist(E,x)+\dist(E,y)\}$ (which is obviously well-defined on equivalence classes). The metric $\rho$ is the largest metric on $X/A$ such that $\pi: X \to X/A$ is 1-Lipschitz. Whenever $X$ is pointed and $A$ contains the basepoint, we equip $X/A$ with basepoint $\pi(A)$. There is a weak$^*$-weak$^*$-continuous, surjective linear isometry from $\Lip_0(X/A)$ (which equals $\LF(X/A)^*$) to the weak$^*$-closed subspace $\Lip_A(X) := \{f\in \Lip_0(X): f(A) = \{0\}\}$ given by $g \mapsto g \circ \pi$. Whenever $A \sbs X$ has finite Nagata dimension (say, dimension $n\in\N$ with constant $\gamma<\infty$), it holds that $\LF(X) \approx \LF(A) \oplus \LF(X/A)$, where the isomorphism constant depends only on $n,\gamma$ (e.g., \cite[Lemma~3.2]{FG}). In particular, if $X$ has finite Nagata dimension, then the isomorphism $\LF(X) \approx \LF(A) \oplus \LF(X/A)$ holds for every closed $A \sbs X$, with isomorphism constant independent of $A$.

% \cite[Proposition~1.30(i)]{Weaver} Leibniz rule.

\begin{proof}[Proof of Lemma~\ref{lem:Nagatal1sum}]
Suppose $X$ has Nagata dimension $n\in\N$. Let $s\in (0,\infty)$ be a scale and $A \sbs X$ a maximal $s$-separated subset. For this paragraph, let $\lambda$ denote an arbitrary constant in $(0,\infty)$. It will be chosen specifically later in the proof (independent of $s$). Since $\LF(X) \approx \LF(A) \oplus \LF(X/A)$ and $\LF(B_{\lambda s}(a)) \approx \LF(B_{\lambda s}(a) \cap A) \oplus \LF(B_{\lambda s}(a)/(B_{\lambda s}(a) \cap A))$ for every $a \in A$ (with isomorphism constants independent of $\lambda,s,a$ -- see preceding discussion), it suffices to show that $\LF(X/A) \cembed \oplus^1_{a\in A}\LF(B_{\lambda s}(a)/(B_{\lambda s}(a) \cap A))$. We will achieve this by constructing bounded, weak$^*$-weak$^*$-continuous linear maps $\Lip_A(X) \overset{R}{\to} \oplus_{a\in A}^\infty\Lip_{B_{\lambda s}(a) \cap A}(B_{\lambda s}(a)) \overset{T}{\to} \Lip_A(X)$ such that $T \circ R = id_{\Lip_A(X)}$, where the middle Banach space in this sequence is an $\ell^\infty$-sum. The map $R$ is simply a sum of restriction maps: $R(f)_a := f\big|_{B_{\lambda s}(a)}$. We will use a special Nagata covering and associated Lipschitz partition of unity to construct the retraction $T$ of $R$.

By \cite[Proposition~4.1(i)(ii)]{LS} there is a constant $\gamma<\infty$ (independent of $s$) such that $X$ admits a Nagata cover $\{B_j\}_{j \in J}$ of dimension $n$, constant $\gamma$, and scale $s$, and with the additional property that for every $x\in X$, there exists an index $j \in J$ such that $B_{s}(x) \sbs B_{j}$. For each $a \in A$, choose $j_a \in J$ such that $B_s(a) \sbs B_{j_a}$. By this containment and maximality of $A$, the collection $\{B_{j_a}\}_{a\in A}$ covers $X$. We repeat the construction used in the proof of Lemma~\ref{lem:Nagatathreshold}. For each $a \in A$, define $\varphi_a: X \to [0,1]$ by $\varphi_a(x) := \max\{0,1-2s^{-1}\dist(B_{j_a},x)\}$. As in the proof of Lemma~\ref{lem:Nagatathreshold}, it holds that, for every $x\in X$,
\begin{itemize}
    \item $\Lip(\varphi_a) \leq 2s^{-1}$ for every $a$ and
    \item $\varphi_a(x) \neq 0$ for at most $n+1$ values of $a$,
\end{itemize}
and additionally, since $\{B_{j_a}\}_{a\in A}$ covers $X$,
\begin{itemize}
    \item there exists $a \in A$ with $\varphi_a(x) = 1$.
\end{itemize}
Furthermore, for any $a \in A$, since $a \in B_{j_a}$ and $\diam(B_{j_a}) \leq \gamma s$, for any $x \in X$ with $\varphi_a(x) \neq 0$, it holds that $d(a,x) \leq \diam(B_{j_a}) + \dist(B_{j_a},x) < \gamma s + \frac{s}{2} = (\gamma+\frac{1}{2})s$. Thus,
\begin{itemize}
    \item $\supp(\varphi_a) = \varphi_a^{-1}((0,1]) \sbs B_{\lambda s}(a)$, where $\lambda := \gamma+\frac{1}{2}$.
\end{itemize}
Together, these items imply that the functions $\{\psi_a: X \to \R\}_{a\in A}$ defined by
\begin{equation*}
    \psi_a := \frac{\varphi_a}{\sum_{b\in A}\varphi_b}
\end{equation*}
are well-defined and satisfy, for every $a \in A$,
\begin{itemize}
    \item $\psi_a(X) \sbs [0,1]$,
    \item $\supp(\psi_a) \sbs B_{\lambda s}(a)$,
    \item $\Lip(\psi_a) \leq Cs^{-1}$, where $C$ depends only on $n$, and
    \item $\sum_{a\in A}\psi_a \equiv 1$.
\end{itemize}
Also, as for $\varphi_a$,
\begin{itemize}
    \item for each $x \in X$, the set $A_x := \{a \in A: \psi_a(x) \neq 0\}$ has cardinality at most $n+1$.
\end{itemize}

Now we define our operator $T$. Given $(f_a)_{a\in A} \in \oplus^\infty_{a \in A}\Lip_{B_{\lambda s}(a) \cap A}(B_{\lambda s}(a))$, define $T((f_a)_a): X \to \R$ by
\begin{equation*}
    T((f_a)_a)(x) := \sum_{a\in A} \psi_a(x)f_a(x),
\end{equation*}
where we interpret the product $\psi_a(x)f_a(x)$ as 0 if $x \notin B_{\lambda s}(a)$. It is clear that $T((f_a)_a)$ vanishes on $A$. We will verify that $T((f_a)_a)$ is Lipschitz and bound the operator norm of $T$. Once this has been established, it is clear that $T$ is a weak$^*$-weak$^*$-continuous bounded linear map and that $T \circ R = id_{\Lip_A(X)}$ since $\sum_{a\in A} \psi_a \equiv 1$.

Note that $\|f_a\|_{L^\infty(B_{\lambda s}(a))} \leq \Lip(f_a)\lambda s$ since $f_a(a) = 0$. By the Lipschitz Leibniz rule, for each $a \in A$, we have
\begin{align*}
    \Lip(\psi_a f_a) &\leq \|\psi_a\|_\infty\Lip(f_a) + \Lip(\psi_a)\|f_a\|_{L^\infty(B_{\lambda s}(a))} \\
    &\leq \Lip(f_a) + Cs^{-1}\Lip(f_a)\lambda s \\
    &= (1+C\lambda)\Lip(f_a).
\end{align*}
Using this, we get, for every $x,y \in X$,
\begin{align*}
    |T((f_a)_a)(x) - T((f_a)_a)(y)| &= \left|\sum_{a\in A}\psi_a(x)f_a(x)-\sum_{a\in A}\psi_a(y)f_a(y)\right| \\
    &\leq \sum_{a\in A_x \cup A_y}|\psi_a(x)f_a(x)-\psi_a(y)f_a(y)| \\
    &\leq 2(n+1)(1+C\lambda)\sup_{a\in A}\Lip(f_a)d(x,y).
\end{align*}
\end{proof}

A \emph{(subRiemannian) Carnot group} is a special type of nilpotent Lie group equipped with a left-invariant subRiemannian metric, which is, in general, not Riemannian. Carnot groups are always complete and doubling. In particular, they are separable with finite Nagata dimension. The abelian examples are exactly the Euclidean spaces $\R^n$ equipped with an inner product. See \cite[page~7306]{AACD} and \cite{LeDonne} for further background. For us, the relevant free-space-theoretic features of a Carnot group $G$ -- due to Albiac, Ansorena, C\'uth, and Doucha -- are that $\LF(G)$ is isomorphic to its countable $\ell^1$-sum $\oplus^1_{n\in\N} \LF(G)$ (\cite[Corollaries~5.5,~5.6]{AACD}) and $\LF(U) \approx \LF(G)$ whenever $U \sbs G$ is nonempty and open (\cite[Corollary~5.6]{AACD}). We will use these features in the proof of the next theorem.

\begin{theorem} \label{thm:NagataCarnot}
Let $(X,d)$ be a finite Nagata-dimensional metric space and $G$ a Carnot group. Then the following statements hold.
\begin{enumerate}
    \item[(1)] If there exists a nonempty open subset of $G$ that biLipschitzly embeds into $X$, then for any nonempty uniformly discrete subset $A \sbs X$, $\LF(A) \oplus \LF(G) \cembed \LF(X)$.
    \item[(2)] If $X$ is separable and if there exist a constant $L<\infty$ and scale $t > 0$ such that $B_t(x)$ biLipschitzly embeds with distortion $L$ into $G$ for every $x \in X$, then there is a uniformly discrete subset $A \sbs X$ such that $\LF(X) \cembed \LF(A) \oplus \LF(G)$.
    \item[(3)] If (1) and (2) hold, and if $\LF(A) \cembed \LF(G)$ for every nonempty uniformly discrete $A \sbs X$, then $\LF(X) \approx \LF(G)$.
\end{enumerate}
\end{theorem}

\begin{proof}
For item (1), assume that a nonempty open subset $U \sbs G$ biLipschitzly embeds into $X$, and let $A \sbs X$ be a $\theta$-separated nonempty subset for some $\theta>0$. By restricting the embedding $f: U \to X$ to an even smaller open subset, we may assume that $\diam(f(U)) \leq \theta/3$. Then it must hold that $\dist(f(U),A\setminus\{a_0\}) \geq \theta/3$ for some $a_0 \in A$. Choose a basepoint $x_0 \in f(U)$, and consider the new $\theta/3$-separated set $A' := (A \setminus\{a_0\}) \cup \{x_0\}$ equipped with basepoint $x_0$. It holds that $f(U) \cap A' = \{x_0\}$ and, for any $x \in f(U)$ and $y \in A'$, $d(x,x_0) + d(x_0,y) \leq 3d(x,y)$. Hence, by \cite[Proposition~5.1]{Kaufmann},
\begin{equation*}
    \LF(f(U) \cup A') \approx \LF(f(U)) \oplus \LF(A').
\end{equation*}
Furthermore, since $A \setminus \{a_0\} \sbs A,A'$ and $|A \setminus (A \setminus \{a_0\})| = 1 = |A' \setminus (A \setminus \{a_0\})|$, it follows that
\begin{equation*}
    \LF(A) \approx \LF(A').
\end{equation*}
Then, since $\LF(f(U)) \approx \LF(U) \approx \LF(G)$, we can combine this with the previous two isomorphisms and obtain
\begin{equation*}
    \LF(f(U) \cup A') \approx \LF(G) \oplus \LF(A).
\end{equation*}
Finally, since $\LF(f(U) \cup A') \cembed \LF(X)$ (because $X$ has finite Nagata dimension), item (1) is proved.

For item (2), assume that $X$ is separable and that there exist a constant $L<\infty$ and scale $t > 0$ such that $B_t(x)$ biLipschitzly embeds with distortion $L$ into $G$ for every $x \in X$. Let $\lambda<\infty$ be the constant appearing in the conclusion of Lemma~\ref{lem:Nagatal1sum}, and set $s := t/\lambda$. By Lemma~\ref{lem:Nagatal1sum}, we have that $\LF(X) \cembed \left(\oplus^1_{a\in A}\LF(B_t(a))\right) \oplus \LF(A)$, where $A \sbs X$ is an $s$-separated subset. Then $A$ must be countable since it is uniformly discrete and $X$ is separable. Since each ball $B_t(a)$ biLipschitzly embeds into $G$ with distortion $L$, and since $G$ has finite Nagata dimension (since it is doubling), we have that $B_t(a)$ is $L$-isomorphic to a $C$-complemented subspace of $\LF(G)$, where $C$ doesn't depend on $a$. Then $\oplus^1_{a\in A}\LF(B_t(a)) \cembed \oplus^1_{n\in\N} \LF(G)$, and this latter space is itself isomorphic to $\LF(G)$. This completes the proof of item (2).

For item (3), assume that (1) and (2) hold and that $\LF(A) \cembed \LF(G)$ for every uniformly discrete $A \sbs X$. Let $A \sbs X$ be the uniformly discrete subset given by item (2). Then since $\LF(G) \approx \oplus_{n\in\N}^1 \LF(G)$, the Pe\l czy\'{n}ski decomposition method implies $\LF(A) \oplus \LF(G) \approx \LF(G)$. Therefore, we have by items (1) and (2) that $\LF(X) \cembed \LF(G) \cembed \LF(X)$, and so another application of the Pe\l czy\'{n}ski decomposition method yields $\LF(X) \approx \LF(G)$.
\end{proof}

Combining the preceding results yields the main result of this section, which proves Theorem~\ref{thm:D}.

\begin{corollary} \label{cor:LF(Hn)}
$\LF(\H^n) \approx \LF(\R^n)$.
\end{corollary}

\begin{proof}
We will show that the hypotheses of all three items of Theorem~\ref{thm:NagataCarnot} are satisfied with $X = \H^n$ and $G = \R^n$. Indeed, the hypotheses of items (1) and (2) are satisfied by virtue of the fact that $\H^n$ is a homogeneous $n$-dimensional Riemannian manifold. It remains to show that $\LF(A) \cembed \LF(\R^n)$ for every uniformly discrete subset $A \sbs \H^n$.

Let $A \sbs \H^n$ be uniformly discrete. Then $A$ is locally finite since $\H^n$ is proper. If $A$ is finite, then $\LF(A) \cembed V$ for every infinite dimensional Banach space $V$, so we may assume that $A$ is infinite. Then by Lemma~\ref{lem:embedintoHn}, we have $\LF(A) \approx \ell^1$. By \cite[Theorem~1.1(i)]{CDW}, $\ell^1 \cembed \LF(\R^n)$, completing the proof.
\end{proof}

\section{Uniformly Lipschitz Affine Actions on $\ell^1$}
\label{s:actions}
As a second application of Corollary~\ref{cor:Nagatafilling}, we obtain, for every hyperbolic group, a proper, uniformly Lipschitz affine action on $\ell^1$. This generalizes a recent result of Dru\c{t}u-Mackay \cite[Theorem~1.6]{DM} (albeit without control on the Lipschitz constant of the action), which asserts the same conclusion under the additional assumption that the hyperbolic group is residually finite\footnote{A group is \emph{residually finite} if the intersection of all its finite index normal subgroups is trivial. It is a well-known open question whether \emph{all} hyperbolic groups are residually finite \cite[$\S$11.25]{DKap}.}. See also the recent work of Vergara \cite{Vergara} where it is proved that every finitely generated hyperbolic group admits a proper, uniformly Lipschitz affine action on a subspace of $L^1$. As explained in $\S$\ref{ss:results}, this line of research is inspired by a conjecture of Shalom (\cite[Open~Problem~14]{Ober}, \cite[Conjecture~35]{Nowak}).

\begin{lemma} \label{lem:fghyperbolic}
Every infinite, finitely generated hyperbolic group stochastic biLipschitzly embeds into a pointed $\R$-tree and has free space isomorphic to $\ell^1$.
\end{lemma}

\begin{proof}
Let $\Gamma$ be an infinite, finitely generated hyperbolic group. Then $\Gamma$ is uniformly discrete, and locally finite, and admits a rough biLipschitz embedding into $\H^n$ for some $2 \leq n \in \N$ by \cite[Embedding Theorem~1.1]{BS}. Then the conclusion follows from Lemma~\ref{lem:embedintoHn}.
\end{proof}

The last result of this article proves Theorem~\ref{thm:E}.

\begin{theorem} \label{thm:l1action}
For every infinite hyperbolic group $\Gamma$ and finite generating set $S \sbs \Gamma$, there exists an isometric affine action $\alpha$ of $\Gamma$ on a Banach space $V$ isomorphic to $\ell^1$ such that the orbit map $\gamma \mapsto \alpha(\gamma)(0)$ is an isometric embedding $\Gamma \hookrightarrow V$ (with respect to the $S$-word metric). In particular, every hyperbolic group admits a proper, uniformly Lipschitz affine action on $\ell^1$.
\end{theorem}

\begin{proof}
Let $\Gamma$ be an infinite hyperbolic group with finite generating set $S \sbs \Gamma$. We consider $\Gamma$ as a pointed metric space with the left-invariant $S$-word metric and basepoint the identity element $1 \in \Gamma$. By Lemma~\ref{lem:fghyperbolic}, $\LF(\Gamma) \approx \ell^1$. We will show that $V=\LF(\Gamma)$ satisfies remaining properties.

There is a natural affine isometric action of $\Gamma$ on $\LF(\Gamma)$, which is essentially due to the fact that $X \mapsto \LF(X)$ is functorial. Indeed, let $W$ be the linear subspace spanned by $\{\delta_\gamma\}_{\gamma\in\Gamma} \sbs \LF(\Gamma) \sbs \Lip_0(\Gamma)^*$. Since $W$ is dense in $\LF(X)$, it suffices to produce the affine isometric action on $W$. We define $\alpha: \Gamma \times W \to W$ by $\alpha(\gamma)\left(\sum_{\gamma'} c_{\gamma'}\delta_{\gamma'}\right) := \sum_{\gamma'} c_{\gamma'}(\delta_{\gamma\gamma'}-\delta_\gamma) + \delta_\gamma$. The map $\alpha(\gamma): W \to W$ is well-defined since $\delta_1 = 0$ and $\{\delta_{\gamma'}\}_{\gamma'\neq 1}$ is linearly independent. It is also clearly affine and satisfies $\alpha(\delta_1) = id_W$ since $\delta_1 = 0$. To check the associativity property, let $\gamma_1,\gamma_2 \in \Gamma$ and $\sum_{\gamma'} c_{\gamma'}\delta_{\gamma'} \in W$. Then we have
\begin{align*}
    \alpha(\gamma_1)\alpha(\gamma_2)\left(\sum_{\gamma'} c_{\gamma'}\delta_{\gamma'}\right) &= \alpha(\gamma_1)\left(\sum_{\gamma'} c_{\gamma'}(\delta_{\gamma_2\gamma'}-\delta_{\gamma_2}) + \delta_{\gamma_2}\right) \\
    &= \left(\sum_{\gamma'} c_{\gamma'}(\delta_{\gamma_1\gamma_2\gamma'}-\delta_{\gamma_1\gamma_2}) + (\delta_{\gamma_1\gamma_2} - \delta_{\gamma_1})\right) + \delta_{\gamma_1} \\
    &= \sum_{\gamma'} c_{\gamma'}(\delta_{\gamma_1\gamma_2\gamma'}-\delta_{\gamma_1\gamma_2}) + \delta_{\gamma_1\gamma_2} \\
    &= \alpha(\gamma_1\gamma_2)\left(\sum_{\gamma'} c_{\gamma'}\delta_{\gamma'}\right).
\end{align*}
For each $\gamma\in\Gamma$, it is easy to see that $\alpha(\gamma) = \pi(\gamma)^* + \delta_\gamma$, where $\pi(\gamma): \Lip_0(\Gamma) \to \Lip_0(\Gamma)$ is the linear isometry defined by $\pi(\gamma)(f)(x) := f(\gamma x) - f(\gamma)$ and $\pi(\gamma)^*: \Lip_0(\Gamma)^* \to \Lip_0(\Gamma)^*$ is its adjoint. Hence, $\alpha(\gamma)$ is an isometry. Finally, observe that the orbit map $\gamma \mapsto \alpha(\gamma)(0) = \delta_\gamma$ is the isometric embedding $\delta: \Gamma \hookrightarrow \LF(\Gamma)$.

The proper and uniformly Lipschitz affine action of $\Gamma$ on $\ell^1$ is given by $\gamma \cdot x := (T \circ \alpha(\gamma) \circ T^{-1})(x)$, where $T: \LF(\Gamma) \to \ell^1$ is any linear isomorphism.
\end{proof}

\end{document}